\documentclass[leqno,12pt]{article}

\usepackage[ansinew]{inputenc}
\usepackage[pdftex]{graphicx}
\usepackage[authoryear]{natbib}
\usepackage{float}
\usepackage{latexsym}
\usepackage{amsmath,amssymb,amsthm,tabu,amsfonts}
\usepackage{enumitem}
\usepackage{makeidx}
\usepackage{pdfpages}
\usepackage{color}
\usepackage{colortbl}
\usepackage[colorlinks, allcolors=blue]{hyperref}
\newcommand{\tabj}{\newline\newline}
\usepackage{eurosym}
\usepackage{multirow}
\usepackage{bm}

\allowdisplaybreaks

\newtheoremstyle{dotless}{}{}{\itshape}{}{\bfseries}{}{ }{}
\theoremstyle{dotless}
\newtheorem{theorem}{Theorem}[section]
\newtheorem{corollary}[theorem]{Corollary}
\newtheorem{assump}[theorem]{Assumption}
\newtheorem{lemma}[theorem]{Lemma}

\newtheoremstyle{definition}{}{}{}{}{\bfseries}{}{ }{}
\theoremstyle{definition}

\newtheorem{remark}[theorem]{Remark}
\newtheorem{algorithm}[theorem]{Algorithm}

\numberwithin{equation}{section}

\newcommand{\R}{\mathbb{R}}

\newcommand{\Z}{\mathbb{Z}}
\newcommand{\N}{\mathbb{N}}
\newcommand{\M}{\mathbb{M}}
\newcommand{\FM}{F_{\mathbb{M}}}

\newcommand{\qT}{q(T)}

\newcommand{\Cov}{\operatorname{Cov}}
\newcommand{\Corr}{\operatorname{Corr}}
\newcommand{\Var}{\operatorname{Var}}
\newcommand{\E}{\mathbb{E}}

\newcommand{\Pb}{\mathbb{P}}
\newcommand{\PbX}{\mathbb{P}_{|\mathcal{X}}}

\newcommand{\hatSdalpha}{\widehat{\mathcal{S}}_{d,\alpha}}

\newcommand{\hatmu}{\widehat{\mu}}

\newcommand{\floor}[1]{\lfloor #1 \rfloor}
\newcommand{\ceil}[1]{\lceil #1 \rceil}
\newcommand{\convd}{\overset{\mathcal{D}}{\Longrightarrow}}
\newcommand{\eqd}{\overset{\mathcal{D}}{=}}
\newcommand{\convp}{\overset{\mathbb{P}}{\Longrightarrow}}

\newcommand{\limd}{\lim_{d\to\infty}}
\newcommand{\limx}{\lim_{x\to\infty}}
\newcommand{\limmd}{\lim_{m,d\to\infty}}

\newcommand{\maxkTm}{\max\limits_{k=1}^{Tm}}

\newcommand{\hatsigma}{\hat{\sigma}}
\newcommand{\hatgamma}{\hat{\gamma}}

\newcommand{\maxhd}{\max\limits_{h=1}^d}
\newcommand{\minhd}{\min\limits_{h=1}^d}
\newcommand{\maxhid}{\max\limits_{h,i=1}^d}
\newcommand{\hatE}{\hat{E}}
\newcommand{\Tmd}{\mathcal{T}_{m,d}}
\newcommand{\hatTmd}{\widehat{\mathcal{T}}_{m,d}}
\newcommand{\hatTmdZ}{\widehat{\mathcal{T}}_{m,d}^{(Z)}}
\newcommand{\tildeTmdZ}{\widetilde{\mathcal{T}}_{m,d}^{(Z)}}
\newcommand{\Wd}{\mathcal{W}_d}
\newcommand{\tilderho}{\tilde{\rho}}
\newcommand{\TmdZ}{\mathcal{T}_{m,d}^{(Z)}}
\newcommand{\hatZ}{\widehat{Z}}
\newcommand{\tildeZ}{\widetilde{Z}}
\newcommand{\tildebVz}{\widetilde{\bm{V}}^{(z)}}
\newcommand{\tildeVz}{\widetilde{V}^{(z)}}
\newcommand{\bVz}{\bm{V}^{(z)}}
\newcommand{\Vz}{V^{(z)}}
\newcommand{\hatbVz}{\widehat{\bm{V}}^{(z)}}
\newcommand{\hatVz}{\widehat{V}^{(z)}}
\newcommand{\dV}{d_{\bm{V}}}
\DeclareMathOperator{\Erfc}{Erfc}
\DeclareMathOperator{\Erf}{Erf}
\newcommand\Tstrut{\rule{0pt}{3ex}}

\definecolor{darkblue}{rgb}{.1, 0.1,.8}
\definecolor{darkgreen}{rgb}{0,0.8,0.2}
\definecolor{darkred}{rgb}{.8, .1,.1}

\newcommand{\vep}{\varepsilon}

\usepackage{mathtools}
\mathtoolsset{showonlyrefs}

\def\spacingset#1{\renewcommand{\baselinestretch}%
{#1}\small\normalsize} \spacingset{1}

\renewcommand{\baselinestretch}{1.4}

\begin{document}
\title{Sequential change point detection in high dimensional time series}
{

\author{
{\small Josua G\"osmann, Christina Stoehr, Johannes Heiny and Holger Dette}\\
{\small Ruhr-Universit\"at Bochum}\\
{\small Fakult\"at f\"ur Mathematik}\\
{\small 44780 Bochum, Germany}\\
}

\date{}
\maketitle

\begin{small}
\begin{abstract}
Change point detection in high dimensional data has found considerable interest in recent years. 
{Most of the literature either designs methodology for a retrospective analysis, where the whole sample is already available when the statistical inference begins,
or considers online detection schemes controlling the average time until a false alarm.}
This paper takes a different point of view and develops monitoring schemes for the online scenario, where high dimensional data arrives successively and the goal is to detect changes as fast as possible controlling at the same time the probability of a type I error of a false alarm.
We develop a sequential procedure capable of detecting changes in the mean vector of a successively observed high dimensional time series with spatial and temporal dependence.
The statistical properties of the method are analyzed in the case where both, the sample size and dimension tend to infinity.
In this scenario, it is shown that the new monitoring scheme has asymptotic level alpha under the null hypothesis of no change and is consistent under the alternative of a change in at least one component of the high dimensional mean vector.
The approach is based on a new type of monitoring scheme for one-dimensional data which turns out to be often more powerful than the usually used CUSUM and Page-CUSUM methods, and the component-wise statistics are aggregated by the maximum statistic.
{For the analysis of the asymptotic properties of our monitoring scheme we prove that the range of a Brownian motion on a given interval is in the domain of attraction of the Gumbel distribution, which is a result of independent interest in extreme value theory.}
The finite sample properties of the new methodology are illustrated by means of a simulation study and in the analysis of a data example.
\end{abstract}
\end{small}

\medskip

\smallskip

\noindent

\section{Introduction}\label{sec:intro}
As digital transformation processes have accelerated during the last decades, new technologies like smartphones or car sensors are able to gather large amounts of data.
Due to this development companies, states, research institutes etc. face the problem to manage, monitor and examine huge data sets, which regularly exceed the means of traditional tools.
Thus the demand for so-called big data technology is steadily growing and thereby the requirement for theoretical foundation has put a lot of attention at the topic of high dimensional statistics in recent years.

Especially, the topic of change point analysis or detection of structural breaks has regained attraction and numerous authors have started to embed commonly used multivariate methods into a high dimensional framework or even develop new methodology from scratch.
Among many others, high dimensional change point problems have been considered by \cite{Cho2015}, \cite{Wang2018}, who develop methodology to identify multiple change points by a (wild) binary segmentation algorithm under sparsity assumptions.
\cite{Jirak2015} and \cite{Dette2018} aggregate component-wise CUSUM-statistics by the maximum functional to detect structural breaks in a sequence of means of a high dimensional time series.
\cite{LevyLeduc2009} analyze internet traffic data, by applying a compenent-wise CUSUM-test to dimension-reduced censored data.

\cite{Enikeeva2019} employ the Euclidean norm of the CUSUM-process to obtain a linear and a scan statistic of $\chi^2$-type, that is minimax-optimal under the regime of independent Gaussian observations.
Change point problems in high dimensional covariance matrices are studied by \cite{Wang2017}, \cite{Avanesov2018} and \cite{Dette2018a} using (wild) binary segmentation, multiscale methods and U-statistics, respectively.
U-statistics are also used by \cite{Wang2019} and \cite{Wang2020} to develop testing and estimation methodology for a structural break in the mean.

All listed references on high dimensional change point problems have in common that the proposed methods are designed for a \textit{retrospective or offline analysis}, where the whole sample is already available when the statistical inference is commenced.
In contrast to this, sequential change point detection deals with methods that are applicable for monitoring data in a so-called \textit{online scenario}.
In such a setup, data arrives steadily and methods are constructed to detect changes as fast as possible, while the problem is reevaluated with each new data point.  
Starting with the seminal paper of \cite{wald1945} on the sequential probability ratio test, an enormous amount of literature has been published discussing the problem from different perspectives.
Several concepts have been proposed to model the situation of online monitoring and we discuss the different paradigms addressing the sequential change point problem in more detail at the end of this introduction.

To the authors best knowledge, a common feature of most of the literature on sequential change point detection consists in the fact that it does not consider the time and spatially dependent high dimensional scenario, where the dimension of the data increases with the sample size.
The purpose of the present paper is to address this problem in the context of detecting changes in the mean.
For this purpose, we follow consider the paradigm of \cite{Chu1996}, which provides a model for online change point detection, such that the probability of a type I error (false alarm) can be controlled (asymptotically).
We develop a sequential algorithm in the high dimensional regime aggregating component-wise sequential detection schemes by the maximum statistic.
For the individual components we use a novel monitoring procedure, which screens for all possible positions of the change point and takes into account that the change does not necessarily occur in the first observations after the initial sample - see Section \ref{sec:testProblem} for more details. 
A nice feature of this approach consists in the fact that the limiting distribution of the statistic used to monitor each component (after appropriate standardization) is given by the {\it range of the Brownian motion}, that is 
$\M = \max_{0 \leq t \leq q} W(t) - \min_{0 \leq t \leq q} W(t)$, where $W$ is a Brownian motion on the interval $[0,q]$ and $0<q <1$ is a known constant.
The distribution of the random variable $\M$  appears as the weak  limit of the range of cumulative sums of i.i.d.~random variables with variance $1$ [see \cite{Feller1951}], and we will show that it belongs to the domain of attraction of the Gumbel distribution.
This result is of independent interest in extreme value theory and allows us to aggregate component-wise statistics by the maximum.

As a consequence we can develop a sequential monitoring scheme in the high dimensional regime using the quantiles of the Gumbel distribution.
For this purpose, we combine Gaussian approximations tools for high dimensional statistics [see \cite{Chernozhukov2013,Zhang2018b}] and show that the statistic can be approximated by a counterpart computed from Gaussian observations with the same long-run correlation structure as the observed time series.
Using Gaussian comparison and anti-concentration inequalities we show that this statistic is sufficiently close to the maximum of ranges of dependent Brownian motions.
Finally, we use Poisson approximation via the Chen-Stein method to eliminate the independence condition, such that the new result for the independent case can be applied.
As the rates of most convergence results in extreme value theory are known to be rather slow, we also propose a simple bootstrap procedure, which improves the performance of the sequential monitoring scheme for sample sample sizes.
To our best knowledge, our paper provides the first rigorous analysis of a sequential change point detector in the high dimensional regime for the model introduced by \cite{Chu1996}.

The remaining part of this paper is organized as follows.
We conclude this introduction with a discussion and comparison of the different philosophies in sequential change point detection.
In Section \ref{sec:testProblem}, we introduce the specific testing problem under consideration and present the new monitoring procedure for structural breaks in the sequence of means from a high dimensional time series.
Section \ref{sec3} is devoted to our main results and to the analysis of the asymptotic properties of the new procedure. 
In particular, we prove that the maximum of the individual test statistics converges weakly (with increasing dimension and initial sample size) to a Gumbel distribution. These results are used to show that the monitoring scheme has asymptotic level $\alpha$ and is consistent.
In Section \ref{sec:finitesamples}, we investigate the finite sample properties of the new procedures by means of a simulation study and illustrate potential applications in a data example.
Finally, all proofs are deferred to Appendix \ref{sec:technicalDetails} in the online supplement.
\bigskip

\noindent
{\bf Related literature - two paradigms in sequential change point detection} 
\label{sec11}
\medskip

\noindent
{
In the remaining part of this section, we briefly discuss to ways of modelling the problem of sequential change point detection, which have developed rather independently in the last $25$ years. The list of references cited below is by no means complete, because both modelling approaches have been used intensively in the literature.
Rather it represents a selective choice of the authors with the goal to put the approach proposed in this paper in the appropriate context. 
Sequential change point detection in the high dimensional scenario, where the dimension increases with the sample size, has barely been studied in the liteature, where some of the relevant exceptions are cited below.
}

The different models in the literature are used to address different criteria for quantifying the propensity of a detection scheme to cause false alarms.
The traditional \textit{Statistical Process Control} (SPC)-approach has its focus on a fast change point detection, where the average time from monitoring start to a false alarms is controlled.

On the other hand, if the costs of a false alarm are heavy, for instance if a large portfolio has to be restructured in portfolio management, it might be more reasonable to control the probability of a false alarm.
While from an application point of view the criteria are strongly related, different mathematical models have to be used to analyze the statistical properties.
The SPC-approach uses models from classical sequential analysis (often independent observations) and some of the relevant literature is given in part (B) of this section.
As an alternative \cite{Chu1996} introduced a model to control asymptotically the probability of a false alarm, which requires an initial stable sample. 
The related literature for this model is described in part (A) of this section. 
The purpose of the present paper is to develop a monitoring scheme in this model, where the dimension is allowed to grow with the sample size (at a polynomial order).

\noindent The SPC- and the approach proposed in \cite{Chu1996} have been developed rather independently during the last two decades and usually a manuscript either focuses on the one or on the other.
One reason for this development is their different focus in the sequential change point problem and probably other ones are mathematical reasons.
The very recent work of \cite{Yu2020} provides a brief theoretical comparison of both methods in the univariate setting with independent sub-Gaussian observations and also indicates the mathematical difficulties of such a comparison. 
The authors argue that (for their specific model) controlling the type I error over a fixed monitoring period only results in minor extra costs for the expected detection delay and 
{\it "...suggest that upper bounding the overall type I error might be better."}
However, their model does not include a stable initial set, such that no asymptotic investigations can be carried out.

\noindent In general, methods developed in the SPC-paradigm often base upon the premise, that certain statistical parameters, for instance the mean before a change, are exactly known from a long-lasting stable period (Phase I).
The procedure for the actual monitoring (Phase II) is then calibrated using these \textit{exact parameters}; see \cite{Chen2020} for a recent example in the context of high dimensional means.
While such a proceeding might be reasonable in some applications like manufacturing control, in other fields like econometrics, this is unrealistic as a large stable set might be rarely available.
In this context, the approach of \cite{Chu1996} might be advantageous, as the size of the stable set is treated as a parameter (sample size) and quantities derived from this set are handled as estimators.
The problem of separate Phase I and Phase II has also been discussed in the SPC-literature; see \cite{Hawkins2003} or \cite{Ross2014} among others, who discuss SPC-type procedures updating the  initial estimates.

\medskip

\noindent
{\bf (A) Controlling the type I error}
\cite{Chu1996} propose a sequential paradigm, such that the type I error can be controlled (asymptotically).
It is based on the premise of an initial stable data set, that has to be available before monitoring commences.
Their main idea is to invoke invariance principles, such that the probability of a false alarm can be controlled over a possibly infinite monitoring time - asymptotically as the size of the initial set grows.
This modelling approach also allows for the estimation of model parameters, such as the (long-run) variance or coefficients in regression models if covariates are present, and therefore may have advantages for applications, where temporal dependence is present.
The literature distinguishes between the \textit{open-end} and \textit{closed-end scenario}.
A \textit{closed-end scenario} is associated with a fixed endpoint, where monitoring has to be eventually stopped even if no change was detected before.
An \textit{open-end scenario} does not postulate an endpoint meaning that monitoring can (theoretically) continue forever if no change is detected.

Since its introduction, this paradigm has found considerable attention in the literature on change point detection.
For example, \cite{Horvath2004}, \cite{Huskova2005} and \cite{Aue2006} consider changes in the parameters of linear models with statistics based on residuals.
For independent identically distributed (i.i.d.) data, \cite{Kirch2008} and \cite{Huskova2012} propose several bootstrap procedures for sequential change point detection in the mean and in the parameters of linear regression models.
A MOSUM-approach, which employs a moving monitoring window in linear models is introduced by \cite{Chen2010}, while \cite{Ciuperca2013} proposes a generalization of the \textit{sequential} CUSUM statistic to non-linear models.
\cite{Fremdt2014} uses the so-called Page-CUSUM, which scans for changes through the already available monitoring data and is more efficient to detect later changes than the classical sequential CUSUM scheme.
\cite{Hoga2017} proposes an $\ell_{1}$-norm to detect structural breaks in the mean and variance of a multivariate time series and \cite{Dette2019} develop an amplified scanning method combined with self-normalization.
\cite{Otto2019} define a Backward CUSUM statistic based on recursive residuals in a linear model.
Unifying frameworks are provided in \cite{Kirch2015} and \cite{Kirch2018} and a theory based on U-statistic is established in \cite{Kirch2019}.
We also refer to the recent review of sequential procedures in Section 1 of \cite{Anatolyev2018}.
\medskip

\noindent
{\bf (B) Statistical Process Control}
The traditional approach, partially known as \textit{Statistical Process Control} (SPC), can be traced back to the seminal papers of \cite{Page1954,Page1955}.
So-called \textit{Control Charts} have been the first methods developed in this area and their typical field of application was quality control for manufacturing processes.
In this spirit, these methods are made to guarantee a quick change point detection, for which regular false alarms are necessarily accepted.
Thus, procedures developed in the SPC-context are commonly analyzed by two quantities.
The average run length (ARL) measures the expected time until a false alarm is raised under the null hypothesis of no change.
Its counterpart is the expected detection delay (EDD) necessary to raise an alarm under the alternative.
It can be regarded as an analogue to classical hypothesis testing theory, that only one of these key quantities can be optimized while the other is only kept bounded.
Since \cite{Page1954} numerous authors have followed this approach, for instance, \cite{Hinkley1971}, \cite{Moustakides1986} or \cite{Nikiforov1987}, and we delegate the reader to review papers by \cite{Woodall1999} or \cite{Lai2001} and the more recent monograph of \cite{Tartakovsky2014}, where the state of the art is discussed.
In a discussion paper \cite{Mei2008} demonstrates that the average run length (ARL) might not be an appropriate criterion as a measure of false detection if the observations are dependent.
The author and the discussants propose several alternatives.

Although the (basic) SPC-approach is relatively old, many authors have been constantly working on this problem from various perspectives.
In particular, there also exist several papers developing online monitoring procedures for new situations arising in the information age. 
For instance \cite{Tartakovsky2006}, \cite{Mei2010}, \cite{Zou2015} investigate multi-sensor change point problems and \cite{Chu2018} propose an algorithm based on nearest neighbour information for non-euclidean data. 
Most of this work investigates multivariate data, where the dimension is assumed to be fixed, and imposes rather strong assumptions to investigate the ARL and EDD (such that of independent identically distributed observations).\\
On the other hand, there also exist a few references considering the high dimensional case, where the dimension is large compared to the sample size.
\cite{Xie2013} develop a mixture procedure to monitor parallel streams of data for a change point without assuming a spatial structure.
\cite{Soh2017} combine a filtered derivative approach with convex optimization to develop a scalable and statistically efficient monitoring scheme for high dimensional sparse signals, while \cite{Chen2020} monitor for mean changes across coordinates by a sophisticated aggregation technique.
These three contributions assume independent normal distributed observations without spatial dependence.

While these assumptions are helpful to obtain some basic understanding of the properties of a procedure, they might be too restrictive in applications, in particular in economics.
Our approach differs substantially from these high dimensional frameworks.
We use the modelling approach by \cite{Chu1996} to (asymptotically) control the probability of a type I error in the high dimensional regime.
Moreover, we neither assume Gaussian nor independent observations to analyze the theoretical properties of the proposed monitoring procedure. 


\section{Sequential monitoring of high dimensional time series}\label{sec:testProblem}
Let $\{\bm{X}_t\}_{t\in\mathbb{Z}}$ denote a time series of random vectors in $\R^d$ with mean vectors
$$
\bm{\mu}_t
:= (\mu_{t,1},\dots,\mu_{t,d})^\top = \E[\bm{X}_t]
:= \E[(X_{t,1},\dots,X_{t,d})^\top]
 ~.$$
We take the sequential point of view and are interested in monitoring for changes in the vectors $\bm{\mu}_1, \bm{\mu}_2, \ldots $~.
Following \cite{Chu1996} we assume that a historic or initial data set, say $\bm{X}_1,\ldots,\bm{X}_m$, is available, which is known to be mean stable.
Starting with observation $\bm{X}_{m+1}$ we will sequentially test for a change in the mean vector in the monitoring period.
The corresponding testing problem is therefore given by the hypotheses
\begin{align}\label{def:hyposGeneral}
\begin{split}
&H_0: \bm{\mu}_1=\dots=\bm{\mu}_m=\bm{\mu}_{m+1} = \bm{\mu}_{m+2}=\cdots\\
\text{versus}~~ &H_1: \exists k^* \in \N \text{, s.t. }
\bm{\mu}_1=\dots=\bm{\mu}_m=\dots=\bm{\mu}_{m+k^*-1} \neq \bm{\mu}_{m+k^*} = \cdots~.
\end{split}
\end{align}
In the present paper, we consider a closed-end scenario where the procedure stops after $m+Tm$ observations even if no change has been detected [see \cite{Aue2012}, \cite{Wied2013} among many others].}
The factor $T$ determines the length of the monitoring period compared to the size of the initial training set $m$ and so the hypotheses in \eqref{def:hyposGeneral} read as follows
\begin{align}\label{def:hypos}
\begin{split}
&H_0: \bm{\mu}_1=\dots=\bm{\mu}_m=\bm{\mu}_{m+1} = \bm{\mu}_{m+2}=\cdots=\bm{\mu}_{m+Tm}\\
\text{versus}~~   &H_1: \exists k^* \in \{1,\dots,Tm\}\text{, s.t. }\bm{\mu}_1=\dots=\bm{\mu}_{m+k^*-1} \neq \bm{\mu}_{m+k^*}= \cdots=\bm{\mu}_{m+Tm}~.
\end{split}
\end{align}
In the following, we will develop a sequential detection scheme which is capable to distinguish between the hypotheses given in \eqref{def:hypos} in a high dimensional setting where the dimension $d$ of the mean vector is increasing with the initial sample size $m$.
To be precise, we denote by 
\begin{align}
\hatmu_i^j(h) = \dfrac{1}{j-i+1}\sum_{t=i}^j X_{t,h}
\end{align}
the estimator of the mean in component $h \in \{1, \ldots , d\} $ from the sample $X_{i,h},\ldots,X_{j,h}$.
Following \cite{Goesmann2019}, we consider the statistic
\begin{align} 
\label{hol31}
\hatE_{m,h}(k) = \max_{j=0}^{k-1} \dfrac{k-j}{\sqrt{m}\hatsigma_h} \Big|\hatmu_{m+j+1}^{m+k}(h) - \hatmu_1^{m+j}(h) \Big|
\end{align}
at time point $m+k$ in a single component $h$, where $\hatsigma_h^2$ denotes an appropriate estimator of the unknown long-run variance
$$
\sigma_h^2 = \sum_{t \in \Z} \Cov(X_{0,h},X_{t,h})
$$
in the $h$th component (explicit conditions for the existence of the long-run variance are given in Section \ref{sec3}).
Note that $\hatE_{m,h}(k)$ is a weighted CUSUM statistic to detect a change point in the sequence of means corresponding to the data $X_{m+1,h},\ldots,X_{m+k,h}$.
A structural break in the sequence of means $\mu_{m+1,h},\mu_{m+2,h}, \ldots $ is detected as soon as the sequence 
$$ w(1/m) \hatE_{m,h}(1), w(2/m) \hatE_{m,h}(2),\ldots 
$$
exceeds a given threshold, that is $w(k/m) \hatE_{m,h}(k) > c_{\alpha}^{(h)}$, where $w$ is a suitable weight function and the critical value $c_{\alpha}^{(h)}$ is chosen based on the desired test level $\alpha$.
Following \cite{Aue2004}, \cite{Wied2013} and \cite{Fremdt2014}, we will work with the commonly used weight function
$w(t) = 1/(1 + t),$
throughout this paper.

\begin{remark} ~~
\begin{itemize} 
\item[(1)]
Note that most of the literature investigates sequential detectors based on the differences
\begin{align}
\label{schemeCUSUM}
\Big|\hatmu_{m+1}^{m+k}(h) - \hatmu_1^{m}(h) \Big|
\end{align} 
and the corresponding detection schemes are usually called (ordinary) CUSUM tests [see 
\cite{Chu1996}, \cite{Horvath2004}, \cite{Aue2006}].
Another part of the literature focuses on detectors based on the differences
\begin{align}
\label{schemePageCUSUM}
\Big|\hatmu_{m+j+1}^{m+k}(h) - \hatmu_1^{m}(h) \Big|\;\;\text{for}\;\;j=0,\dots,k-1
\end{align}
and the corresponding detection schemes are usually called Page-CUSUM tests [see \cite{Fremdt2014,Fremdt2015}, \cite{Kirch2018}].
The use of the differences $\big|\hatmu_{m+j+1}^{m+k}(h) - \hatmu_1^{m+j}(h) \big|$ 
is motivated by the likelihood principle [see \cite{Dette2019}].
Compared to the differences in \eqref{schemeCUSUM} it avoids the problem that the estimator $\hatmu_{m+1}^{m+k}(h)$ may be corrupted by observations before the change point, which could lead to a loss of power. 
Compared to the differences in \eqref{schemePageCUSUM} the use of $\hatmu_1^{m+j}(h)$ instead of $\hatmu_1^{m}(h)$ may avoid a loss in power in cases of a small initial sample and a rather late change point.
The advantages of detection schemes based on the differences $\big|\hatmu_{m+j+1}^{m+k}(h) - \hatmu_1^{m+j}(h) \big|$ against ordinary sequential CUSUM and the Page-CUSUM procedures have been recently demonstrated by \cite{Goesmann2019}.
\item[(2)]
Several authors consider the more general class of weight functions
$
w_\gamma(t) = (t+1)^{-1}\big(\frac{t}{t+1}\big)^{-\gamma}
$
for $\gamma \in [0,1/2)$ [see for instance \cite{Horvath2004}, \cite{Aue2006} or \cite{Kirch2018}].
The weight function 
$w(t) = 1/(1 + t)$
is obtained for $\gamma=0$ and was proven to be preferable to $\gamma>0$ in many situations except for changes that occur almost immediately [see \cite{Kirch2018}].
It is most likely, that the theoretical results of this paper remain correct in the case $\gamma>0$.
\end{itemize} 
\end{remark}

\noindent
In order to control the probability of erroneously deciding for a structural break in the component $h$ during the monitoring period, one has to determine the probability
$$\Pb \Big( \max_{k=1}^{Tm} w(k/m)\hatE_{m,h}(k) > c_{\alpha}^{(h)} \Big)~.$$
For fixed $h \in \{1, \ldots , d\}$ we can employ a result of \cite{Goesmann2019} who showed that (under appropriate assumptions), as $m\to \infty$, 
\begin{align}\label{conv:1dim}
\max_{k=1}^{Tm} w(k/m)\hatE_{m,h}(k)
\convd 
\M = \max_{0 \leq t \leq \qT} W(t) - \min_{0 \leq t \leq \qT} W(t)\,, 
\end{align}
where the symbol $\convd$ denotes weak convergence, $\qT=T/(T+1)$ and $W$ is a standard one-dimensional Brownian motion.
Note that $T$ is the parameter controlling the length of the monitoring period [see the hypotheses in \eqref{def:hypos}].
$\M$ is known in the probability literature as the range of the Brownian motion on the interval $[0,\qT]$ and its distribution appears as the weak limit of the range of cumulative sums of i.i.d.~random variables with variance $1$ [see \cite{Feller1951}].

For a detection of a change point in the complete mean vector we propose to aggregate the statistics for the different spatial dimensions $h=1,\ldots,d$.
More precisely, we consider the maximum of the different components, that is $\max_{h=1}^{d} w(k/m)\hatE_{m,h}(k)$,
and reject the (closed-end) null hypothesis $H_0: \bm{\mu}_1=\dots=\bm{\mu}_{m+Tm}$
of no structural break in the high dimensional means $\bm{\mu}_1,\ldots,\bm{\mu}_{m+Tm}$ if this quantity exceeds a given threshold, that is
\begin{equation}\label{hol1}
\hatTmd:=\max_{k=1}^{Tm} \maxhd w(k/m) \hatE_m(k) > c_{d,\alpha}~.
\end{equation} 
Here the critical value $c_{d,\alpha}$ is chosen appropriately such that (asymptotically) the probability of erroneously deciding for a change point is controlled.
In the following section, we investigate the weak convergence of the statistic 
$\hatTmd$.
These results will be used to define critical values $c_{d,\alpha}$ in \eqref{hol1} (one by asymptotic theory and one by bootstrap), such that the monitoring procedure is consistent and at the same time controls the probability of the type I error, that is 
\begin{align}\label{size}
\limsup_{m,d\to \infty}\, \Pb_{H_0}\Big( \hatTmd > c_{d,\alpha} \Big) & \leq \alpha \quad \text{ and } \quad
\lim_{m,d\to \infty} \Pb_{H_1}\Big( \hatTmd > c_{d,\alpha} \Big)
= 1\,.
\end{align}

\medskip 

\begin{remark} \label{rob}
{\rm In this paper, we consider mean based detectors corresponding to ordinary least squares estimation.
We expect that similar results can be obtained using robust estimates such as $M$-estimates considered by \cite{Chochola2013}.  For example, the median has been 
discussed as a special case in \cite{Dette2019} and \cite{Goesmann2019}. 
}
\end{remark} 

\section{Main results}\label{sec3}
In this section, we derive the asymptotic properties of the proposed detector defined in \eqref{hol1} in the high dimensional setting where sample size and dimension tend to infinity and we allow for temporal as well as spatial dependencies in the data. 
In particular, we establish as a consequence of Theorem \ref{thm:mainGumbel} below - in case of constant mean vectors - the weak convergence
\begin{align}\label{def:hatTmdConv}
a_d\big(\hatTmd -b_d \big) \convd G\quad\mbox{ as }m,d\rightarrow\infty\,,
\end{align}
where $a_d, b_d$ are suitable sequences and $G$ is a standard Gumbel random variable with c.d.f. $F_G(x) = \exp(-\exp(-x))$, $x\in \R$.
As inevitable in high dimensional time series analysis, we require assumptions on the relation between the (initial) sample size and the dimension as well as assumptions on the dependence structure to control the dependence between components at different time points uniformly.

Throughout this paper, we assume that the observations are drawn from the array $\{ X_{t,h} \}_{t \in \Z, h \in \N}$, for which we suppose the location model 
\begin{align}\label{eq:XmuPlusE}
X_{t,h}
= \mu_{t,h} + e_{t,h}~, ~~t \in \Z,~ h \in \N~,
\end{align}
where $\mu_{t,h}=\E[X_{t,h}]$ is the expectation of the $h$ component and the centered array $\{e_{t,h}\}_{t \in \Z, h \in \N}$ is given as a physical system [see e.g. \cite{Wu2005}], that is
\begin{align}\label{eq:physicalSystem}
e_{t,h} = g_h(\varepsilon_t,\varepsilon_{t-1},\dots)~,~~t \in \Z,~ h \in \N~.
\end{align}
The underlying sequence of innovations $\{\varepsilon_t\}_{t \in \Z}$ consists of i.i.d.~random variables with values in some arbitrary measure space $\mathbb{S}$ and the functions $g_h: \mathbb{S}^{\N}: \to \R$ are assumed to be measurable.
Note that, by the definition above, the random variables $\{e_{t,h}\}_{t\in \Z, h \in \N}$ are (strictly) stationary with respect to the time index $t$, such that for any fixed dimension $d$ the multivariate time series $\big\{\bm{e}_t = \big(e_{t,1}, e_{t,2}, e_{t,3},\dots, e_{t,d}\big)^\top\big\}_{t \in \Z}$
is stationary.
The data generating model defined by formula \eqref{eq:physicalSystem} has received a lot of attention in recent years [see for example \cite{Wu2011}, \cite{Liu2013}, \cite{ElMachkouri2013}, \cite{Berkes2014} among many others].
It covers the major part of prevalent time series models like autoregressive or moving average processes.
Furthermore, it also allows for a natural measurement of temporal dependence which is constructed as follows.
Let $\varepsilon_0'$ be an independent copy of $\varepsilon_0$ and define
\begin{align*}
X_{t,h}'
= \mu_{t,h} + g_h(\varepsilon_t,\varepsilon_{t-1},\dots\varepsilon_{1},\varepsilon_{0}',\varepsilon_{-1},\dots)
\end{align*}
as a counterpart of $X_{t,h}$ where $\varepsilon_0$ is replaced by $\varepsilon'_0$.
If $p \geq 1$ we denote by $\Vert X \Vert_p=\E[|X|^p]^{1/p}$ the ordinary $L_p$-norm of a real-valued random variable $X$ (assuming its existence).
If $\| e_{t,h}\|_p < \infty$, the coefficients
\begin{align*}
\vartheta_{t,h,p}
:= \big\Vert X_{t,h} - X_{t,h}' \big\Vert_p
\end{align*}
measure the influence of innovation $\varepsilon_0$ on $X_{t,h}$ and thereby quantify the (temporal) dependence within the system $\{e_{t,h}\}_{t\in \Z, h \in \N}$ defined by \eqref{eq:physicalSystem}.
If $\| e_{t,h}\|_p < \infty$ for some $p \geq 2$ we define the covariances of cross-components by $\phi_{t,h_1,h_2} := \Cov(X_{0,h_1},X_{t,h_2}), \phi_{t,h}:= \phi_{t,h,h}$
and the long-run covariances and variances by
\begin{align}\label{eq:longruncovariance}
\gamma_{h_1,h_2}
:= \sum_{t \in \Z} \phi_{t,h_1,h_2}
\;\;\;
\text{and}
\;\;\;
\sigma_{h}^2
:= \gamma_{h,h}\,,
\end{align}
respectively.
If $\sigma_{h_1},\sigma_{h_2}>0$, let additionally
\begin{align}\label{eq:longruncorr}
\rho_{h_1,h_2}
:= \dfrac{\gamma_{h_1,h_2}}{\sigma_{h_1}\sigma_{h_2}}
\end{align}
denote the long-run correlations.
It will be crucial for the asymptotic considerations to control 
the coefficients $\vartheta_{t,h,p}$ and the correlations $\rho_{h_1,h_2}$ for increasing time $t$ and spatial distance $|h_1-h_2|$, respectively.
This will be formulated in Assumptions \ref{assump:temporalDependence} and \ref{assump:spatialDependence} below.
Before we state these precisely, we begin with two assumptions on the relation between sample size and dimension and on the tail behavior of the errors in model \eqref{eq:XmuPlusE}.

\begin{assump}[Assumption on the dimension]\label{assump:model}
{\rm
There exist constants $D>0$ and $C_D>0$~, such that
\begin{enumerate}[label={\normalfont (D1)},ref=D1, leftmargin=1.2cm]
\item $ m^{1/C_D} \le d \le C_D \, m^D$.
\label{assump:D1}
\end{enumerate}
}
\end{assump}

\begin{assump}[Structural assumptions]\label{assump:RV}
{\rm 
The random variables $e_{t,h}$ in model \eqref{eq:XmuPlusE} have bounded exponential moments, that is: there exists a positive sequence $B_m$, such that
\begin{enumerate}[label={\normalfont (S1)},ref=S1, leftmargin=1.1cm]
\item $\maxhd \E \big[\exp\big( | e_{1,h} |/B_m\big) \big] \leq C_e\,,$
\label{assump:S1}
\end{enumerate}
where $B_m \leq m^B$ with constants $B<3/8$ and $C_e>1$. 
}
\end{assump}
\begin{assump}[Temporal dependence]\label{assump:temporalDependence}
{\rm
There exist constants $p>2D+4$, $\beta \in [0,1)$, $C_\vartheta>0$ such that for all $t \in \N_0$
\begin{enumerate}[label={\normalfont (TD1)},ref=TD1, leftmargin=1.5cm]
\item $\sup\limits_{h\in \N} \vartheta_{t,h,p} \leq C_\vartheta \beta^t$~.
\label{assump:TD1}
\end{enumerate}
Further, assume that for a positive constant $c_\sigma$ the long-run variances defined in \eqref{eq:longruncovariance} are uniformly bounded from below, that is
\begin{enumerate}[label={\normalfont (TD2)},ref=TD2, leftmargin=1.5cm]
\item $c_\sigma \leq \inf\limits_{h\in \N} \sigma_h~.$
\label{assump:TD2}
\end{enumerate}}
\end{assump}

\begin{assump}[Spatial dependence]\label{assump:spatialDependence}
{\rm
There exist a sequence $r_m$ converging to zero, and a constant $\rho_+ \in [0,1)$, such that the long-run correlations defined in \eqref{eq:longruncorr} fulfill
\begin{enumerate}[label={\normalfont (SD1)},ref=SD1, leftmargin=1.5cm]
\item $|\rho_{i,j}| \leq \big(\log|i-j|\big)^{-2} r_{|i-j|}$ whenever $|i-j|\geq 2$\,, 
\label{assump:SD1}
\end{enumerate}
\begin{enumerate}[label={\normalfont (SD2)},ref=SD2, leftmargin=1.5cm]
\item $\sup\limits_{i,j:\,|i-j|\geq 1}|\rho_{i,j}|\leq\rho_+<1\,.$
\label{assump:SD2}
\end{enumerate}
}
\end{assump}

\noindent
Let us briefly discuss the assumptions above.
Assumptions of the type \eqref{assump:D1} are quite common in high dimensional change point problems.
For example, \cite{Jirak2015}, \cite{Wang2018} and \cite{Dette2018} also assume a polynomial growth of the dimension with the sample size.
Conditions like Assumption \ref{assump:RV} and \ref{assump:temporalDependence} are both indispensable ingredients for Gaussian approximation results in high dimensional statistics and will be used in the proofs of our main results.
Here Assumption \ref{assump:RV} controls the tail behavior of the observations [see also \cite{Chernozhukov2013,Chernozhukov2019}], while Assumption \ref{assump:temporalDependence} states a sufficiently weak temporal dependence and a lower bound for the (long-run) variances [see \cite{Zhang2018b}].
Assumption \ref{assump:spatialDependence} controls the long-run correlations between different components.
The fact that the correlations $\rho_{i,j}$ are sufficiently small for a large distance $|i-j|$ is crucial to obtain the desired extreme value convergence.
Both parts, \eqref{assump:SD1} and \eqref{assump:SD2}, are in line with those proposed for the retrospective change point method of \cite{Jirak2015}.
We point out, that the results of this paper remain correct in the case, where condition \eqref{assump:SD1} in Assumption \ref{assump:spatialDependence}
is only satisfied after an appropriate permutation of the spatial components.
This feature reflects the fact, that the proposed maximum aggregation is invariant with respect to the order of the components.

\subsection{Monitoring using asymptotic quantiles}\label{sec31} 
In this section, we develop the asymptotic theory to define a quantile in the monitoring scheme \eqref{hol1}.
Our first result provides the basis for the proof of the main Theorem \ref{thm:mainGumbel} of this section.
It is stated here because of its independent interest. It shows that the distribution of the maximum of dependent copies of the random variable $\M$ in \eqref{conv:1dim} is in the domain of attraction of the Gumbel distribution if the dependence structure is sufficiently weak.

\begin{theorem}
\label{thm:gumbelWiener.dep}
Let $(W_1,\ldots,W_d)^\top $ be a $d$-dimensional Brownian motion with correlation matrix 
$\Sigma_d =(\rho_{i,j}^{(d)})_{1\leq i,j\leq d}$ whose entries satisfy
$\sup_{1\le i<j\le d}|\rho_{i,j}^{(d)}|\leq\rho_+\in[0,1)$.
Assume that there exists a sequence $L_d$, such that $\rho_{i,j}^{(d)}= 0 $ if $ |i-j| > L_d $ and
\begin{align}\label{assump:Ld}
L_d = o(d^{\Delta})\;\; \text{for some}\;\; \Delta<\left(\frac{1- \rho_+ \sqrt{2-\rho_+^2}}{ 1-\rho_+^2} \right)^2 \text{ as } d \to \infty\,.
\end{align}
Further, denote for $q \in (0,1]$ and $h=1,\dots,d$ by $M_h = \max_{0 \leq t \leq q} W_h(t) - \min_{0 \leq t \leq q} W_h(t)$ the range of the Brownian motion $W_h$ in the interval $[0,q]$.
Then we obtain for $d\to \infty$
\begin{align}\label{eq:maxconv}
a_d \Big ( \maxhd M_h - b_d \Big ) \convd G\,,
\end{align}
where $G$ denotes a standard Gumbel distribution with cdf $F_G(x) = \exp(-\exp(-x))$.
The scaling sequences $a_d,b_d$ are given by
\begin{equation}\label{eq:adbd}
a_d=\sqrt{\frac{2 \log d}{q}} \quad \text{ and } \quad b_d=\sqrt{2q \log d}-\frac{\sqrt{q} \,(\log \log d - \log \tfrac{16}{\pi})}{2\sqrt{2 \log d}}\,.
\end{equation}
\end{theorem}

\noindent 
In the proof of Theorem \ref{thm:mainGumbel} below, we will use a Gaussian approximation which leads to the maximum of the ranges $M_h$
such that Theorem \ref{thm:gumbelWiener.dep} can be applied.
As indicated by \eqref{size}, the limit distribution of the statistic $\hatTmd$ defined in \eqref{hol1} has to be derived for the case $m,d\rightarrow\infty$ in order to determine an appropriate asymptotic critical value.
For this purpose, recall the definition of $\hatE_{m,h}(k)$ in \eqref{hol31} and define by 
\begin{align}\label{def:Tmd}
\Tmd &:= \maxhd \maxkTm w(k/m)\dfrac{\hatsigma_h}{\sigma_h} \hatE_{m,h}(k) = \maxhd \maxkTm w(k/m) \max_{j=0}^{k-1} \dfrac{k-j}{\sqrt{m}\sigma_h} \Big|\hatmu_{m+j+1}^{m+k}(h) - \hatmu_1^{m+j}(h) \Big|
\end{align}
a version of the statistic $\hatTmd$, where all component-wise long-run variance estimators $\hatsigma_h$ have been replaced by the (unknown) true long-run variances $\sigma_h$.

\noindent
In the remainder of this paper, we will work with the sequences $a_d, b_d$ defined in \eqref{eq:adbd} with $q := \qT = T/(T+1)$.
The following theorem yields the asymptotic distribution of $\Tmd$ as $m,d \to \infty$.

\begin{theorem}\label{thm:mainGumbel}
Suppose that the null hypothesis $H_0$ defined in \eqref{def:hypos} holds.
Under the Assumptions \ref{assump:model} - \ref{assump:spatialDependence} it follows that
\begin{align*}
a_d \big( \Tmd - b_d \big) \convd G~~, \mbox{ as } m,d \to \infty~,
\end{align*} 
where $G$ denotes a standard Gumbel random variable.
\end{theorem}

\noindent
Note that due to the choice of $a_d, b_d$ the limit distribution does not depend on the monitoring parameter $T$, which controls the length of the monitoring period, while the statistic $\hatTmd$ does depend on $T$.

Given Theorem \ref{thm:mainGumbel} our final task is to identify suitable long-run variance estimators to obtain the asymptotic distribution of $\hatTmd$.
We will identify a general condition on the estimators in Assumption \ref{assump:varEstimator}, which guarantees that all true long-run variances $\{\sigma_h^2\}_{h=1, \ldots , d} $ in the statistic $\Tmd $ can be replaced by their corresponding estimators.
Explicit estimators satisfying this assumption are constructed in Remark \ref{rem:LRestimator}.

\begin{assump}\label{assump:varEstimator}
{\rm
Suppose that there exists a long-run variance estimator $\hatsigma_h=\hatsigma_h(m)$ based only on the stable initial set, such that $\Pb\big( \maxhd |\hatsigma_h -\sigma_h | \geq m^{-\delta_{\sigma}} \big) \leq c_{\sigma}' m^{-C_{\sigma}'}$,
where $c_\sigma'$ is a sufficiently large constant and $C_{\sigma}'>0$, $\delta_{\sigma}>0$ are sufficiently small constants.
}
\end{assump}

\begin{remark}\label{rem:LRestimator}
In the field of sequential change point detection it is common to use only the initial stable data for the estimation of the long-run variance as this ensures that the estimate cannot be corrupted by a change [see for instance \cite{Aue2012}, \cite{Wied2013} or \cite{Fremdt2014} among many others].
It follows from \cite{Jirak2015} that Assumption \ref{assump:varEstimator} holds for the standard long-run variance estimators
\begin{align}\label{def:LRVestimator}
\hatsigma_{h,strd}^2
= \hat{\phi}_{0,h} + 2\sum_{t=1}^{H_m} \hat{\phi}_{t,h}\,, \qquad h=1, \ldots , d \,,
\end{align}
where $\hat{\phi}_{t,h}$ denotes the lag $t$ auto-covariance estimator in component $h$, that is
\begin{align}\label{def:autocovestimator}
\hat{\phi}_{t,h}
:= \dfrac{1}{m-t} \sum_{i=t+1}^{m}\big(X_{i,h} - \hatmu_1^{m}(h)\big)\big(X_{i-t,h} - \hatmu_1^{m}(h)\big)~.
\end{align}
The bandwidth parameter $H_m$ in \eqref{def:LRVestimator} is bounded by $m^{\eta}$ for some $\eta=\eta(p,D)$ that fulfills constraints with respect to the constants $D$ and $p$ from Assumptions \ref{assump:model} and \ref{assump:temporalDependence} [see Assumption 2.2 in \cite{Jirak2015}].
\end{remark}

\begin{corollary}
\label{cor:Gumbel}
If the null hypothesis $H_0$ defined in \eqref{def:hypos} holds and the Assumptions \ref{assump:model} - \ref{assump:spatialDependence} and \ref{assump:varEstimator}
are satisfied, it follows that
\begin{align*}
a_d \big( \hatTmd - b_d \big) \convd G\,, \quad\mbox{ as }m,d\rightarrow\infty\,~. 
\end{align*}
where $G$ denotes a standard Gumbel random variable.
In particular, $\hatTmd / \sqrt{\log d} $ converges in probability to $\sqrt{2 q(T)}$.
\end{corollary}\vspace{0.2cm}
\noindent 
If $g_{1-\alpha}$ denotes the $(1-\alpha)$ quantile of the standard Gumbel distribution, we obtain from Corollary \ref{cor:Gumbel} that the sequential procedure defined by \eqref{hol1} with $c_{d,\alpha}:= g_{1-\alpha}/a_d+b_d$ has asymptotic size $\alpha$, i.e.
\begin{align}\label{eq:levelalpha}
\limmd \Pb_{H_0}\Big(\hatTmd > \dfrac{g_{1-\alpha}}{a_d} + b_d \Big)
= \alpha\,.
\end{align}
The next theorem yields consistency of this monitoring scheme under the alternative hypothesis of a change in the mean vector.

\begin{theorem}\label{thm:alternative}
Under the alternative hypothesis $H_1$ defined in \eqref{def:hypos}, assume that there is a component $h^*$ and a time point $k^*=k^{*}(m)$ such that
\begin{align}\label{assump:alt1}
\sqrt{\dfrac{m}{\log m}}\cdot \big|\mu_{m+k^*-1,h^*} - \mu_{m+k^*,h^*}\big| \to \infty
\;\;\;
\text{and}
\;\;\;
\limsup_{m\to \infty} \dfrac{k^*}{m} < T~.
\end{align}
If Assumptions \ref{assump:model} - \ref{assump:spatialDependence} and \ref{assump:varEstimator} are satisfied, it follows that
\begin{align*}
\limmd \Pb_{H_1}\Big(\hatTmd > \dfrac{g_{1-\alpha}}{a_d} + b_d \Big) = 1~.
\end{align*}
\end{theorem}\vspace{0.2cm}
\noindent Condition \eqref{assump:alt1} shows that the test is able to detect alternatives which converge to the null hypothesis at the rate of $m^{-1/2}$ up to a factor $c_m \sqrt{\log m}$ with a sequence $(c_m)_{m\in \mathbb{N}}$ tending to $\infty$ at an arbitrary slow rate.
This factor is needed to address for the high dimensional setting.
Note also the time $m+k^*$ of the change is not permitted to be close to the end $m + mT$, which reflects the necessity to have a reasonably large sample after the change point such that the corresponding means can be estimated with sufficient precision.

\begin{remark} \label{aggre}

In this paper, we used the maximum statistic for the aggregation of the different components, which has been used before for other problems [see, for example, \cite{Tartakovsky2006}].
Alternative statistics for aggregation could be used as well and there are numerous possibilities to choose from.
The choice of the statistic depends on the type of alternative, which one is interested in. 
For example, robust schemes as proposed by \cite{Mei2010} are based on the sum of (weighted) local statistics.
In the present context, such a statistic reads as 
\begin{align}
\sum_{h=1}^d \maxkTm w(k/m)\hatE_{m,h}(k)~,
\end{align}
which we suspect to be asymptotic Gaussian under appropriate standardization.
It has been argued (mainly by simulations) that maximum-type statistics are more effective than sums when potential changes occur in only a few data streams
[see \cite{Mei2010,Xie2013,Zou2015}].
On the other hand, when the change occurs in a moderate or large number of components, sum-type statistics outperform the maximum.

An improvement could also be obtained using the spatial information to construct a detector from the individual components.
In the context of offline change point detection such a strategy was proposed by \cite{Wansam2017} for the case of independent Gaussian data (in time and space), where sample splitting is used  to estimate ``optimal directions''.
They also indicate how the methodology can be extended to either temporally or spatially independent Gaussian data. 
In principle, such a strategy could be used in online monitoring of (non-Gaussian) time and stationary dependent data with a stable initial sample.
Alternatively, one could partition the estimate of the covariance matrix from the first step into blocks of highly correlated elements, aggregate the components corresponding to each block by a sum and finally calculate the maximum of sums corresponding to the different blocks. 
As indicated by the discussion in \cite{Wansam2017} a theoretical analysis of such a two stage procedure in the general case considered in this paper is very difficult (maybe intractable).
Moreover, independently of the theoretical aspects, some care is necessary, if such strategies are used in applications.
In particular, it requires estimates of the covariance structure with sufficient precision.
If the estimates from the first step have a too large variability adapted aggregation schemes may perform worse than non-adapted.
To our knowledge, this phenomenon has not been discussed in the context of aggregation of statistics from individual components of high dimensional time series but it is well known in other circumstances such as adaptive designs - see, for example, \cite{Dette2013}.
For high dimensional time series it is difficult to obtain estimates of the (long-run) covariances with sufficient precision such that it can be used in a two-step procedure.
Therefore we do not recommend the use of an adaptive aggregation in this context.
\end{remark}

\subsection{Bootstrap quantiles}\label{sec32}
Equation \eqref{eq:levelalpha} and Theorem \ref{thm:alternative} show that the new sequential testing procedure \eqref{hol1} with $c_{d,\alpha}= g_{1-\alpha}/a_d+b_d$ has asymptotic level $\alpha$ and is consistent.
However, the approach so far is based on an approximation of the distribution of the statistic by a Gumbel distribution featuring the well-known disadvantage that the convergence rates in such limiting results are rather slow.
As a consequence for small sample sizes, these quantiles may yield slightly imprecise approximations in practical applications.
To tackle this problem, we will propose a bootstrap procedure.
Note that the development of resampling procedures in the sequential regime is a difficult problem. 
On the one hand, critical values can be computed only from the initial stable sample, but this set can be too small to obtain reliable values.
On the other hand, one can compute new critical values with each new data point, which is computationally expensive and can be corrupted by an undetected structural break.
Therefore, both approaches have natural advantages and disadvantages.
For i.i.d.~data \cite{Kirch2008} proposes a bootstrap procedure for sequential detection of a structural break in the mean of an one-dimensional sequence by a combination of both methods following ideas of \cite{Steland2006}.
Nevertheless, the construction of bootstrap methodology for sequential change point detection in the high dimensional regime remains challenging.

To be precise, note that by Lemma \ref{lem:Truncation} and \ref{lem:gaussianapprox} used in the proof of Theorem \ref{thm:mainGumbel} in the online supplement we obtain the approximation
\begin{align}\label{eq:BootstrapClose}
\Pb \Big( a_d\big(\Tmd-b_d\big) \leq x \Big) - \Pb \Big( a_d\big(\TmdZ - b_d\big) \leq x \Big) = o(1)~.
\end{align}
Here the statistic $\TmdZ$ is the counterpart of $\Tmd$ computed from standard Gaussian random variables $Z_{t,h}$, which are independent in time and have spatial dependence structure $\Cov\big(Z_{0,h},Z_{0,i}\big) = \rho_{h,i}$,
where $\rho_{h,i}$ are the long-run correlations defined in \eqref{eq:longruncorr}.
In view of the approximation \eqref{eq:BootstrapClose}, it is therefore reasonable to obtain the quantiles for the statistic $\Tmd$ from those of the statistic $\TmdZ$, which can easily be simulated if the correlations $\rho_{h,i}$ were known.
These parameters can be straightforwardly estimated from the initial stable data set $\bm{X}_1,\dots,\bm{X}_m$.

Compared to a bootstrap procedure continuously performed during monitoring, this idea exhibits two important advantages.
Firstly, it ensures that the correlation estimates cannot be corrupted by a mean change, that may occur during the monitoring period.
Secondly, it requires less computational effort, as the quantile is only computed once before monitoring is commenced.
This is of vital importance in a high dimensional setup, where the method on its own is already quite expensive and resampling and/or repeated estimations during monitoring may quickly exceed the computational resources.

Before discussing the technical details of this resampling procedure, we state a necessary assumption regarding the precision of the estimates of the long-run covariances.
\begin{assump}\label{assump:Bootstrap}
{\rm
Suppose that there exists a long-run covariance estimator $\hatgamma_{h,i}=\hatgamma_{h,i}(m)$ based on the stable initial set $\bm{X}_1,\dots,\bm{X}_m$, such that $\Pb\big( \max_{h,i=1}^d \big| \hat{\gamma}_{h,i} - \gamma_{h,i} \big| \geq m^{-\delta_\gamma} \big) \leq c_\gamma m^{-C_\gamma}$,
where $c_\gamma$ is a sufficiently large constant and $C_\gamma>0$, $\delta_{\gamma}>0$ are sufficiently small constants.}
\end{assump}

\begin{remark}
A canonical choice for a long-run covariance estimator that satisfies Assumption~\ref{assump:Bootstrap}, is the standard estimator
\begin{align}\label{def:lrcov}
\hat{\gamma}_{h,i,strd}
= \hat{\phi}_{0,h,i} + \sum_{t=1}^{H_m} \hat{\phi}_{t,h,i} + \sum_{t=1}^{H_m}\hat{\phi}_{t,i,h}~,
\end{align}
where $H_m$ is an appropriate bandwidth and the involved cross-components covariance estimators are given by $\hat{\phi}_{t,h,i}
:= (m-t)^{-1} \sum_{j=t+1}^{m}\big(X_{j,h} - \hatmu_1^{m}(h)\big)\big(X_{j-t,i} - \hatmu_1^{m}(i)\big)$.
Note that these definitions are natural extensions of the long-run variance and auto-covariance estimators in \eqref{def:LRVestimator} and \eqref{def:autocovestimator}, respectively.
Therefore one can use similar arguments as given in \cite{Jirak2015} (for the verification of Assumption \ref{assump:varEstimator}) to prove the consistency stated in Assumption~\ref{assump:Bootstrap}.
\end{remark}\vspace{0.2cm}

\noindent In the following, denote by $\mathcal{X}=\sigma(\bm{X}_1,\dots,\bm{X}_m)$ the $\sigma$-algebra generated by the initial sample and let $\PbX, \Cov_{|\mathcal{X}}$ denote the conditional probability and covariance with respect to $\mathcal{X}$.
To define the bootstrap statistic, let
\begin{align}\label{eq:BootstrapZ}
\Big\{ \widehat{\bm{Z}}_t=\big(\hatZ_{t,1},\dots,\hatZ_{t,d}\big)^\top \Big\}_{t=1,\dots,m+Tm}
\end{align}
denote centered random vectors, that are - conditionally on $\mathcal{X}$ - independent and Gaussian distributed with covariance structure
\begin{align}\label{eq:CovarBootstrapZ}
\Cov_{|\mathcal{X}}\Big(\hatZ_{t,h},\,\hatZ_{t,i}\Big) = \hat{\rho}_{h,i}~,
\end{align}
where $\hat{\rho}_{h,i}$ are correlation estimators canonically defined by $\hat{\rho}_{h,i} = \hatgamma_{h,i} /(\hatsigma_h\hatsigma_i)$ for $h\neq i$ and \\ $\hat{\rho}_{h,h} = 1$.
Note that by the definition in \eqref{eq:CovarBootstrapZ} the random vectors preserve the (estimated) spatial correlation structure of the time series.
Next, denote the component-wise mean estimators for subsamples of $\{\hatZ_{t,h}\}_{t=1}^{m+Tm}$ by 
$\hat{z}_i^j(h) := (j-i+1)^{-1} \sum_{t=i}^j \hatZ_{t,h}$
and the final bootstrap statistic by
\begin{align}\label{def:BSstatistic}
\hatTmdZ :=
\max_{k=1}^{Tm} \maxhd \max_{j=0}^{k-1} \dfrac{w(k/m)(k-j)}{\sqrt{m}}\Big| \hat{z}_{m+j+1}^{m+k}(h) - \hat{z}_1^{m+j}(h) \Big|~.
\end{align}
Once the correlation estimates $\hat{\rho}_{h,i}$ are computed, the conditional distribution of $\hatTmdZ$
can be approximated by Monte-Carlo simulations with arbitrary precision, generating replicates of $\{\hatZ_{t,h}\}_{t=1, \ldots ,m +mT}^{h=1,\ldots d} $.
Thus, provided with a batch of realizations of the statistic $\hatTmdZ$, one can compute the corresponding empirical quantile for the desired test level and launch the sequential procedure with this bootstrap quantile instead of the (possibly less precise) Gumbel quantile.
The following result yields the validity of our proposed bootstrap procedure.
\begin{theorem}[Bootstrap consistency]\label{thm:Bootstrap}
Under the Assumptions \ref{assump:model}-\ref{assump:spatialDependence} and \ref{assump:Bootstrap} we have 
\begin{align*}
\sup_{x \in \R} 
\bigg| \Pb_{|\mathcal{X}}\Big( a_d\big(\hatTmdZ - b_d\big) \leq x \Big) 
- \Pb_{H_0}\Big( a_d\big(\hatTmd - b_d\big) \leq x \Big) \bigg|
= o_{\Pb}(1)~,
\end{align*}
where $\Pb_{H_0}$ denotes the probability under the null hypothesis of no change in any component.
\end{theorem}

\noindent 
Combining Theorem \ref{thm:Bootstrap} with Corollary \ref{cor:Gumbel} and Theorem \ref{thm:alternative}, it follows that the use of the quantiles of the bootstrap distribution in \eqref{hol1} yields a consistent monitoring scheme, which keeps its pre-specified nominal level.

\begin{remark} {\rm ~\\
(1) Note that the bootstrap procedure uses correlation estimates from the stable sample, but does not use these estimates for the aggregation of the componentwise detectors in the monitoring scheme.
\\
(2) An extension of methodology is provided in Appendix \ref{sec:diffchanges} in the online supplement.
We consider possible non-simultaneous change points in different components and in Theorem \ref{thm:algorithmConsistent} we analyze an algorithm which can identify all components that are affected by a change.
}
\end{remark}

\section{Finite sample properties}\label{sec:finitesamples}

In this section, we investigate the finite sample properties of the new monitoring schemes by means of a simulation study and illustrate potential applications in a data example.

\subsection{Simulation study} \label{sec41} 
In our simulation study we consider the following models:
\begin{enumerate}
\item[(M1)] $X_{t,h} = \theta_{t,h}$~,~~~~~~~~~~~~~~~~~~~~~~~~~~~~~~~~~ 
(M2)~$X_{t,h} = 0.1X_{t-1,h} + \varepsilon_{t,h}$~,
\item[(M3)] $X_{t,h} = \eta_{t,h} + 0.3\eta_{t-1,h} -0.1\eta_{t-2,h}$~,~~~~
(M4)~$X_{t,h} = \tilde{\varepsilon}_{t,h}$~,
\end{enumerate}
where $\{\theta_{t,h}\}_{t \in \N, h \in \N}$ is an array of i.i.d. random variables that follow a $\chi^2_{10}/\sqrt{20}$-distribution, $\{\varepsilon_{t,h}\}_{t \in \N, h \in \N}$ is an array of i.i.d. standard Gaussian random variables,
$\{\eta_{t,h}\}_{t \in \N, h \in \N}$ is an array of i.i.d. random variables that follow a Laplace(0,1)-distribution and $\{\tilde{\varepsilon}_{t,h}\}_{t \in \N, h \in \N}$ are random variables, such that $\big\{\tilde{\bm{\varepsilon}}_t=(\tilde{\varepsilon}_{t,1},\dots, \tilde{\varepsilon}_{t,d})^\top \big\}_{t \in \Z}$ are i.i.d. $d$-dimensional, centered Gaussian random vectors with covariance structure
$\Cov(\tilde{\varepsilon}_{1,j}, \tilde{\varepsilon}_{1,i})= 1/(|j-i|+1)$.
For the alternative hypothesis, we also consider the models (M1)-(M4) and add a shift in the mean, at some point $m+k^*$, that is 
\begin{align}\label{eq:addChangesArt}
X_{t,h}^{(\delta ,A)} =
\begin{cases}
X_{t,h} & \text{if }~ t < m + k^*~, \\
X_{t,h} + \delta\cdot I\{h \in A\} & \text{if }~ t \geq m + k^*~,
\end{cases}
\end{align}
where $I$ denotes the indicator function, $A$ is the set of spatial components affected by the change and $\delta$ is the size of the change.
In order to examine the influence of both parameters on the procedure, we will consider different values of $\delta$ and three different choices 
of the set $A$ below.

For the long-run variance estimation, we use the quadratic spectral kernel estimator [see \cite{Andrews1991}] in each component, that is
\begin{align}\label{eq:varEstimator}
\hat{\sigma}_h^2 
:= \sum_{|t| \leq m-1} k\Big( \dfrac{t}{H_m}\Big)\hat{\phi}_{t,h}~,
\end{align}
with the empirical auto-covariances $\hat{\phi}_{t,h}
:= m^{-1}\sum_{i=t+1}^{m}\big(X_{i,h} - \hatmu_1^{m}(h)\big)\big(X_{i-t,h} - \hatmu_1^{m}(h)\big)$
and underlying kernel 
\begin{align*}
k(x) := \dfrac{25}{12\pi^2x^2}\bigg( \dfrac{\sin(6\pi x/5)}{6\pi x/5} - \cos(6\pi x/5)\bigg)~.
\end{align*}
In particular, we employ the implementation of the estimator \eqref{eq:varEstimator} provided by the R-package 'sandwich' [see \cite{Zeileis2004}] and select the bandwidth parameter as $H_m=\log_{10}(m)$.
Note that we only use the stable set $\bm{X}_1,\dots,\bm{X}_m$ for the estimation of the long-run variance, which avoids corruption from observations after the potential change point under the alternative [see the discussion in Remark \ref{rem:LRestimator}].
All results presented in this section are based on $1000$ simulation runs and for the AR(1)-process (M2) we employ a burn-in sample of 200 observations.
The test level is always fixed at $\alpha = 0.05$.
\begin{table}[t]
\centering
\begin{tabular}{cccccccccc}
\hline 
\vspace{2mm} 
&& \multicolumn{2}{c}{m=100} & \multicolumn{2}{c}{m=200} & \multicolumn{2}{c}{m=500} \\
T& model & d=100 & d=200 & d=200 & d=500 & d=200 & d=500\\[5pt]
\hline
\Tstrut
\multirow{4}{*}{1}
&(M1) & 10.4\% & 13.7\% &  6.7\% & 10.2\% &  4.5\% &  6.1\%\\
&(M2) &  7.4\% & 11.3\% &  8.0\% &  9.4\% &  4.1\% &  4.9\%\\
&(M3) &  5.8\% &  7.6\% &  4.2\% &  7.4\% &  3.6\% &  3.2\%\\
&(M4) &  4.2\% &  7.3\% &  4.1\% &  5.6\% &  2.8\% &  3.8\%\\[5pt]
\hline
\Tstrut
\multirow{4}{*}{2}
&(M1) &  9.9\% & 12.9\% &  7.7\% & 10.5\% &  3.8\% &  6.6\%\\
&(M2) &  9.0\% & 11.2\% &  7.5\% &  9.2\% &  5.0\% &  4.4\%\\
&(M3) &  5.7\% &  7.4\% &  4.6\% &  5.1\% &  2.3\% &  2.8\%\\
&(M4) &  4.7\% &  7.3\% &  5.2\% &  7.0\% &  3.0\% &  3.7\%\\[5pt]
\hline
\Tstrut
\multirow{4}{*}{4}
&(M1) & 10.9\% & 16.6\% &  8.9\% & 11.1\% &  5.0\% &  5.5\%\\
&(M2) &  9.1\% & 12.6\% &  7.7\% &  9.5\% &  3.8\% &  5.4\%\\
&(M3) &  6.0\% &  8.4\% &  4.6\% &  5.6\% &  2.4\% &  2.9\%\\
&(M4) &  6.3\% &  7.2\% &  5.2\% &  5.8\% &  2.7\% &  5.5\%\\[5pt]
\hline
\end{tabular}
\caption{\it Approximation of the nominal level by the detector defined in \eqref{hol1} 
for different choices of initial sample size $m$, dimension $d$ and monitoring duration $m\cdot T$.
Critical values are obtained from Corollary \ref{cor:Gumbel} (approximation by Gumbel distribution).
The nominal level is $\alpha=0.05$.
\label{tab:nullapproxGumbel}}
\end{table}

In Table \ref{tab:nullapproxGumbel} and \ref{tab:nullapproxGaussianQuantiles}, we illustrate the finite sample properties of the detection scheme \eqref{hol1} under the null hypothesis for different choices of the sample size $m$, the dimension $d$ and the length of the monitoring period determined by $T$.
The results in Table \ref{tab:nullapproxGumbel} are based on the weak convergence in Corollary \ref{cor:Gumbel} and therefore we use the critical value $c_{d,\alpha} = g_{1 - \alpha}/a_d+b_d$ in \eqref{hol1}, where $g_{1 - \alpha}$ is the quantile of the Gumbel distribution.

The results in Table \ref{tab:nullapproxGaussianQuantiles} are obtained by the bootstrap procedure as described in Step 2 of Algorithm \ref{alg:complete} in the online supplement.
We note that the use of (any) sequential monitoring scheme in a simulation study is computationally demanding, in particular for a high dimensional setup.
In our case, we have to estimate the spatial correlation structure in each simulation run and then simulate the quantile of the distribution of the statistic $\hatTmdZ$ defined in \eqref{def:BSstatistic}.
Of course, this is no problem in data analysis as in this case the monitoring procedure has only to be conducted once, but it requires large computational resources in a simulation, where the same procedure is repeated $1000$ times.
Therefore, in order to reduce the computational complexity of the bootstrap approach in the simulation study, we do not estimate the spatial correlation structure for the bootstrap but employ temporal and spatial independent Gaussian random variables to generate the bootstrap statistics defined in \eqref{def:BSstatistic}.
With this adaption, the quantiles are fixed within each column of Table \ref{tab:nullapproxGaussianQuantiles}, which makes the simulation study practicable.
Moreover, it can easily be seen from the theory developed in Section \ref{sec3} that the use of these quantiles also yields a consistent test in \eqref{hol1}.
Here we point to Lemma \ref{lem:gumbel} in Appendix \ref{sec:technicalDetails} in the online supplement.

In Table \ref{tab:nullapproxGumbel}, we observe a reasonable approximation of the nominal level by the asymptotic test in many cases, which becomes more accurate with larger initial sample size $m$ and dimension $d$.
For instance, consider the model (M2) for the choice $T=2$, where we have obtained a type I error of $9.0\%$ for $m=d=100$ and $7.5\%$ for $m=d=200$.
This finally reduces to an appropriate approximation of $4.4\%$ for the choice $m=d=500$.
As common for high dimensional procedures, the relation of sample size $m$ and dimension $d$ has a severe impact on the performance of the monitoring procedure.
For example, an empirical type I error of $4.2\%$ was measured for model (M4) with $m=d=100$ and $T=1$, which increases to $7.3\%$ if the dimension is set to $d=200$.
This effect becomes weaker, when the sample size is generally increased.

In Table \ref{tab:nullapproxGaussianQuantiles}, we display the type I error for the method where the quantiles are calculated by the bootstrap as described above.
For the sake of brevity, we focus on the case $T=1$, as the results obtained for different choices of $T$ in Table \ref{tab:nullapproxGumbel} are similar.
We observe a very reasonable approximation of the desired test level and - compared to the results in Table \ref{tab:nullapproxGumbel} - some improvement by the bootstrap procedure for small sample sizes.
In particular, application of the bootstrap is recommended in cases where $m$ and $d$ are relatively small, where the approximation of the nominal level using the quantiles from the Gumbel distribution is rather imprecise, while the extra computational costs for the bootstrap are still tolerable.

\begin{table}[H]
\centering
\begin{tabular}{cccccccccc}
\hline 
\vspace{2mm} 
&& \multicolumn{2}{c}{m=100} & \multicolumn{2}{c}{m=200} & \multicolumn{2}{c}{m=500}\\
T& model & d=100 & d=200 & d=200 & d=500 & d=200 & d=500\\[5pt]
\hline
\Tstrut
\multirow{4}{*}{1}
&(M1) &  9.7\% & 13.1\% &  7.3\% &  8.5\% & {9.8\%} &  {8.4\%}\\
&(M2) &  6.9\% & 11.1\% &  8.1\% &  7.0\% & {8.1\%} &  {6.8\%}\\
&(M3) &  5.2\% &  7.5\% &  4.2\% &  5.6\% &  6.4\% &  4.7\%\\
&(M4) &  4.0\% &  7.0\% &  4.1\% &  4.9\% &  6.4\% &  5.5\%\\[5pt]
\end{tabular}
\caption{\it Approximation of the nominal level using the detector defined in \eqref{hol1} for different choices of initial sample size $m$ and dimension $d$. 
Critical values are computed by the bootstrap with spatial independence.
The nominal level is $\alpha=0.05$.
\label{tab:nullapproxGaussianQuantiles}}
\end{table}

\begin{figure}[H]
\centering
\begin{tabular}{b{1.2cm}cc}
(A1) \tabj\tabj &
\includegraphics[width=6cm, height=4cm]{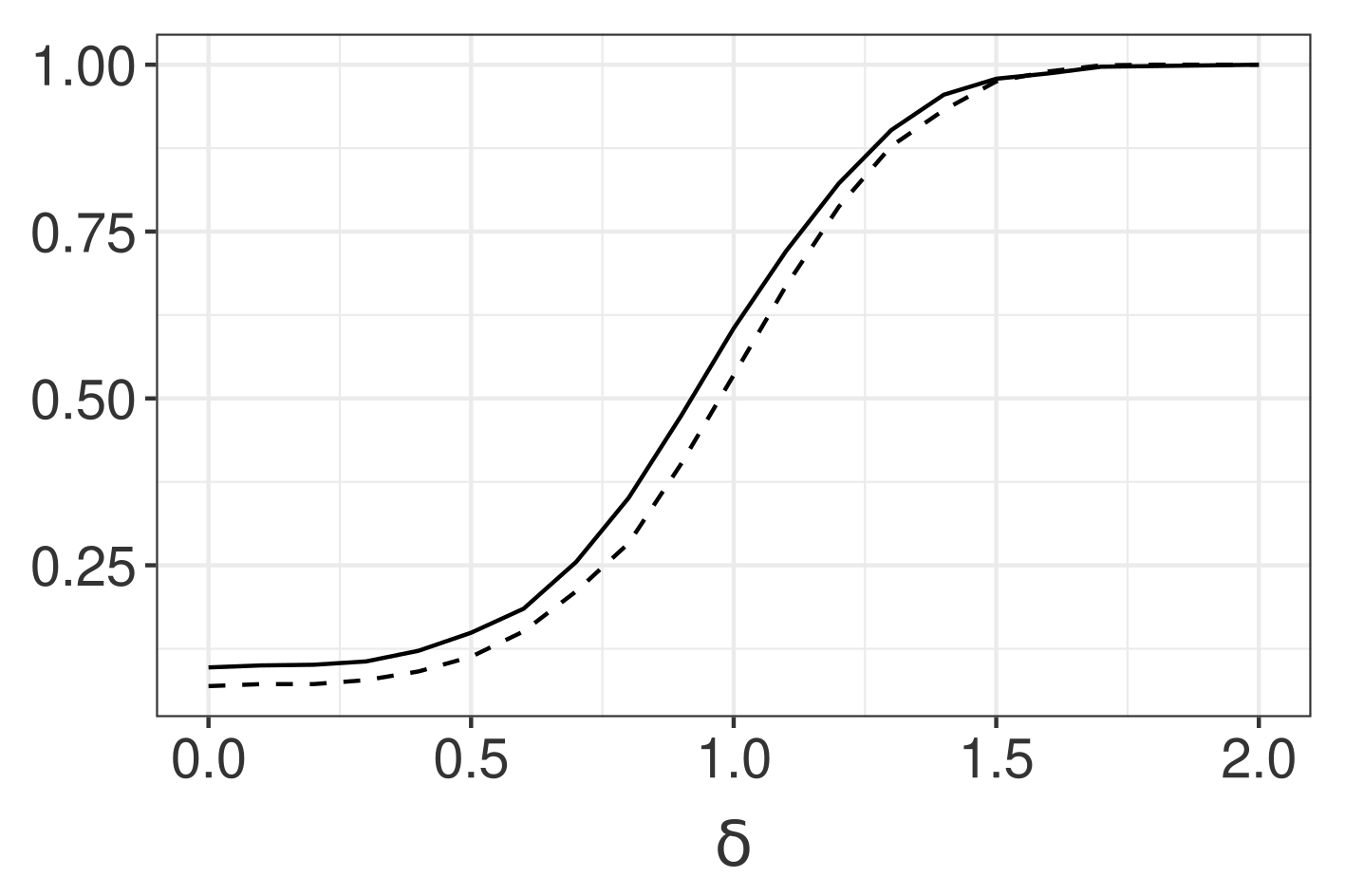} &
\includegraphics[width=6cm, height=4cm]{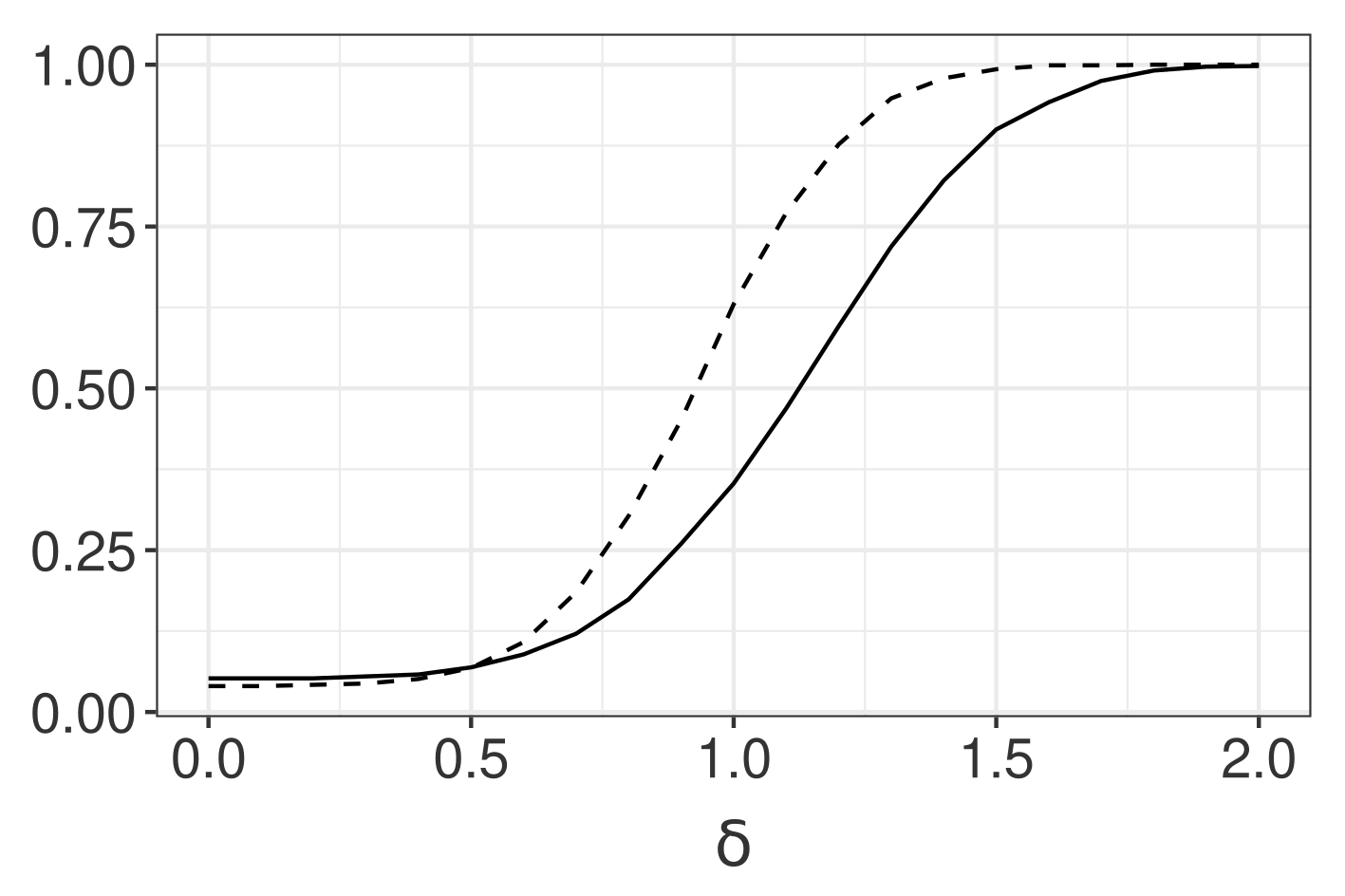} \\
(A2) \tabj\tabj &
\includegraphics[width=6cm, height=4cm]{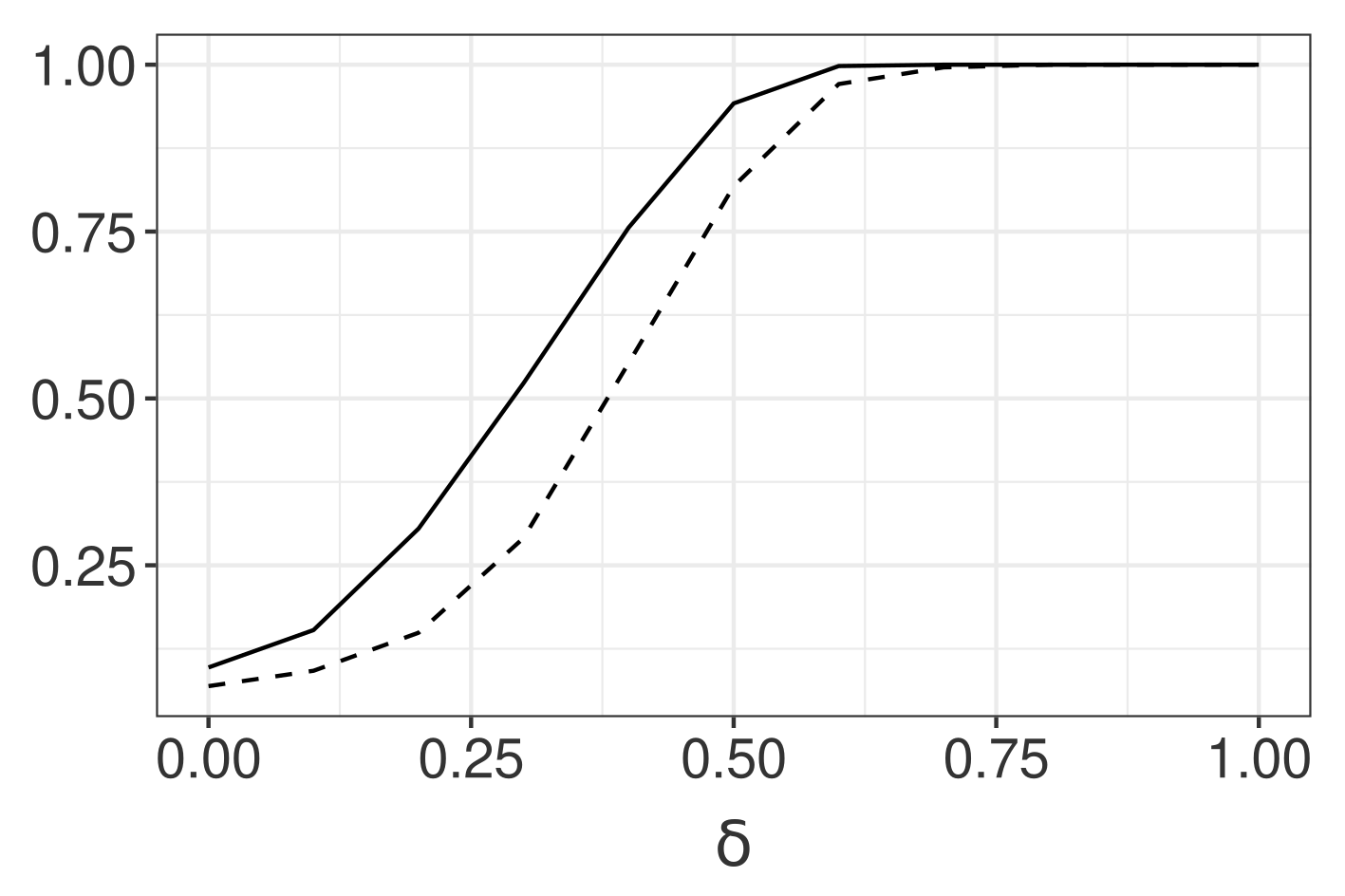} &
\includegraphics[width=6cm, height=4cm]{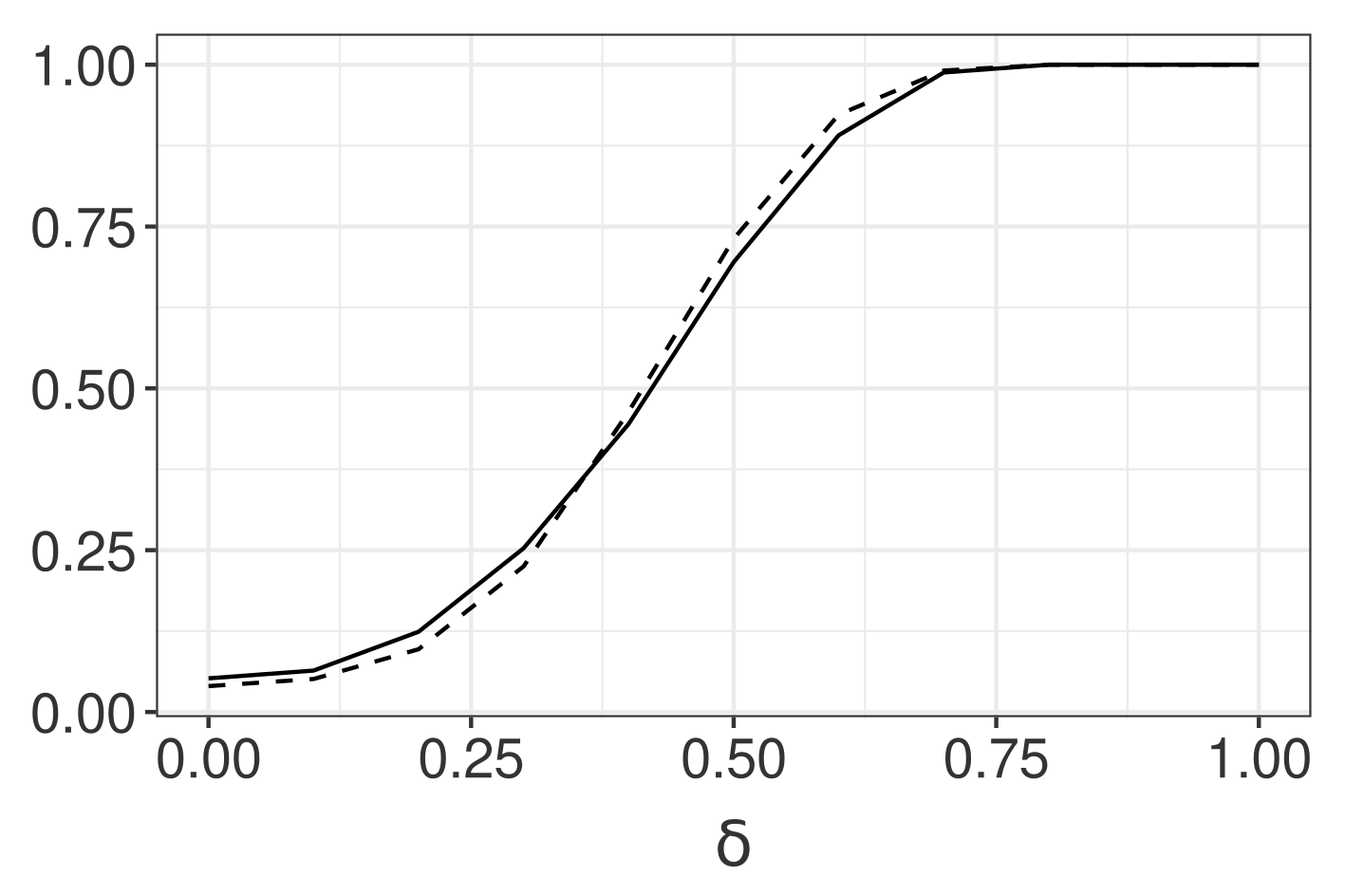} \\
(A3) \tabj\tabj &
\includegraphics[width=6cm, height=4cm]{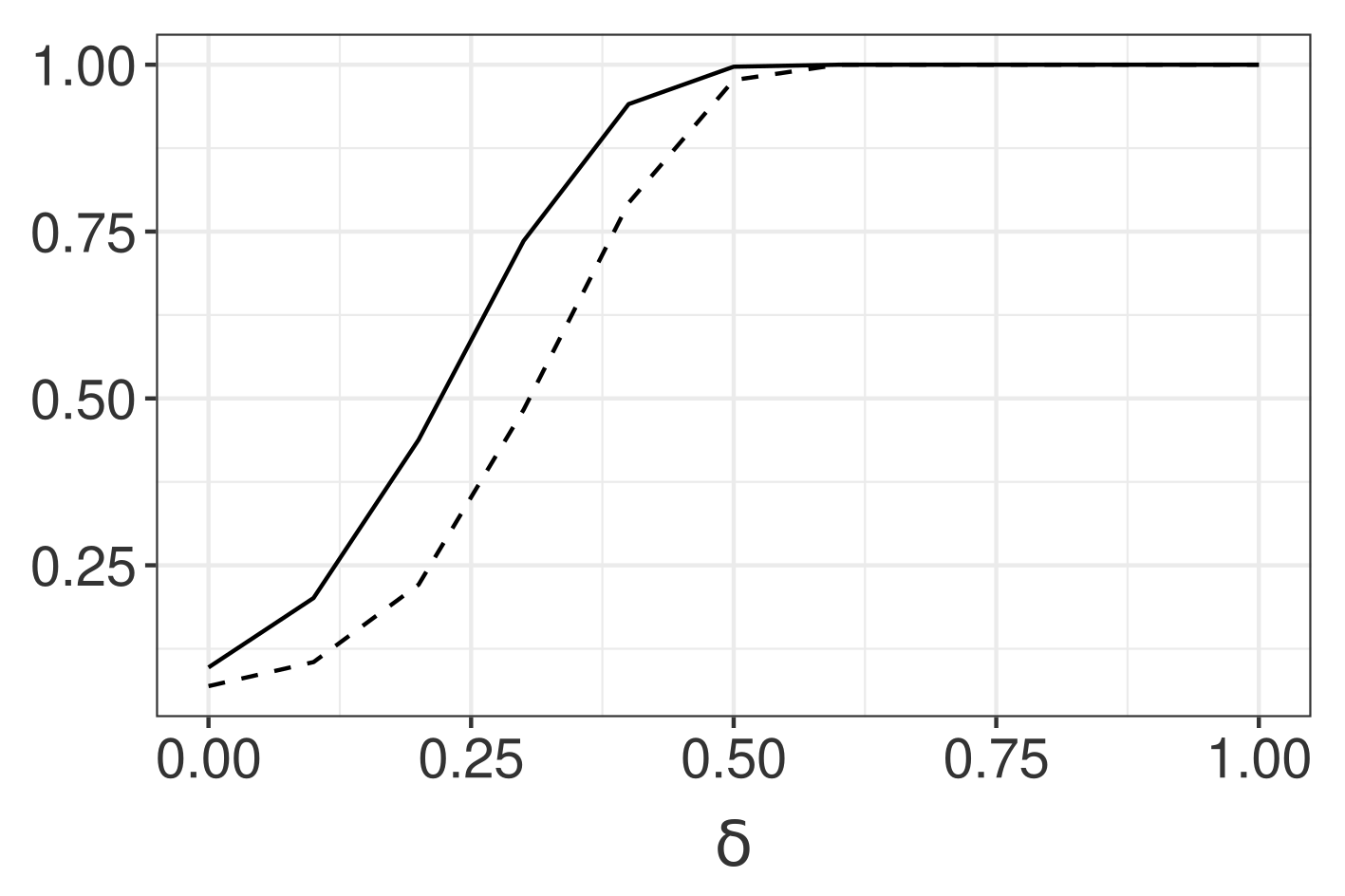} &
\includegraphics[width=6cm, height=4cm]{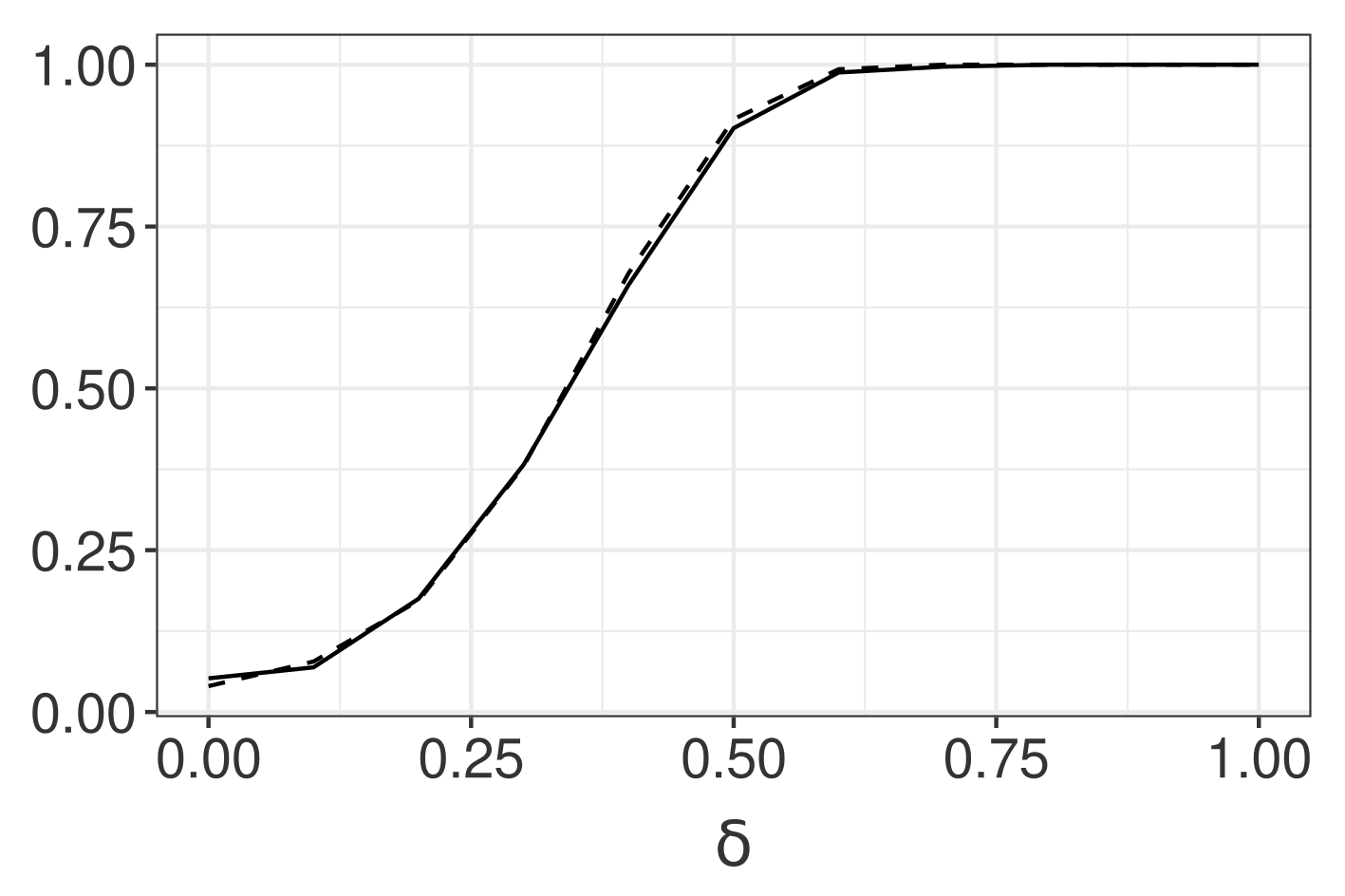}
\end{tabular}
\caption{\it
Simulated power of the monitoring scheme \eqref{hol1} for different size $\delta$ of the change, sample size $m=100$ and dimension $d=100$.
Left panels: Solid line (M1), dashed line (M2).
Right panels: Solid line (M3), dashed line (M4).
The nominal level is $\alpha=0.05$.
\label{fig:alt1}}
\end{figure}

\noindent To analyze the performance of the sequential procedure under the alternative hypothesis we consider the model \eqref{eq:addChangesArt}, where the processes $X_{t,h}$ are defined by (M1)-(M4).
We distinguish between the following three scenarios:
\begin{enumerate}[label=(A\arabic*), leftmargin=.948cm]
\item The change occurs only in one component. This corresponds to the choice $A=\{1\}$.
\item The change occurs 
in $50\%$ of the components, i.e. $A=\{1,\dots,d/2\}$.
\item The change occurs 
in all components, i.e. $A=\{1,\dots,d\}$.
\end{enumerate}
For the sake of brevity and readability, we focus on the case $T=1$ under the alternative and only consider change positions in the middle of the monitoring period, i.e. we fix $k^* = m/2$.
In Figures \ref{fig:alt1} and \ref{fig:alt2}, we display the rejection probabilities of the detection rule \eqref{hol1} for these scenarios, different values of the change, different sample sizes and dimensions.
The critical values in \eqref{hol1} are given by $c_{d,\alpha} = q/a_d + b_d$, where $q$ is the quantile obtained by the spatial independent bootstrap, which appears more accurate than the quantiles derived from the Gumbel distribution.

\noindent The results can be summarized as follows.
In all considered scenarios the new monitoring procedure \eqref{hol1} for a change in the high dimensional mean vector has reasonable power under the alternative, and in all cases the type II error approaches zero for an increasing size $\delta$ of the change.
As expected, the power is lower under alternative (A1), where the change occurs in only one coordinate.
To give an example, consider model (M1) and (M2) corresponding to the left columns in Figures \ref{fig:alt1} and \ref{fig:alt2}.
The results for the different alternatives (A1), (A2) and (A3) can be found in the first, second and third rows of the figures, respectively.
If the sample size and dimension are given by $m=100$ and $d=100$ we observe from Figure \ref{fig:alt1} that for $\delta=0.7$ the power for model (M1) and (M2) under alternative (A1) is approximately given by $0.25$, while it is close to one under alternative (A2).
Interestingly, the differences between alternatives (A2) and (A3) are not so strong, but they are still clearly visible.
For instance, a comparison between the left parts of the second and third row of Figure \ref{fig:alt2} shows that the power of the detection scheme \eqref{hol1} in model (M1) for $\delta=0.3$ is approximately $0.88$ for alternative (A2) and $0.99$ for alternative (A3).

The differences between the four data generating models are in general not substantial with one exception.
In model (M3) under alternative (A1) the power of the detection scheme \eqref{hol1} is considerably smaller [see the first rows in Figure \ref{fig:alt1} and Figure \ref{fig:alt2}].

We summarize the discussion of the finite sample properties emphasizing that our numerical results have supported the theoretical findings developed in Section \ref{sec3} in all cases under consideration.

\begin{figure}[H]
\centering
\begin{tabular}{b{1.2cm}cc}
(A1) \tabj\tabj &
\includegraphics[width=6cm, height=4cm]{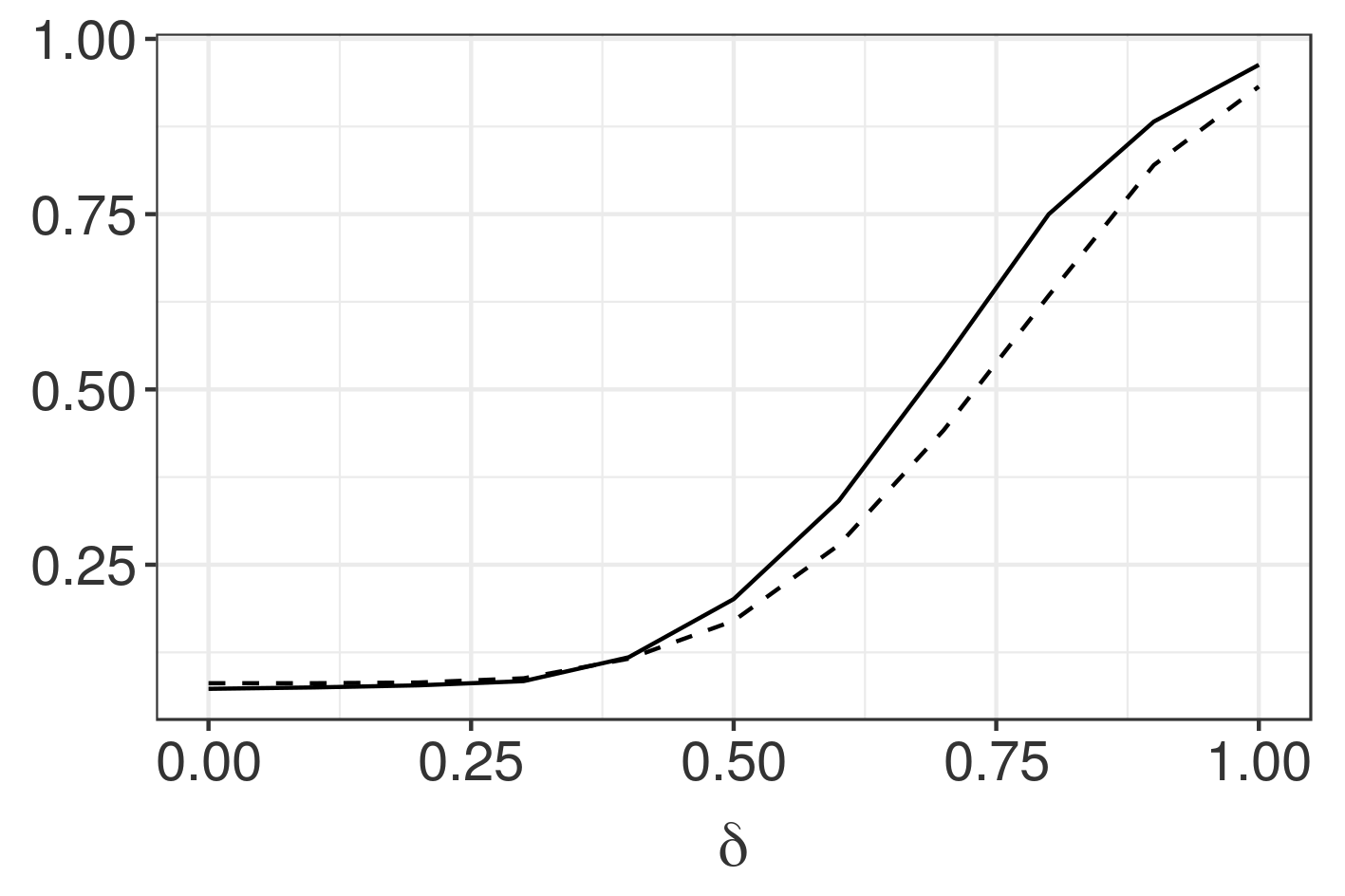} &
\includegraphics[width=6cm, height=4cm]{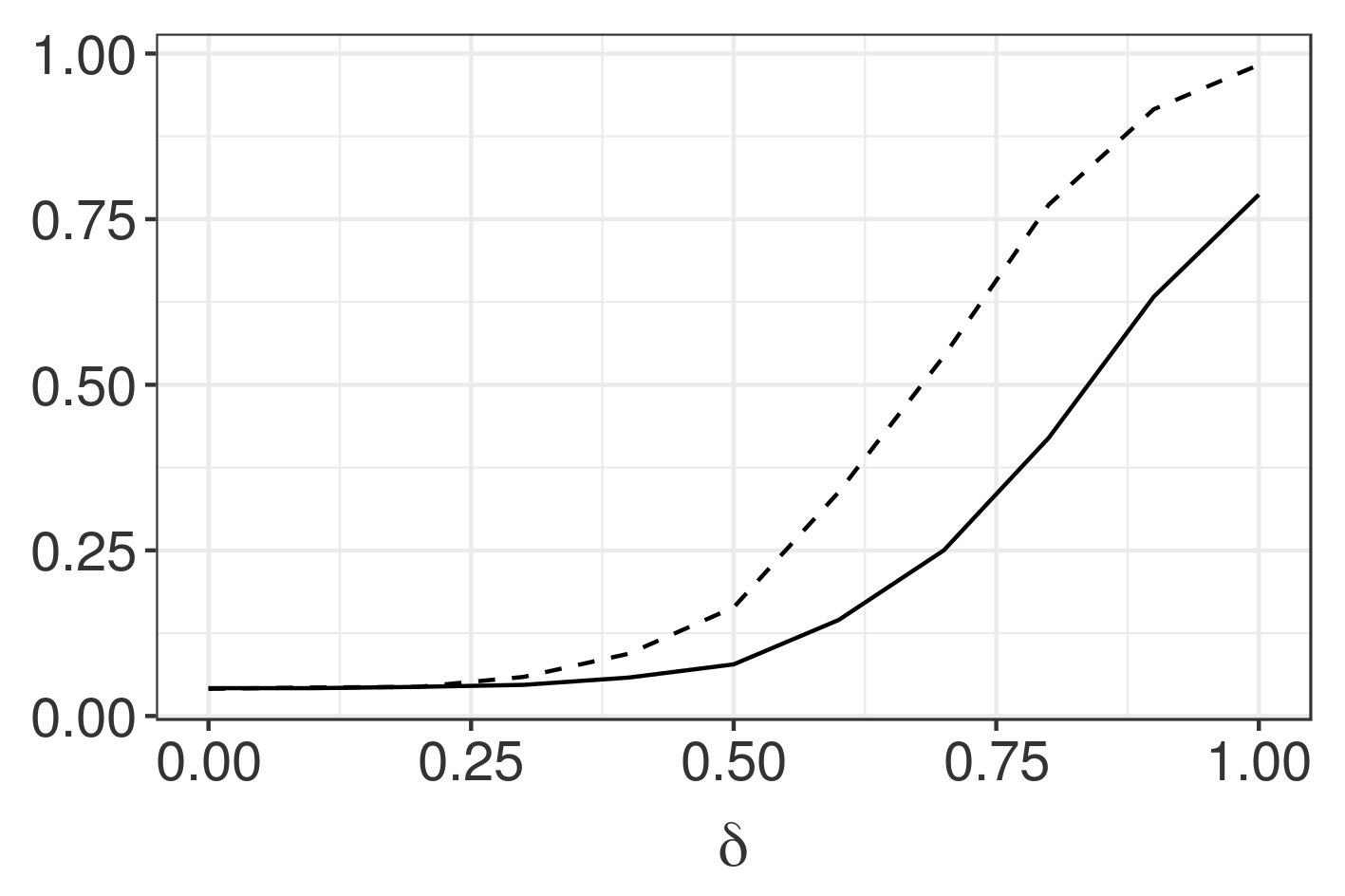}
\\
(A2) \tabj\tabj &
\includegraphics[width=6cm, height=4cm]{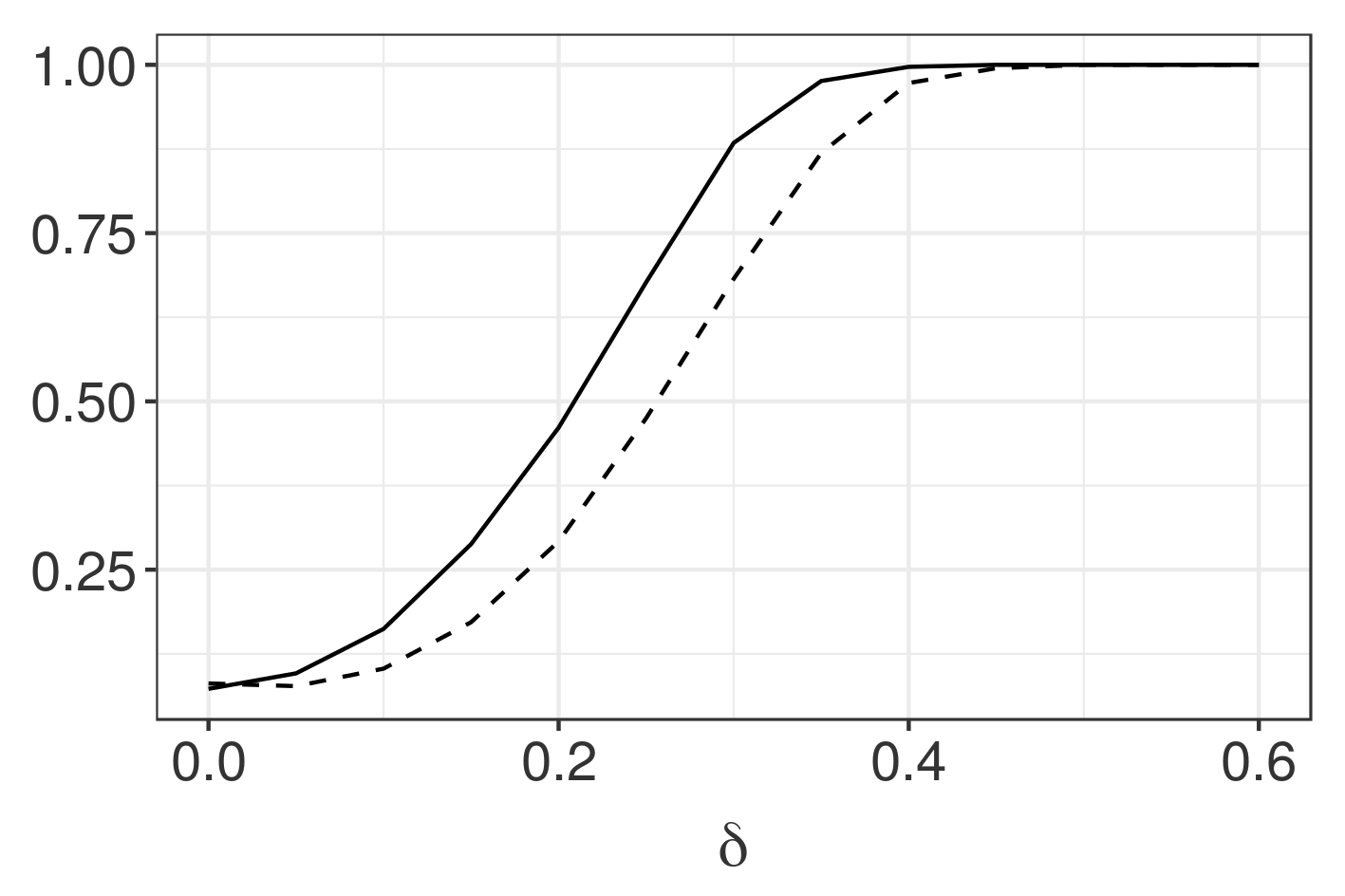} &
\includegraphics[width=6cm, height=4cm]{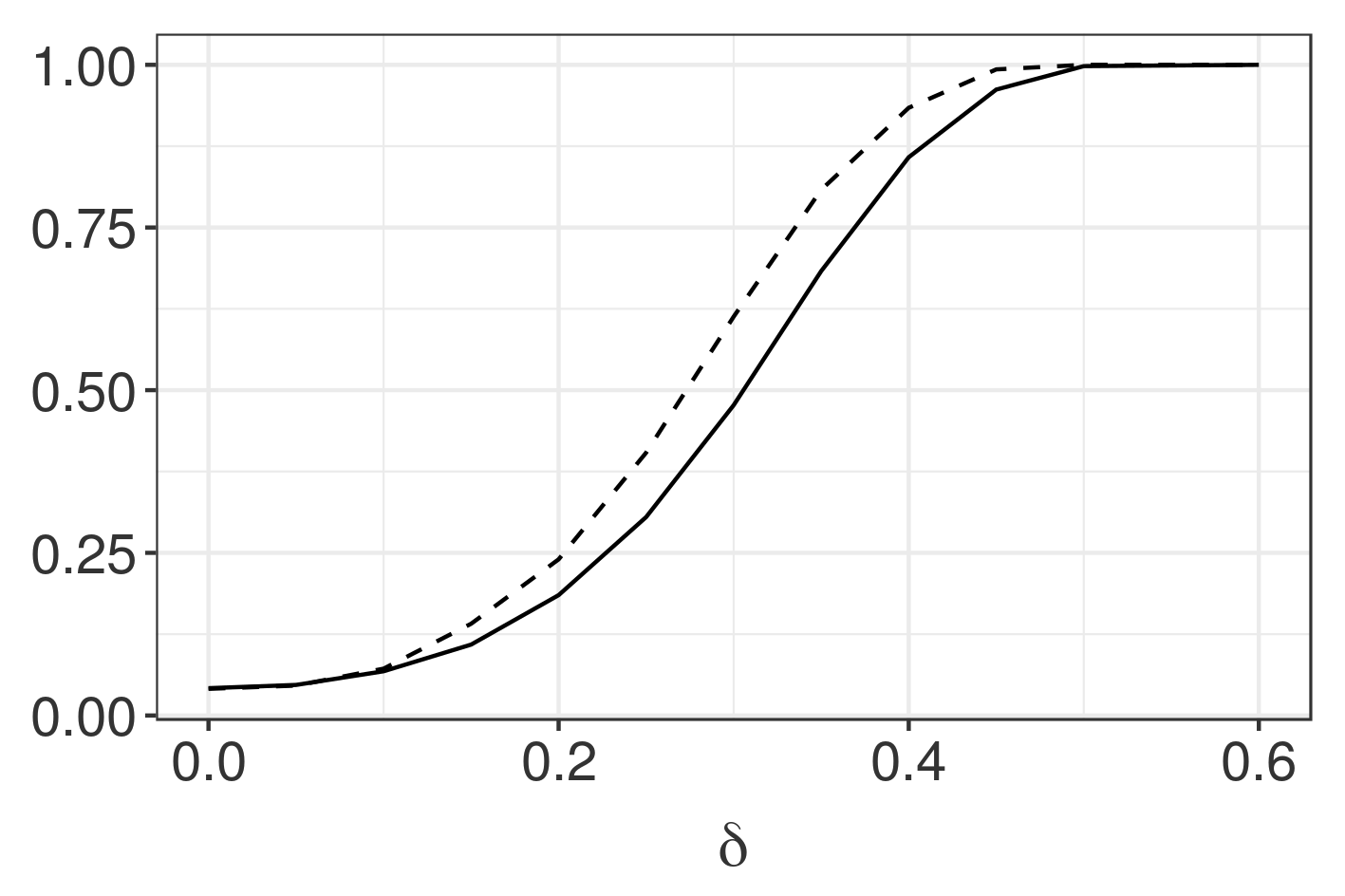}
\\
(A3) \tabj\tabj &
\includegraphics[width=6cm, height=4cm]{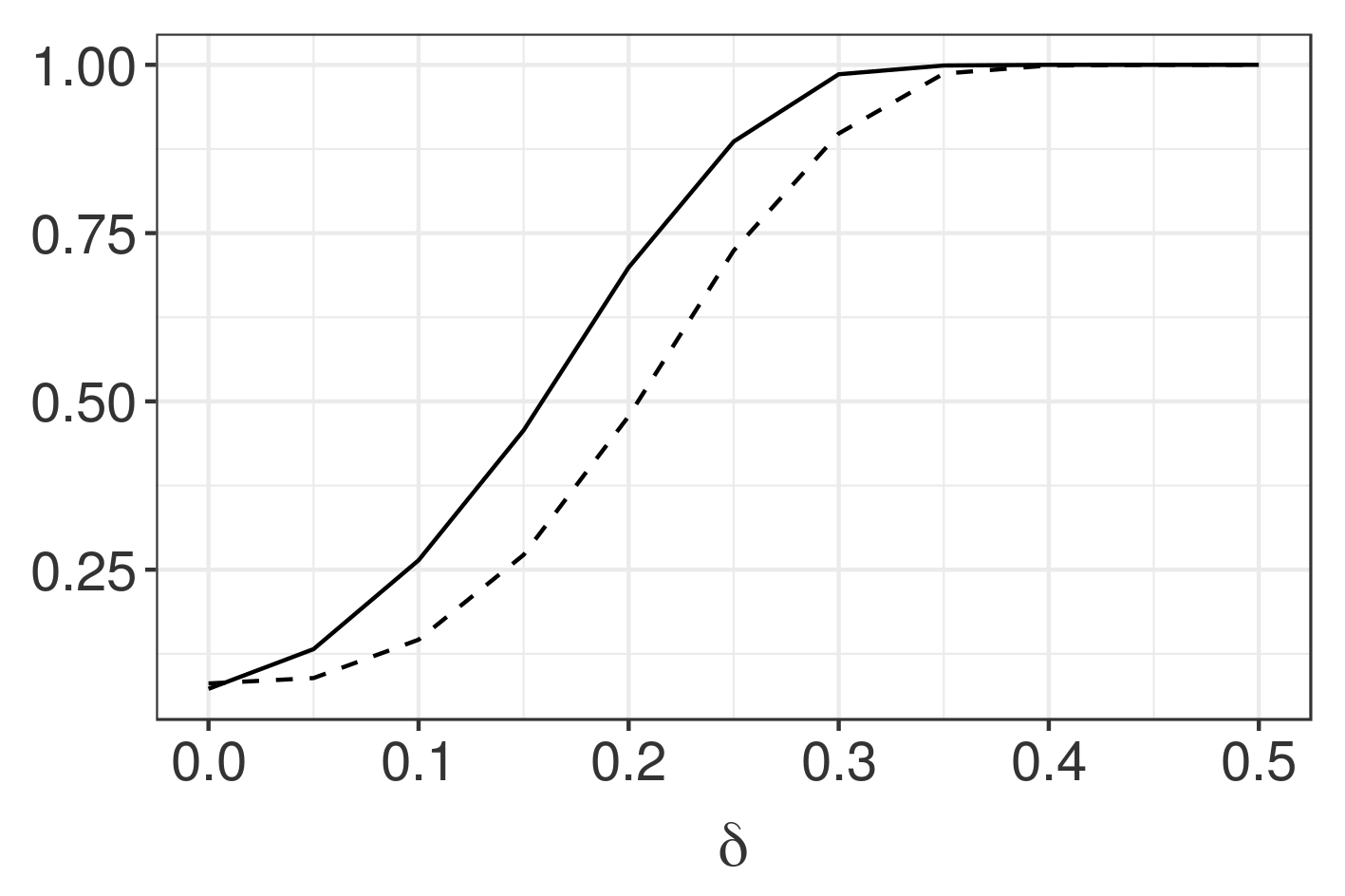} &
\includegraphics[width=6cm, height=4cm]{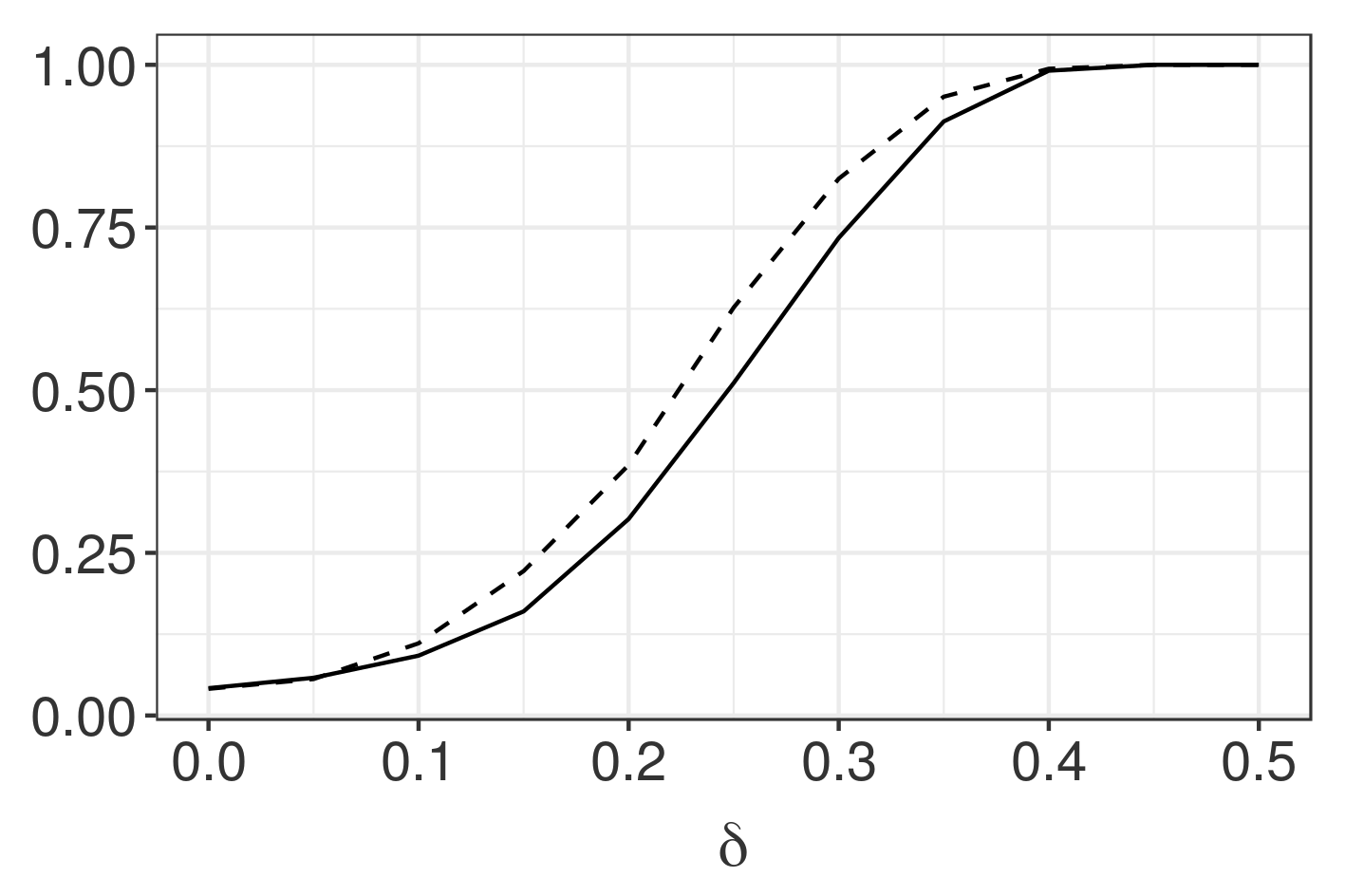}
\end{tabular}
\caption{\it
\it Simulated power of the monitoring scheme \eqref{hol1} for different size $\delta$ of the change, sample size $m=200$ and dimension $d=200$.
Left panels: Solid line (M1), dashed line (M2)
Right panels: Solid line (M3), dashed line (M4).
The nominal level is $\alpha=0.05$.
\label{fig:alt2}}
\end{figure}

\subsection{Data example}\label{sec:dataExample}
In this section, we illustrate potential applications of the new monitoring scheme 
in a data example.
For this purpose, we consider a data set sampled in hydrology, which consists of the average daily flows measured in $m^3$/sec of the river Chemnitz at G\"oritzhain in Saxony, Germany, for the years 1909-2013.

This data has been previously analyzed in the (retrospective) change point literature by \cite{Sharipov2016}, who developed methodology for detecting change points in functional data.
The data set consists of a sample of $n=105$ observations with dimension $d = 365$, such that each vector $\bm{X}_t=(X_{t,1},\dots,X_{t,365})^\top$ contains the daily average flows of one (German) hydrological year, which lasts from 1st of November to 31st of October.
For instance, the data point $X_{1,1}$ represents the daily average flow of the 1st of November 1909, while $X_{105,365}$ is the same key figure for the 31st of October 2014.
By this transformation, \cite{Sharipov2016} located a change in the annual flow curves in the year 1964.
\cite{Dette2018} propose a retrospective test for relevant changes in a high dimensional time series.
They consider the same data and locate $4$ different mean changes that exceed a test threshold of $0.63$ and are traced back to the dates 10th of July 1950, 18th of March 1956, 23rd of December 1965, 7th of February 1979, which correspond to spatial components $252$, $137$, $53$, $99$ , respectively.

Based on these prior analyses, we consider the first 35 observations as our initial stable data set and will use the remaining 70 observations as the monitoring period corresponding to a choice of $m=35$ and $T=2$.
>From the initial set $\bm{X}_1,\dots, \bm{X}_{35}$ the spatial correlation structure is estimated via the implementation of the estimator based on the quadratic spectral kernel provided in the R-package 'sandwich' [see \cite{Zeileis2004}].
By the bootstrap in Algorithm \ref{alg:complete} the quantiles are obtained as $q_{0.99}= 7.43$ and $q_{0.95}=4.93$, for which we conduct our monitoring method.
During the monitoring period, we proceed as described in Algorithm \ref{alg:complete}:
If the detection scheme rejects the null hypothesis of no change at a certain time point, we report the corresponding component(s) as unstable and remove it/them from the sample.
Afterwards, we continue monitoring with the remaining components until there is another rejection or the end of the monitoring period is attained.

The results of this procedure are displayed in Table \ref{tab:dataExample} for the test levels $\alpha=0.01$ and $\alpha=0.05$ and can be summarized as follows.
For a test level of $0.05$ more unstable components (33) are identified than for $0.01$ (17 components).
Naturally, all breaks identified by the lower test level, are also detected by the other one, while the time of detection is sometimes earlier in the latter case.
As the data exhibits (positive) spatial correlation, breaks partially occur in clusters, for example consider components 285, 286 and 287, for which both test levels detect changes or components 214, 215 and 216, for which changes are found at test level $0.05$.
Apparently, these \textit{clusters} have to be considered as one change affecting several components, for instance a flood event or a new seasonal effect.

It is worth mentioning, that our findings match three out of four unstable components identified by the threshold procedure of \cite{Dette2018}.
Namely, we refer to components 53, 99, 252, which are likewise identified to contain a break by our sequential analysis.
To illustrate the data set, we finally display the average daily flow over the years for these three components in Figure \ref{fig:dataExample}.
The plots indicate that the break in component 252 (10th of July) is most probably caused by a huge outlier in the hydrological year starting in November 1953, which leads to an immediate rejection.
This observation can be easily linked to a flood event in Saxony in summer of 1954 [see for instance \cite{Schroeter2013}].
For the components 53 (23rd of December) and 99 (7th of February), a visual inspection of the plots indicate \textit{actual} structural changes.

\begin{table}[H]
\centering
\begin{small}
\setlength{\tabcolsep}{5pt}
\begin{tabular}{cc|cc||cc|cc}
\multicolumn{4}{c||}{$\alpha=0.05$} & \multicolumn{4}{c}{$\alpha=0.01$}\\
\hline
component & year & component & year & component & year & component & year \\
\hline
101      & 1945 & 54, 209           & 1995 & 101      & 1945 & 280     & 2009\\
252, 253 & 1953 & 264               & 1996 & 252, 253 & 1953 & 54, 215 & 2012\\
249, 251 & 1954 & 285, 286, 287     & 2001 & 249      & 1957 & 209     & 2013\\
247      & 1957 & 92                & 2002 & 105      & 1960\\
105      & 1960 & 99,192            & 2003 & 189      & 1977\\
189, 191 & 1977 & 138               & 2004 & 191      & 1979\\
104      & 1979 & 280, 283          & 2009 & 100      & 1980\\
100, 190 & 1980 & 44                & 2010 & 104      & 1986\\
102      & 1986 & 55, 214, 215, 216 & 2012 & 53       & 1995\\
53, 57   & 1993 & 199               & 2013 & 285, 286, 287 & 2001
\end{tabular}
\end{small}
\caption{\it Structural breaks detected in the river flow data for a test level of $\alpha=0.05$ (left column) and $\alpha=0.01$ (right column).
The column 'year' specifies the (hydrological) year after which the rejection occurred.
For instance 1945 means that the data from the hyrodological year 01st of November 1945 to 31st of October 1946 was already under consideration.
\label{tab:dataExample}}
\end{table}
\begin{figure}
\centering
\includegraphics[width=12cm, height=4.5cm]{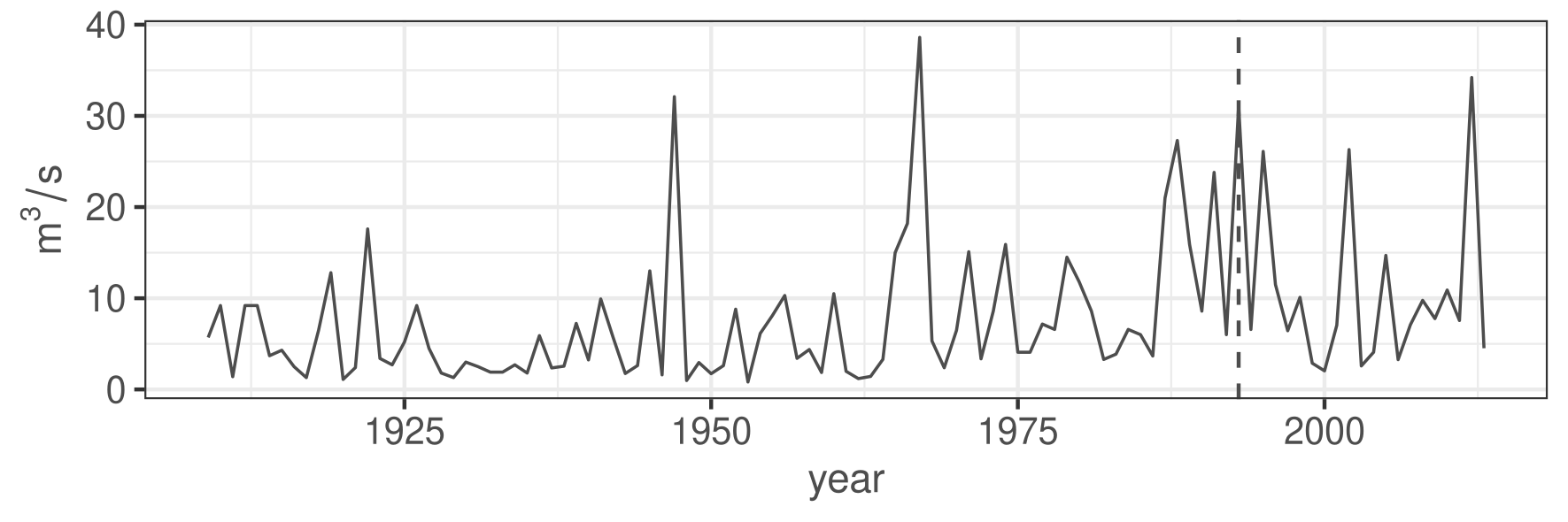}
\includegraphics[width=12cm, height=4.5cm]{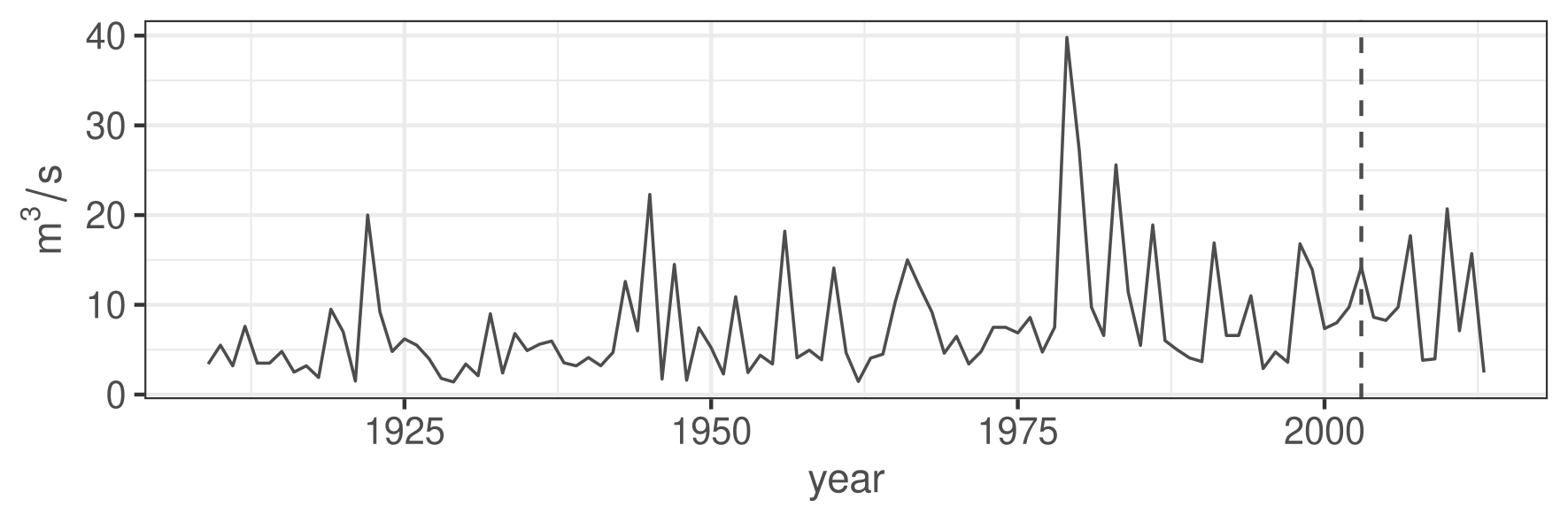}
\includegraphics[width=12cm, height=4.5cm]{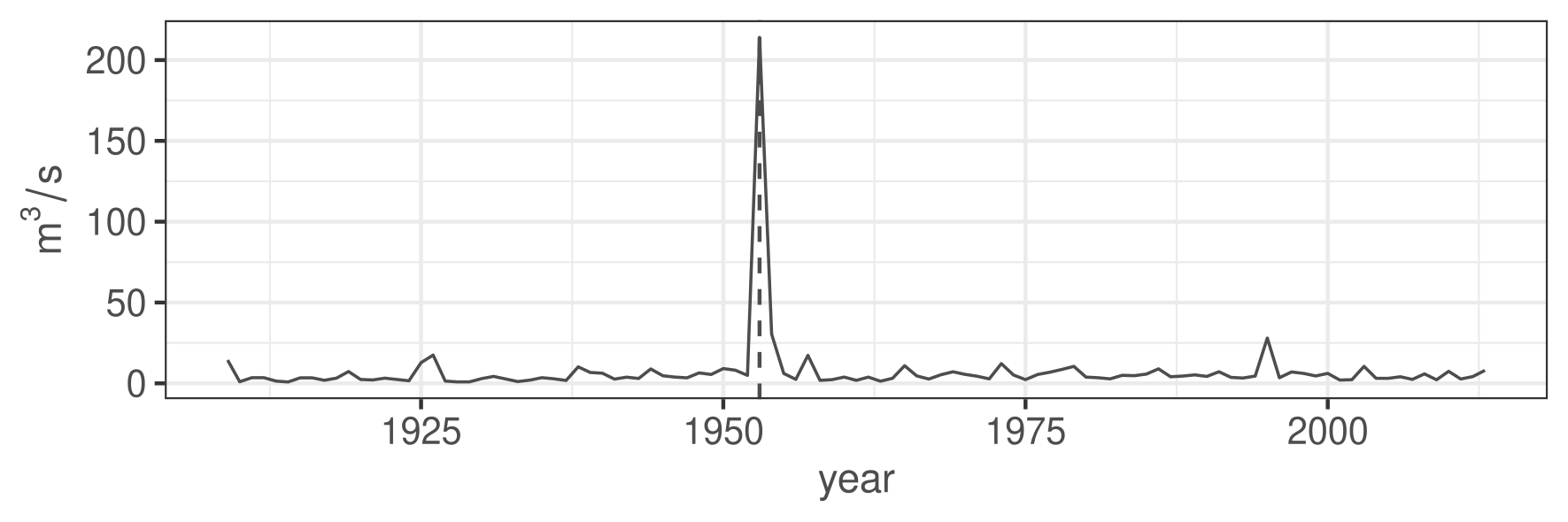}
\caption{\it Average daily flows for the dates 23rd of December (spatial component 53, upper row), 7th of February (spatial component 99, middle row) and 10th of July (spatial component 252, lower row).
Vertical dashed lines indicate the time points at which a break was detected by the sequential method with a test level of $5\%$.
\label{fig:dataExample}}
\end{figure}

\noindent \textbf{Acknowledgments}
This work has been supported in part by the Collaborative Research Center ``Statistical modeling of nonlinear dynamic processes'' (SFB 823, Teilprojekt A1, C1) and the Research Training Group 'high-dimensional phenomena in probability - fluctuations and discontinuity' (RTG 2131) of the German Research Foundation (DFG).
The authors are grateful to Florian Heinrichs for extremely helpful discussions during the preparation of this manuscript and to Andreas Schumann and Svenja Fischer from the Institute of Engineering Hydrology and Water Resources Management of the Ruhr-Universit\"at Bochum, who provided hydrological data analyzed in Section \ref{sec:dataExample}.
Finally, we would like to thank all reviewers for their constructive comments on an earlier version of the manuscript.
 
{\footnotesize
 \setlength{\bibsep}{0pt}

}
\newpage

\appendix

\section{Non-simultaneous changes among components}
\label{sec:diffchanges} 

{In this section, we present an algorithm to monitor for non-simultaneous change points in the mean vector of a high dimensional time series at test level $\alpha$.
For this purpose, we denote the set of all components without a change in the mean by
\begin{align}\label{def:StableSet}
\mathcal{S}_d := \big\{ h \in \{1,\dots,d\}\; \big| \; \mu_{1,h}=\mu_{2,h}=\dots =\mu_{m+Tm,h} \big\}~.
\end{align}
For the complement set $\mathcal{S}_d^c= \{1,\dots,d\}\setminus \mathcal{S}_d$ let $k_h^* \in \{1,\ldots , mT\}$ denote the time of change in the component $h \in \mathcal{S}_d^c$ within the monitoring period, that is
\begin{align*}
\mu_{1,h} = \dots = \mu_{m,h} = \dots = \mu_{m+k_h^*-1,h} \neq \mu_{m+k_h^*,h} = \dots =\mu_{m+Tm,h}~.
\end{align*}
Note that, in this model, there is still at most one change point per component, however we allow the change times to be different among components.}
For what follows, recall the construction of the bootstrap statistic from Section \ref{sec32}.
The algorithm below is capable of identifying the sets $\mathcal{S}_d$ and $\mathcal{S}_d^c$, see Theorem \ref{thm:algorithmConsistent} below.

\begin{algorithm}\label{alg:complete}\phantom{a} ~
\begin{enumerate}[label = Step \arabic*:, leftmargin=1.4cm]
\item Either choose the quantile $q_{1-\alpha}$ using the approximation by the Gumbel distribution, that is $q_{1-\alpha}=g_{1-\alpha}$, where $g_{1-\alpha}$ is the $(1-\alpha)$-quantile of the Gumbel distribution or, alternatively, obtain the quantile from the bootstrap as follows:
\begin{enumerate}[label = Step 1.\arabic*:, leftmargin=1.8cm]
\item Compute the long-run correlation estimates $\big(\hat{\rho}_{i,j}\big)_{i,j=1}^d$ from the initial set $\bm{X}_1,\dots,\bm{X}_m$.\vspace{0.2cm}
\item 
Based on these estimates, generate $N$ independent realizations of the Gaussian vectors $(\hatZ_{t,1}, \ldots , \hatZ_{t,d})^\top $ with covariance structure $\big(\hat{\rho}_{i,j}\big)_{i,j=1}^d$ for $t=1, \ldots, m +Tm$ and compute the corresponding bootstrap statistics 
$$a_d\big(\hatTmdZ(1) - b_d\big),a_d\big(\hatTmdZ(2)-b_d\big),\dots,a_d\big(\hatTmdZ(N)-b_d\big)$$
defined in \eqref{def:BSstatistic}.
\item Compute $q_{1-\alpha}$ as the empirical $(1-\alpha)$-quantile of the sample
$$\Big\{a_d\Big(\hatTmdZ(n)-b_d\Big)\Big\}_{n=1,\dots,N}~.
$$
\end{enumerate}
\item Initialize $\hatSdalpha := \{1,\dots,d\}$ and set $k=1$. 

Monitoring: While $k\le Tm$ compute the statistics $\hatE_{m,h}(k)$.
If the inequality
\begin{equation*}
\max_{h \in \hatSdalpha} w(k/m) \hatE_{m,h}(k) > \dfrac{q_{1-\alpha}}{a_d} + b_d
\end{equation*}
holds, reject the null hypothesis in favor of the alternative.
Eliminate the components that led to the rejection, i.e. 
\begin{align}\label{eq:compWithChange}
\hatSdalpha \longleftarrow \hatSdalpha \setminus \Big\{ h \in \hatSdalpha \;\big|\; w(k/m) \hatE_{m,h}(k) > \dfrac{q_{1-\alpha}}{a_d} + b_d\, \Big\}
\end{align}
and continue monitoring for $k\longleftarrow k+1$ with the remaining components in $\hatSdalpha$. 
\item
If there was no rejection during monitoring, decide for the null hypothesis of no change in the mean vector.
In case of rejections, decide for the alternative of a change in at least one component.
Then, it holds that
\begin{align}\label{eq:hatSdalpha}
\hatSdalpha = \Big\{ h \in \{1,\dots,d\} \;\;\Big|\;\; \maxkTm w(k/m) \hatE_{m,h}(k) \leq \dfrac{q_{1-\alpha}}{a_d} + b_d \Big\}
\end{align}
and the components remaining in this set are assumed as mean stable.
\end{enumerate}
\end{algorithm}

\noindent The following theorem states, that Algorithm \ref{alg:complete} is able to separate the sets $\mathcal{S}_d$ and $\mathcal{S}_d^c$ correctly.

\begin{theorem}\label{thm:algorithmConsistent}
Let Assumptions \ref{assump:model}-\ref{assump:spatialDependence} and \ref{assump:Bootstrap} be satisfied and assume that
Algorithm \ref{alg:complete} was launched either with the Gumbel or with the bootstrap quantile.
The set $\hatSdalpha$ defined in \eqref{eq:hatSdalpha} satisfies
\begin{align}\label{ineq:SdSubset}
\limsup_{m,d\to \infty}\, \Pb\Big( \mathcal{S}_d \subset \hatSdalpha \Big) \geq 1 - \alpha~.
\end{align}
If further 
\begin{align}\label{assump:uniformConsistency}
\sqrt{\dfrac{m}{\log m}} \, \min_{h \in \mathcal{S}_d^c} | \mu_{m+k_h^*-1,h} - \mu_{m+k_h^*,h}| \to \infty
\;\;\;
\text{and}
\;\;\;
\limsup_{m,d\to \infty} \max_{h \in \mathcal{S}_d^c} \dfrac{k_h^*}{m} < T\,,
\end{align}
then
\begin{align}\label{eq:SdC}
\limmd \Pb\Big( \mathcal{S}_d \supset \hatSdalpha \Big) = 1~.
\end{align}
\end{theorem}
\noindent By Theorem \ref{thm:algorithmConsistent}, the set $\hatSdalpha$ contains all the components without a change in the mean with high probability.
It may also contain components in which the change in the mean is so small that it cannot be reliably detected.
However, if we assume that that all changes in the mean are large enough and not too close to the monitoring end [see \eqref{assump:uniformConsistency}], then $\hatSdalpha= \mathcal{S}_d$ with high probability.
That is, one detects all components exhibiting a change point without causing a false alarm.

\section{Proofs of main results}\label{sec:technicalDetails}

Let us first introduce some necessary notation, which will be frequently used throughout this section.
The symbol $\lesssim$ denotes an inequality up to a constant, which does not depend on size $m$ of the training sample and the dimension $d$.
For two sequences $e_n$ and $f_n$, we write $e_n \sim f_n$, whenever $f_n/e_n \to 1$ as $n \to \infty$.

Let $q\in (0,1]$. By $\FM$ we denote the distribution function of the range of a (standard) Brownian motion, that is $\M = \max_{t \in [0,q]} W(t) - \min_{t \in [0,q]} W(t)$, which can be found in \cite{Borodin1996}, page 146, and is given by
\begin{align*}
\FM(x) 
=
\begin{cases}
1 + 4 \sum\limits_{k=1}^\infty (-1)^k k \Erfc \Big( \dfrac{kx}{\sqrt{2q}}\Big)\;\;&\text{if}\;\; x > 0~,\\
0 &\text{otherwise}~,
\end{cases}
\end{align*}
where $\Erfc=1-\Erf$ denotes the complementary error function.
Using the elementary property $\Erf(x) = 2\Phi(x\sqrt{2})-1$ we obtain
\begin{align}\label{eq:repF}
\FM(x)
= 1 + 8\sum_{k=1}^\infty (-1)^k k \Phi\Big(-\frac{kx}{\sqrt{q}}\Big)\;\;\text{for}\;x>0\,,
\end{align}
where throughout this paper $\Phi$ denotes the cumulative distribution function of a standard normal distribution.

\subsection{Some preliminary results} \label{seca1} 

We will begin with an auxiliary result.
In Lemma \ref{lem:gumbel} we investigate the weak convergence of the maximum of independent identically distributed random variables with the same distribution as the random variable $\mathbb{M}$ defined in \eqref{conv:1dim}.

\begin{lemma}\label{lem:gumbel}
Let $q\in (0,1]$ and $M_1',M_2',\ldots$ be independent identically distributed random variables with
$$M_1'\eqd \max_{0 \leq t \leq q} W(t) - \min_{0 \leq t \leq q} W(t)~,$$
where $W$ denotes a standard Brownian motion.
Then, as $d \to \infty$, it holds that
\begin{align*}
a_d \Big ( \maxhd M_h' - b_d \Big ) \convd G\,,
\end{align*}
where $G$ is a standard Gumbel distributed random variable and $a_d$ and $b_d$ are given by
\begin{equation}
a_d=\sqrt{\frac{2 \log d}{q}} \quad \text{ and } \quad b_d=\sqrt{2q \log d}-\frac{\sqrt{q} \,(\log \log d - \log \tfrac{16}{\pi})}{2\sqrt{2 \log d}}\,.
\end{equation}
\end{lemma}
\begin{proof}[Proof of Lemma \ref{lem:gumbel}]
The desired extreme value convergence will be derived from the distribution function of $M_1'$, which is given by $\FM$ in \eqref{eq:repF}.
Observe, that $\FM$ is twice differentiable with derivatives (for $x > 0$)
\begin{align*}
\FM'(x) 
=\dfrac{4\sqrt{2}}{\sqrt{\pi q}}\sum_{k=1}^\infty (-1)^{k+1} k^2 \exp \Big(- \dfrac{k^2x^2}{2q}\Big)~,\\
\FM''(x)
= \dfrac{4\sqrt{2}}{\sqrt{\pi q}q}\sum_{k=1}^\infty (-1)^{k} k^4 x \exp \Big(- \dfrac{k^2x^2}{2q}\Big)~,
\end{align*}
where we used that the series converge uniformly on all intervals $[\varepsilon, \infty)$ for $\varepsilon>0$ and therefore term by term differentiation is allowed.
Thus, by Theorem 1.1.8 from \cite{DeHaan2006} the distribution function $\FM$ is in the domain of attraction of the Gumbel distribution if 
\begin{align}\label{eq.gumbel}
\lim_{x\to \infty} \dfrac{(1-\FM(x))\FM''(x)}{(\FM'(x))^2} = -1\,.
\end{align}
Next, we prove \eqref{eq.gumbel}.
To this end, note that by the definition of the complimentary error function, we obtain for $x>0$
\begin{align}\label{eq:firstGumbelres}
\begin{split}
x\big(1-\FM(x)\big)\exp \Big(\dfrac{x^2}{2q}\Big)
&= 4 \sum\limits_{k=1}^\infty (-1)^{k+1} k x\Erfc \Big( \dfrac{kx}{\sqrt{2q}}\Big)\exp\Big(\dfrac{x^2}{2q}\Big)\\
&= A_1(x) + A_2(x)\,,
\end{split}
\end{align}
where the two summands on the right-hand side are given by
\begin{align*}
A_1(x) &= \dfrac{8}{\sqrt{\pi}}\int_{\frac{x}{\sqrt{2q}}}^\infty \exp(-\tau^2)d\tau  \cdot x \exp \Big(\dfrac{x^2}{2q}\Big)~,\\
A_2(x) &= \dfrac{8}{\sqrt{\pi}} \sum\limits_{k=2}^\infty (-1)^{k+1} k x\int_{\frac{kx}{\sqrt{2q}}}^\infty \exp(-\tau^2)d\tau  \cdot \exp\Big(\dfrac{x^2}{2q}\Big)~.
\end{align*}
We will treat the two summands of the last display separately.
For the first summand we obtain that
\begin{align*}
\limx A_1(x)
=& \dfrac{8}{\sqrt{\pi}} \limx \frac{\int_{\frac{x}{\sqrt{2q}}}^\infty \exp(-\tau^2)d\tau }{x^{-1} \exp \Big(\dfrac{-x^2}{2q}\Big)}\\
=& \dfrac{8}{\sqrt{\pi}} \limx \frac{\dfrac{-1}{\sqrt{2q}} \exp \Big(\dfrac{-x^2}{2q}\Big)}{-x^{-2} \exp \Big(\dfrac{-x^2}{2q}\Big)-\frac{1}{q}\exp \Big(\dfrac{-x^2}{2q}\Big)}
=4\dfrac{\sqrt{2q}}{\sqrt{\pi}}
\end{align*}
by L'H\^{o}spital's rule.
For the second summand of the last display in \eqref{eq:firstGumbelres} note that
\begin{align*}
\Big| \dfrac{\sqrt{\pi}}{8}A_2(x) \Big|
&\leq \sum\limits_{k=2}^\infty k x\int_{\frac{kx}{\sqrt{2q}}}^\infty \exp(-\tau^2)d\tau 
\cdot \exp\Big(\dfrac{x^2}{2q}\Big)\\
&\leq \sqrt{2q}\sum\limits_{k=2}^\infty \int_{\frac{kx}{\sqrt{2q}}}^\infty \tau \exp(-\tau^2)d\tau
\cdot \exp\Big(\dfrac{x^2}{2q}\Big)\\
&= \sqrt{\dfrac{q}{2}} \exp\Big(\dfrac{-x^2}{2q}\Big)\sum\limits_{k=2}^\infty  \exp\Big(-\frac{(k^2-2)x^2}{2q}\Big)\\ 
& \overset{(x\geq 1)}{\leq} \sqrt{\dfrac{q}{2}} \exp\bigg(\dfrac{-x^2}{2q}\bigg)\sum\limits_{k=2}^\infty  \exp\bigg(-\frac{k^2-2}{2q}\bigg)
= o(1)\;\;\text{as}\;\;x \to \infty~.
\end{align*}
Combining the last statements with the decomposition in \eqref{eq:firstGumbelres} yields
\begin{align}\label{eq:F0}
\limx x\big(1-\FM(x)\big)\exp \Big(\dfrac{x^2}{2q}\Big)
= 4\dfrac{\sqrt{2q}}{\sqrt{\pi}}~.
\end{align}
For the denominator of \eqref{eq.gumbel} note that
\begin{align}\label{eq:F1}
\begin{split}
\limx \FM'(x)\exp \Big(\dfrac{x^2}{2q}\Big)
&= \dfrac{4\sqrt{2}}{\sqrt{\pi q}}
+\limx \dfrac{4\sqrt{2}}{\sqrt{\pi q}}\sum_{k=2}^\infty (-1)^{k+1} k^2 \exp \Big(- \dfrac{(k^2-1)x^2}{2q}\Big)\\
&= \dfrac{4\sqrt{2}}{\sqrt{\pi q}}~,
\end{split}
\end{align}
where we used that for $x\geq 1$
\begin{align*}
\sum_{k=2}^\infty k^2 \exp \Big(- \dfrac{(k^2-1)x^2}{2q}\Big)
&\leq \exp\Big(\dfrac{-x^2}{2q}\Big)\sum_{k=2}^\infty k^2 \exp \Big(- \dfrac{(k^2-2)}{2q}\Big)
=o(1)\;\;\text{as}\;\;x \to \infty~.
\end{align*}
Using similar arguments we obtain
\begin{align}\label{eq:F2}
\begin{split}
\limx &x^{-1}\FM''(x)\exp \Big(\dfrac{x^2}{2q}\Big)\\
=&\dfrac{-4\sqrt{2}}{\sqrt{\pi q}q}
+\limx \dfrac{4\sqrt{2}}{\sqrt{\pi q}q}\sum_{k=2}^\infty (-1)^{k} k^4 \exp \Big(- \dfrac{x^2(k^2-1)}{2q}\Big)
=\dfrac{-4\sqrt{2}}{\sqrt{\pi q}q}~.
\end{split}
\end{align}
Combining \eqref{eq:F0}, \eqref{eq:F1} and \eqref{eq:F2}, it follows that 
\begin{align*}
&\limx \dfrac{(1-\FM(x))\FM''(x)}{(\FM'(x))^2}= \frac{\lim\limits_{x\to \infty} x(1-\FM(x))\exp \Big(\dfrac{x^2}{2q}\Big) \lim\limits_{x\to \infty} x^{-1}\FM''(x)\exp \Big(\dfrac{x^2}{2q}\Big)}{\lim\limits_{x\to \infty} \Big (\FM'(x)\exp \Big(\dfrac{x^2}{2q}\Big)\Big  )^2}
=-1,
\end{align*}
which completes the proof of \eqref{eq.gumbel}.
By definition of the maximum domain of attraction, \eqref{eq.gumbel} is equivalent to the existence of sequences $a_d>0, b_d$ such that 
\begin{align*}
a_d \Big( \maxhd M_h' - b_d \Big) \convd G\;\; \text{as}\;\; d \to \infty~.
\end{align*}
Our next goal is to find an explicit formula for these sequences, for which we will employ Proposition 3.3.28 in \cite{Embrechts1997}.
Therefore consider the distribution function
$$ H(x):= 1 - 4 \frac{\sqrt{2 q}}{\sqrt{\pi}} x^{-1} \exp\big(-x^2/(2q\big)\,,\qquad x>x_0~.$$
Note that we can rewrite equation \eqref{eq:F0} as
\begin{equation}
\lim_{x \to \infty} (1-\FM(x))\Big(4 \frac{\sqrt{2 q}}{\sqrt{\pi}} x^{-1} \exp\big(-x^2/(2q\big) \Big)^{-1}= 1\,
\end{equation}
and thereby obtain the tail-equivalence
$$\lim_{x \to \infty} \frac{1-\FM(x)}{1-H(x)}=1\,.$$
Thus, by Proposition 3.3.28 in \cite{Embrechts1997} we have to find sequences $a_d, b_d$ such that $\lim_{d \to \infty} H^d(a_d^{-1} x+b_d)=F_G(x)$ for all $x\in \R$.
Since $H$ has Weibull-like tails, a possible choice of the norming sequences is given in Table 3.4.4 in \cite{Embrechts1997}:
\begin{align*}
b_d&=\sqrt{2q \log d}+\frac{1}{2\sqrt{2q \log d}} \Big(-q\log\big(2q\log d\big)+2q \log \sqrt{\frac{32 q}{\pi}} \Big)\\
a_d^{-1}&= q (2q \log d )^{-1/2}\,.
\end{align*}
After simplification we get the desired result
\begin{equation}
a_d=\sqrt{\frac{2 \log d}{q}} \quad \text{ and } \quad b_d=\sqrt{2q \log d}-\frac{\sqrt{q} \,(\log \log d - \log \tfrac{16}{\pi})}{2\sqrt{2 \log d}}\,,
\end{equation}
which finishes the proof of Lemma \ref{lem:gumbel}.
\end{proof}

\subsection{Proof of Theorem \ref{thm:gumbelWiener.dep}} \label{seca2} 
Let $q\in (0,1]$ and recall the definition of $M_h$ in Theorem \ref{thm:gumbelWiener.dep}.
By Theorem 1 from \cite{Arratia1989} in the form as presented in Lemma A.4 in \cite{jiang2004} we obtain for any
$x\in\mathbb{R}$ the inequality 
\begin{align}\label{eq:b123}
\left|\Pb\left(\maxhd M_h\leq u_d(x)\right)-\exp(-\lambda)\right|
\leq \big(1\wedge\lambda^{-1}) (\Lambda_1 + \Lambda_2 + \Lambda_3\big)~,
\end{align}
with
$\lambda=\sum_{h=1}^d\Pb\big(M_h>u_d(x)\big)$ and
\begin{align*}
\Lambda_1 &= \sum_{i=1}^d \sum_{\substack{1\leq j \leq d\\ |i-j|\leq L_d}}\Pb\big(M_i>u_d(x)\big)\Pb\big(M_j>u_d(x)\big)~,\\*
\Lambda_2 &= \sum_{i=1}^d \sum_{\substack{1\leq j \leq d\\ j \neq i,\, |i-j|\leq L_d}}\Pb\big(M_i>u_d(x),M_j>u_d(x)\big)~,\\*
\Lambda_3 &= \sum_{i=1}^d \E\left| \Pb\Big(M_i > u_d(x)\,|\,\sigma\big( M_j: |i-j|>L_d \big) \Big) - \Pb(M_i > u_d(x)\big)\right|~,
\end{align*}
where $u_d(x)=x/a_d+b_d$ and $\sigma\big( M_j: |i-j|>L_d \big)$ denotes the $\sigma$-algebra generated by the set $\big\{ M_j: |i-j| > L_d \big\}$.
In the remainder of the proof we fix $x \in \R$ and due to $a_d,\, b_d \to \infty$ we can assume that $d$ is sufficiently large such that $u_d(x) >0$.
Further let $M_1',M_2',\ldots$ be i.i.d.~random variables with $M_1'\eqd M_1$.
With Lemma \ref{lem:gumbel} we have 
\begin{align}\label{eq:maxiid}
\limd \Pb\Big(\maxhd M_h'\leq u_d(x)\Big)
= \limd\Big(\Pb\big(M_1\leq u_d(x)\big)\Big)^d
= \exp\big(-\exp(-x)\big)~.
\end{align}
As $b_d\rightarrow \infty$ and $\lim\limits_{x \rightarrow 0}\frac{x}{\log(1-x)}=-1$, \eqref{eq:maxiid} yields
\begin{align}\label{eq:lam1}
\begin{split}
\lambda
= &\sum_{h=1}^d\Pb\big(M_h > u_d(x)\big)
= d\Pb\big(M_1>u_d(x)\big)\\
= &-d\log\left(1-\Pb\big(M_1>u_d(x)\big)\right)\big(1+o(1)\big)\\
= &-\left(\log\left(\Pb(M_1\leq u_d(x))^d\right)\right)\big(1+o(1)\big)
\rightarrow\exp(-x)\quad\mbox{as }d\rightarrow\infty~.
\end{split}
\end{align}
To treat $\Lambda_1$, note that \eqref{eq:lam1} yields for $d\to\infty$
\begin{align}\label{sim:tail}
1 - \FM(u_d(x)) \sim \dfrac{\exp(-x)}{d}~.
\end{align}
Since $L_d=o(d)$ by assumption, we obtain 
\begin{align}\label{eq:b1}
\Lambda_1
= 2dL_d\big(1-\FM(u_d(x))\big)^2
&\sim 2 \dfrac{L_d}{d} \exp(-2x)=o(1)~.
\end{align}
To derive the asymptotic properties of $\Lambda_2$, observe the bound
\begin{align}\label{eq:Lambda2bound}
\begin{split}
\Lambda_2 
&\leq  2d L_d \max_{\substack{1\leq i<j\leq d\\ \rho_{i,j}=0,\,|i-j|\leq L_d,}} \Pb \big(M_i>u_d(x),M_j>u_d(x)\big)\\
&\quad + 
2d L_d \max_{\substack{1\leq i<j\leq d\\ \rho_{i,j}\neq 0,\,|i-j|\leq L_d,}} \Pb \big(M_i>u_d(x),M_j>u_d(x)\big),
\end{split}
\end{align}
where throughout this proof we write $\rho_{i,j}$ instead of $\rho_{i,j}^{(d)}$ for simplicity. 
Note, that in case of $\rho_{i,j}=0$ it holds by \eqref{sim:tail}, as $d \to \infty$~,
\begin{align*}
\Pb \big(M_i>u_d(x),M_j>u_d(x)\big)&= \big(\Pb \big(M_1 > u_d(x))\big)^2\sim \left(\frac{\exp(-x)}{d} \right)^2~,
\end{align*}
which implies that the first summand on the right-hand side in \eqref{eq:Lambda2bound} vanishes.
To treat the other one, we follow the idea in Lemma B.11 from \cite{Jirak2015supp} and use that a comparison of the covariance structures of the two Gaussian processes yields
\begin{align*}
\Big\{\big(W_i(t),W_j(t)\big)\Big\}_{t \geq 0}
\eqd \Big\{\left(W_i(t),\sqrt{1-\rho_{i,j}^2}W_j'(t) + \rho_{i,j}W_i(t)\right)\Big\}_{t \geq 0}~,
\end{align*}
where $W_j'$ is a standard Wiener process that is independent of $W_i$.
Consequently, it also holds that
\begin{align*}
\left(M_i,\, M_j\right)
\eqd \left(M_i, \sup_{0 \leq t \leq q} \sup_{0 \leq s \leq t}
\left|\sqrt{1-\rho_{i,j}^2}\big(W_j'(t) - W_j'(s)\big) + \rho_{i,j}\big(W_i(t)-W_i(s)\big)\right|\right)
\end{align*}
and by the triangle inequality
\begin{align*}
&\sup_{0 \leq t \leq q} \sup_{0 \leq s \leq t}
\left|\sqrt{1-\rho_{i,j}^2} \big( W_j'(t)-W_j'(s) \big) + \rho_{i,j}\big(W_i(t)-W_i(s)\big)\right|\\
&\leq \sqrt{1-\rho_{i,j}^2} \sup_{0 \leq t \leq q} \sup_{0 \leq s \leq t}
\left|W_j'(t)-W_j'(s)\right|+|\rho_{i,j}|\sup_{0 \leq t \leq q} \sup_{0 \leq s \leq t}\left|W_i(t)-W_i(s)\right|\\
&\eqd\sqrt{1-\rho_{i,j}^2} M_j'+|\rho_{i,j}|M_i~,
\end{align*}
where $M_j'$ has the same distribution as $M_j$ but is independent of $M_i$.
Now, we obtain for $\delta \in (0,1)$ and $\vep=\delta \big(\tfrac{1}{|\rho_{i,j}|}-1\big)$,
\begin{align*}
&\Pb \big(M_i>u_d(x),M_j>u_d(x)\big)
\leq \Pb\left(M_i>u_d(x),\sqrt{1-\rho_{i,j}^2}M_j'+|\rho_{i,j}|M_i>u_d(x)\right)\\
&= \int_{u_d(x)}^{\infty} \Pb\left(M_j'\geq \frac{u_d(x)-y|\rho_{i,j}|}{\sqrt{1-\rho_{i,j}^2}}\right)\Pb_{M_i}(dy)\\
&\le  \int_{u_d(x)}^{u_d(x)\big(\frac{1}{|\rho_{i,j}|}-\vep\big)}  \Pb\left(M_j'\geq \frac{u_d(x)-y|\rho_{i,j}|}{\sqrt{1-\rho_{i,j}^2}}\right)\Pb_{M_i}(dy)
+ \Pb \Big(M_i > u_d(x)\big(\tfrac{1}{|\rho_{i,j}|}-\vep \big) \Big)\\
&:= P_{ij}^{(1)} + P_{ij}^{(2)}\,.
\end{align*}
First, we bound $P_{ij}^{(2)}$. To this end, recall from \eqref{eq:F0} that
\begin{equation}\label{eq:A.4} 
\Pb(M_1>x) \sim K x^{-1} \exp(-x^2/(2q)) \,, \qquad x \to \infty~,
\end{equation}
where $K=4 \frac{\sqrt{2 q}}{\sqrt{\pi}}$.
By $|\rho_{i,j}|\le \rho_+<1$ we obtain
\begin{align}\label{eq:P(2)}
\begin{split}
P_{ij}^{(2)} &= \Pb \Big(M_i > u_d(x)\big(\tfrac{1}{|\rho_{i,j}|}(1-\delta) +\delta \big) \Big)
\leq \Pb \Big(M_i > u_d(x)\big(\tfrac{1}{\rho_+}(1-\delta) +\delta \big) \Big) \\
&\sim K \frac{1}{u_d(x)\big(\tfrac{1}{\rho_+}(1-\delta) +\delta \big)} \exp\Big( -\frac{ u_d(x)^2}{2q} \big(\tfrac{1}{\rho_+}(1-\delta) +\delta \big)^2\Big)\\
&\leq \frac 12 \exp\Big( -\frac{ u_d(x)^2}{2q} \big(\rho_+^{-1} (1-\delta)+\delta \big)^2\Big)\,,
\end{split}
\end{align}
where the last inequality holds for sufficiently large $d$ depending on $\rho_+$ since $u_d(x) \to \infty$.
As regards $P_{ij}^{(1)}$, we have
\begin{align*}
P_{ij}^{(1)}
&\leq \Pb\big(M_1 > u_d(x)\big) \, \Pb\Big( M_1 > u_d(x) \tfrac{\vep |\rho_{i,j}|}{\sqrt{1-\rho_{i,j}^2}} \Big)\\
&= \Pb\big(M_1 > u_d(x)\big) \, \Pb\Big( M_1 > u_d(x)\delta \tfrac{1 - |\rho_{i,j}|}{\sqrt{1-\rho_{i,j}^2}} \Big) 
\leq \Pb\big(M_1 > u_d(x)\big) \, \Pb\Big( M_1 > u_d(x)\delta \tfrac{1 - \rho_+}{\sqrt{1-\rho_+^2}} \Big) ~,
\end{align*}
where we used that the map $x \mapsto (1-x)/\sqrt{1-x^2}$ is decreasing on $[0,1)$.
Employing again \eqref{eq:A.4}, we conclude for sufficiently large $d$
\begin{align}\label{eq:P(1)}
\begin{split}
\Pb\big(M_1 > u_d(x)\big) \, \Pb\Big( M_1 > u_d(x)\delta \tfrac{1 - \rho_+}{\sqrt{1-\rho_+^2}} \Big)
&\sim K^2 \frac{\sqrt{1-\rho_+^2}}{u_d(x)^2\delta(1-\rho_+)} \exp\left( -\frac{u_d(x)^2}{2q} \left(1+ \tfrac{\delta^2(1-\rho_{+})^2}{1-\rho_+^2}\right) \right)\\
&\leq \frac 12 \exp\left( -\frac{u_d(x)^2}{2q} \left(1+ \delta^2 (1 - \rho_{+})^2\right)  \right)\,.
\end{split}
\end{align}
Combining \eqref{eq:P(2)} and \eqref{eq:P(1)} we get 
\begin{align*}
\Pb \big(M_i>u_d(x),M_j>u_d(x)\big)&\le \exp\Big( -\frac{ u_d(x)^2}{2q} \min(f(\delta), g(\delta))\Big)\,,
\end{align*}
where the functions $f,g$ are defined by $f(\delta)=\big(\rho_+^{-1} (1-\delta)+\delta \big)^2$ and $g(\delta)=1+ \delta^2(1- \rho_{+})^2$.
Next, we will optimize this bound in $\delta \in (0,1)$.
Observe that $f(0)>g(0)$ and $f(1)<g(1)$. Since $f$ is decreasing on $[0,1]$ while $g$ is increasing on $[0,1]$, we deduce by continuity of $f$ and $g$ that there exists a unique $\delta^{\star}\in (0,1)$ such that $f( \delta^{\star})=g(\delta^{\star})$. Solving this equation, we find that 
\begin{equation}\label{eq:deltastar}
\delta^{\star} = \frac{1-\sqrt{2 \rho_+^2-\rho_+^4}}{1-\rho_+-\rho_+^2+\rho_+^3} \qquad \text{ and } \qquad g(\delta^{\star})
=1+\frac{\left(1- \sqrt{2\rho_+^2 - \rho_+^4}\right)^2}{ (1-\rho_+^2)^2}\,.
\end{equation}
Thus, it follows that
\begin{align}\label{eq:joint1}
2d L_d \max_{\substack{1\leq i<j\leq d\\ \rho_{i,j}\neq 0\,, |i-j|\leq L_d}}
\Pb \big(M_i > u_d(x), M_j > u_d(x)\big)
\leq 2d L_d \exp\Big( -\frac{ u_d(x)^2}{2q} g(\delta^{\star})\Big).
\end{align}
Using the definition of $u_d(x)$ we deduce for any $\Delta < g(\delta^{\star})-1$, 
\begin{align}\label{eq:joint2}
d L_d \exp\Big( -\frac{ u_d(x)^2}{2q} g(\delta^{\star})\Big)
&\le d L_d \exp\big( - (\Delta+1) \log d \big)= d^{-\Delta} L_d,
\end{align}
provided $d$ is sufficiently large.
By the assumption in \eqref{assump:Ld} there exists a $\Delta\in (0,g(\delta^{\star})-1)$ such that $d^{-\Delta} L_d\to 0$ as $d \to \infty$, which together with \eqref{eq:Lambda2bound} establishes $\Lambda_2\to 0$. 

Due to $\rho_{i,j} = 0$ for $|i-j| > L_d$ we obtain that the Gaussian processes $W_i$ and $W_j$ are already independent whenever $|i-j| > L_d$ [see for instance \cite{Billingsley1999}] and therefore we have that $\Lambda_3 = 0$.
The assertion now follows by combining this fact with \eqref{eq:b123}, \eqref{eq:lam1}, \eqref{eq:b1} and $\Lambda_2\to 0$.

\subsection{Proof of Theorem \ref{thm:mainGumbel}} \label{seca3}
 
{\it Throughout this proof, we will work with the sequences $a_d, b_d$ defined in \eqref{eq:adbd} with $q = T/(T+1)=q(T)$.}
Recall that the detector \eqref{hol1} is based on differences of component-wise mean estimators
\begin{align*}
\hatmu_{m+j+1}^{m+k}(h) - \hatmu_1^{m+j}(h)
\end{align*}
and   we may without loss of generality assume $\E[X_{t,h}]=0$ throughout the proof.
First, we introduce some necessary notations.
Analogously to Theorem \ref{thm:gumbelWiener.dep} let $ \big(W_1,\dots,W_d\big)^\top$ denote a $d$-dimensional Brownian motion on the interval $[0,\qT]$ with correlations: 
\begin{align}\label{eq:longruncorrW}
\Corr\big(W_h(t),W_i(t)\big) = \tilderho_{h,i}:= \rho_{h,i}\cdot I\{ |h-i| \leq L_d\}~,
\end{align}
where $\rho_{h,i}$ denotes the long-run correlation defined in \eqref{eq:longruncorr}
and $L_d := d^\Delta\to \infty$ is a sequence, where 
\begin{align*}
\Delta<\left(\frac{1- \rho_+ \sqrt{2-\rho_+^2}}{ 1-\rho_+^2} \right)^2 
\end{align*}
and $\rho_+$ is the constant from Assumption \ref{assump:spatialDependence}. Note that $\tilderho_{h,i}$ depends on $d$ only through the indicator $I\{ |h-i| \leq L_d\}$. Denote again by
\begin{align}\label{def:Mh}
M_h
:= \max_{t \in [0,\qT]} W_h(t) - \min_{t \in [0,\qT]} W_h(t)
= \max_{t \in [0,\qT]} \max_{s \in [0,t]} \big|W_h(s) - W_h(t)\big|
\end{align}
the range of $W_h$.
For $0< c < T$ define additionally the truncated version
\begin{align}\label{def:Mhc}
M_h(c)
= \max_{t \in [q(c),\qT]} \max_{s \in [0,\,q(q^{-1}(t)-c)]} \big|W_h(s) - W_h(t)\big|\,,
\end{align}
where $q(x)=x/(x+1)$, $q^{-1}(x) = x/(1-x)$ and consider the overall maxima of these quantities,
\begin{align}\label{def:wd}
\mathcal{W}_d &= \maxhd M_h\quad \text{ and } \quad \mathcal{W}_d(c) = \maxhd M_h(c)~.
\end{align}
Recall the definition of the Gaussian statistic $\hatTmdZ$ in \eqref{def:BSstatistic} based on the random variables $\{\hatZ_{t,h}\}_{t=1,\dots,m+mT}^{h=1,\dots,d}$ defined in \eqref{eq:BootstrapZ}.
We introduce two additional sets of independent, centered Gaussian random vectors
$$
\big\{ \bm{Z}_{t} = \big(Z_{t,1},\dots,Z_{t,d}\big)^\top\big\}_{t=1,\dots,m+mT}
\;\;\;
\text{and}
\;\;\;
\big\{ \bm{\tilde{Z}}_{t} = \big(Z_{t,1},\dots,Z_{t,d}\big)^\top\big\}_{t=1,\dots,m+mT}
$$
with spatial covariance structures
\begin{align}\label{def:GaussFields}
\begin{split}
\Cov\big(Z_{t,h},\,Z_{t,i}\big) = \rho_{h,i}
\;\;&\text{and}\;\;
\Cov\big(\tildeZ_{t,h},\,\tildeZ_{t,i}\big) = \tilderho_{h,i}\,,
\end{split}
\end{align}
where $\rho_{h,i}$ and $\tilderho_{h,i}$  are the long-run correlations and truncated long-run correlations defined in \eqref{eq:longruncorr} and \eqref{eq:longruncorrW}, respectively.

Next, we define analogues of the statistic $\hatTmdZ$ in \eqref{def:BSstatistic} by 
\begin{align}\label{def:TmdZs}
\begin{split}
\TmdZ &:= \maxhd \maxkTm \max_{j=0}^{k-1} \frac{(k-j)w(k/m)}{\sqrt{m}}\Big|z_{m+j+1}^{m+k}(h) - z_1^{m+j}(h) \Big|~,\\
\tildeTmdZ &:= \maxhd \maxkTm \max_{j=0}^{k-1} \frac{(k-j)w(k/m)}{\sqrt{m}}\Big|\widetilde{z}_{m+j+1}^{m+k}(h) - \widetilde{z}_1^{m+j}(h) \Big|~,
\end{split}
\end{align}
where 
\begin{align}\label{def:smallzTildez}
z_i^{j}(h) := \dfrac{1}{j-i+1}\sum_{t=i}^j Z_{t,h}
\;\;\text{and}\;\;
\widetilde{z}_i^{j}(h) := \dfrac{1}{j-i+1}\sum_{t=i}^j \tildeZ_{t,h}~.
\end{align}
For a constant $0< c < T$, such that $cm \in \N$, we will now consider truncated versions of the statistics $\Tmd, \TmdZ, \hatTmdZ$ and $\tildeTmdZ$ defined by
\begin{align}\label{def:truncatedTmds}
\begin{split}
\Tmd(c)
:= \maxhd \max_{k=cm+1}^{Tm} \max_{j=0}^{k-cm-1} \frac{(k-j)w(k/m)}{\sqrt{m}\sigma_h}\Big|\hatmu_{m+j+1}^{m+k}(h) - \hatmu_1^{m+j}(h) \Big|~,\\
\TmdZ(c)
:= \maxhd \max_{k=cm+1}^{Tm} \max_{j=0}^{k-cm-1} \frac{(k-j)w(k/m)}{\sqrt{m}}\Big|z_{m+j+1}^{m+k}(h) - z_1^{m+j}(h) \Big|~,\\
\hatTmdZ(c)
:= \maxhd \max_{k=cm+1}^{Tm} \max_{j=0}^{k-cm-1} \frac{(k-j)w(k/m)}{\sqrt{m}}\Big| \hat{z}_{m+j+1}^{m+k}(h) - \hat{z}_1^{m+j}(h) \Big|~,\\
\tildeTmdZ(c)
:= \maxhd \max_{k=cm+1}^{Tm} \max_{j=0}^{k-cm-1} \frac{(k-j)w(k/m)}{\sqrt{m}}\Big| \widetilde{z}_{m+j+1}^{m+k}(h) - \widetilde{z}_1^{m+j}(h) \Big|~.
\end{split}
\end{align}
Finally, recall the definition
\begin{align}\label{def:udx}
u_d(x) = x/a_d + b_d~,\qquad x \in \R~,
\end{align} 
with the sequences $a_d$ and $b_d$ given by \eqref{eq:adbd} with the adaptation $q=q(T)$.

\noindent The proof of Theorem \ref{thm:mainGumbel} is now split into the following five Lemmas. 
If these are proven, then the claim is a consequence of Theorem \ref{thm:gumbelWiener.dep}. 

\begin{lemma}[Truncation]
\label{lem:Truncation} 
For any sufficiently small constant $t_0>0$ and for fixed $x \in \R$ we have
\begin{align*}
\bigg| \Pb\Big(\Tmd \leq u_d(x) \Big)
- \Pb\Big(\Tmd (t_0)\leq u_d(x) \Big) \bigg| = o(1)\quad\mbox{as }m,d \to \infty~.
\end{align*}
\end{lemma}

\begin{lemma}[Gaussian approximation]
\label{lem:gaussianapprox}
For $t_0>0$ it holds that
\begin{align*}
\sup_{x \in \R} \bigg|\Pb \big( \Tmd(t_0)  \leq x \big) - \Pb \big( \TmdZ(t_0) \leq x \big) \bigg| = o(1)\quad\mbox{as }m,d \to \infty~.
\end{align*}
\end{lemma}

\begin{lemma}[Relaxation of correlation structure]
\label{lem:relaxCorr}
For $t_0>0$ it holds that
\begin{align*}
\sup_{x \in \R} \bigg|\Pb \big( \tildeTmdZ(t_0)  \leq x \big) - \Pb \big( \TmdZ(t_0) \leq x \big) \bigg| = o(1)\quad\mbox{as }m,d \to \infty~.
\end{align*}
\end{lemma}

\begin{lemma}[Discretization of limit process]
\label{lem:discretize}
For $t_0>0$ and fixed $x \in \R$ it holds that
\begin{align*}
\bigg|\Pb \big( \tildeTmdZ(t_0) \leq u_d(x) \big) - \Pb \big( \mathcal{W}_d(t_0)  \leq u_d(x)  \big) \bigg|
= o(1)\quad\mbox{as }m,d \to \infty~.
\end{align*}
\end{lemma}

\begin{lemma}[Removing truncation]\label{lem:removeTrunc}
For fixed $x \in \R$ and any sufficiently small constant $t_0>0$ it holds that
\begin{align*}
\bigg|\Pb \big(  \mathcal{W}_d(t_0) \leq u_d(x) \big) - \Pb \big( \mathcal{W}_d  \leq u_d(x)  \big)\bigg|
= o(1)\quad\mbox{as }m,d \to \infty~.
\end{align*}
\end{lemma}
\bigskip

\begin{proof}[Proof of Lemma \ref{lem:Truncation}]
First, note that 
\begin{align*}
\Tmd
&= \maxhd \max_{k=1}^{Tm} \max_{j=0}^{k-1}  \frac{(k-j)w(k/m)}{\sqrt{m}\sigma_h}\Big|\hatmu_{m+j+1}^{m+k}(h) - \hatmu_1^{m+j}(h) \Big|\\
&=  \max \bigg\{\Tmd(t_0)~,  \maxhd \max_{k=t_0m+1}^{Tm} \max_{j=k-t_0m}^{k-1} \frac{(k-j)w(k/m)}{\sqrt{m}\sigma_h}\Big|\hatmu_{m+j+1}^{m+k}(h) - \hatmu_1^{m+j}(h) \Big|~,\\
&\hspace{4.7cm}\max_{h=1}^d\max_{k=1}^{t_0m} \max_{j=0}^{k-1} \frac{(k-j)w(k/m)}{\sqrt{m}\sigma_h}\Big|\hatmu_{m+j+1}^{m+k}(h) - \hatmu_1^{m+j}(h) \Big|\bigg\}~.
\end{align*}
Hence, we obtain
\begin{align}\label{ineq:trunc1}
\Big|\Pb\Big(\Tmd(t_0) \leq u_d(x) \Big) - \Pb\Big(\Tmd \leq u_d(x) \Big)\Big|
\leq P_1(x) + P_2(x)~,
\end{align}
where
\begin{align*}
P_1(x) &= \Pb\Big( \maxhd \max_{k=t_0m+1}^{Tm} \max_{j=k-t_0m}^{k-1} \frac{(k-j)}{\sqrt{m}\sigma_h}\Big|\hatmu_{m+j+1}^{m+k}(h) - \hatmu_1^{m+j}(h) \Big|\geq u_d(x)\Big)~,\\
P_2(x) &= \Pb\Big(\max_{h=1}^d\max_{k=1}^{t_0m} \max_{j=0}^{k-1} \frac{(k-j)}{\sqrt{m}\sigma_h}\Big|\hatmu_{m+j+1}^{m+k}(h) - \hatmu_1^{m+j}(h) \Big|\geq u_d(x)\Big)~.
\end{align*}
and we additionally used that $w(k/m)\leq 1$.
We will treat the summands on the right-hand side of the last display separately.
For the term $P_1(x)$ note that  
\begin{align}\label{eq:trunc2}
\begin{split}
& \Pb\left(\max_{k=t_0m+1}^{Tm} \max_{j=k-t_0m}^{k-1} \frac{(k-j)}{\sqrt{m}\sigma_h}\Big|\hatmu_{m+j+1}^{m+k}(h) - \hatmu_1^{m+j}(h) \Big|\geq u_d(x)\right)\\
&\qquad\qquad\leq \;\Pb \left(\max_{k=t_0m+1}^{Tm} \max_{j=k-t_0m}^{k-1} \frac{(k-j)}{\sqrt{m}\sigma_h}\Big|\hatmu_{m+j+1}^{m+k}(h) \Big|\geq \frac{u_d(x)}{2}\right)\\
&\qquad\qquad\qquad + \Pb\left(\max_{k=t_0m+1}^{Tm} \max_{j=k-t_0m}^{k-1} \frac{(k-j)}{\sqrt{m}\sigma_h}\Big|\hatmu_1^{m+j}(h) \Big|\geq \frac{u_d(x)}{2}\right)~.
\end{split}
\end{align}
Using stationarity and Assumption \ref{assump:temporalDependence} \eqref{assump:TD2}, we have 
\begin{align*}
\Pb\Big(&\max_{k=t_0m+1}^{Tm} \max_{j=k-t_0m}^{k-1} \frac{(k-j)}{\sqrt{m}\sigma_h}\Big|\hatmu_{m+j+1}^{m+k}(h)  \Big|\geq \frac{u_d(x)}{2}\Big)\\
&\leq \sum_{k=t_0m+1}^{Tm} \Pb \left(\max_{j=k-t_0m}^{k-1}\Big|\sum_{i=m+j+1}^{m+k}X_{i,h}  \Big|\geq \frac{\sqrt{m}c_\sigma u_d(x)}{2}\right).
\end{align*}
Observing \eqref{eq:adbd} and Lemma \ref{lem:nagaev}, we obtain the following bound for the last display, which holds uniformly for $1\leq h\leq d$
\begin{align*}
C_p\frac{Tt_0m^{2-p/2}}{c_\sigma^p u_d(x)^p} + C_p Tm\exp\left(-c_p\frac{c_\sigma^2 u_d(x)^2}{4t_0}\right)
\lesssim \frac{m^{2-p/2}}{(\log(d))^{p/2}} + md^{-{\tilde C_p}/{t_0}}~,
\end{align*}
where $\tilde C_p>0$ is a sufficiently small constant.
The second summand on the right-hand side of \eqref{eq:trunc2} can be estimated similarly, that is 
\begin{align*}
\Pb \Big(&\max_{k=t_0m+1}^{Tm} \max_{j=k-t_0m}^{k-1} \frac{(k-j)}{\sqrt{m}\sigma_h}\Big|\hatmu_1^{m+j}(h) \Big|\geq \frac{u_d(x)}{2}\Big)
\leq \Pb\left( \max_{j=1}^{Tm+m}\Big|\sum_{i=1}^{j}X_{i,h} \Big|\geq \frac{c_\sigma\sqrt{m}}{t_0}\frac{u_d(x)}{2}\right)\\
&\lesssim C_p\frac{t_0^p(T+1)m^{1-p/2}}{c_\sigma^p u_d(x)^p}+C_p\exp\left( - c_p\frac{c_\sigma^2 u_d(x)^2}{4t_0^2(T+1)}\right)
\lesssim \frac{m^{1-p/2}}{(\log(d))^{p/2}}+d^{-{\tilde C_p}/{t_0^2}}~,
\end{align*}
where $\tilde C_p>0$ is again a sufficiently small constant.
Hence, we obtain by Assumption  \ref{assump:model}, \eqref{eq:trunc2} (observing $p>2D+4$) that  
\begin{align}\label{eq:trunc3}
P_1(x) \lesssim &~\frac{dm^{2-p/2}}{(\log(d))^{p/2}}+md^{1-\frac{\tilde C_p}{t_0}}
\lesssim \frac{m^{D+2-p/2}}{(\log(d))^{p/2}} + m^{1+\left(1-{\tilde C_p}/{t_0}\right)/C_D}=o(1)
\end{align}
if $t_0>0$  is chosen sufficiently small.
Analogously, we obtain for the second summand on the right-hand side of \eqref{ineq:trunc1} with a possibly smaller 
constant $t_0>0$, that
\begin{align}\label{eq:trunc4}
P_2(x) = o(1)~,
\end{align}
where we have used the   following two inequalities which are a consequence of  Lemma \ref{lem:nagaev}
\begin{align*}
\Pb &\left(\max_{k=1}^{t_0m} \max_{j=0}^{k-1} \frac{(k-j)}{\sqrt{m}\sigma_h}\Big|\hatmu_{m+j+1}^{m+k}(h) \Big|\geq \frac{u_d(x)}{2}\right)
\leq \sum_{k=1}^{t_0m} \Pb \left(\max_{j=0}^{k-1}\Big|\sum_{i=m+j+1}^{m+k}X_{i,h}  \Big|\geq \sqrt{m}c_\sigma\frac{u_d(x)}{2}\right)\\
&\leq \sum_{k=1}^{t_0m} \Pb \left(\max_{j=0}^{t_0m-1}\Big|\sum_{i=m+j+1}^{m+t_0m}X_{i,h}  \Big|\geq \sqrt{m}c_\sigma\frac{u_d(x)}{2}\right)\\
&\lesssim C_p\frac{t_0^2m^{2-p/2}}{c_\sigma^p(\log(d))^{p/2}}+C_pt_0m\exp\left(-c_p\frac{c_\sigma^2u_d(x)^2}{4t_0}\right)
\lesssim \frac{m^{2-p/2}}{(\log(d))^{p/2}}+md^{-{\tilde C_p}/{t_0}} 
\end{align*}
and
\begin{align*}
\Pb &\bigg(\max_{k=1}^{t_0m}\max_{j=0}^{k-1} \frac{(k-j)}{\sqrt{m}\sigma_h}\Big|\hatmu_1^{m+j}(h) \Big|\geq \frac{u_d(x)}{2}\bigg)
\leq ~\Pb \bigg( \max_{j=1}^{t_0m+m}\Big|\sum_{i=1}^{j}X_{i,h} \Big|\geq \frac{\sqrt{m}c_\sigma}{t_0}\frac{u_d(x)}{4}\bigg)\\
&\lesssim C_p \frac{(t_0+1)t_0^pm^{1-p/2}}{c_\sigma^p(\log(d))^{p/2}}+C_p\exp\left(-c_p\frac{c_\sigma^2 u_d(x)^2}{4(t_0+1)t_0^2(T+1)}\right)
\lesssim \frac{m^{1-p/2}}{(\log(d))^{p/2}}+d^{-\frac{\tilde C_p}{(t_0+1)t_0^2}}~.
\end{align*}
Combining \eqref{eq:trunc3} and \eqref{eq:trunc4} the assertion of Lemma  \ref{lem:Truncation} now follows from   \eqref{ineq:trunc1}.

\medskip

\begin{proof}[Proof of Lemma \ref{lem:gaussianapprox}]
We will use a  Gaussian Approximation provided in Corollary 2.2 of \cite{Zhang2018b}.
For this purpose we introduce the notation
\begin{align*}
v_{m,k,j,h} := \dfrac{(k-j)w(k/m)}{\sigma_h\sqrt{m}}  \Big( \hatmu_{m+j+1}^{m+k}(h) - \hatmu_1^{m+j}(h) \Big) ~,
\end{align*}
with $k=t_0m+1,\dots, Tm $; $j=0\dots,k-t_0m-1$ and $h=1,\dots,d$.
We stack all these quantities together in one vector
{\footnotesize \begin{align*}
&\bm{V}_+ :=(v_{m,t_0m+1,0,1},v_{m,t_0m+2,0,1},v_{m,t_0m+2,1,1},\dots,v_{m,Tm,Tm-t_0m-1,1},v_{m,t_0m+1,0,2},\dots,v_{m,Tm,Tm-t_0m-1,d})^\top.
\end{align*}}
Next define the vector 
$$\bm{V}  = (V_1,V_2,\dots,V_{d_{\bm{V}}})^\top := \big(\bm{V}^\top_+,-\bm{V}^\top_+\big)^\top
$$
and denote its dimension by   $d_{\bm{V}}$. Observe that by construction the identity
\begin{align*}
\max_{i=1}^{d_{\bm{V}}} V_i
= \Tmd(t_0)
= \maxhd \max_{k=mt_0+1}^{Tm} \max_{j=0}^{k-mt_0-1} \dfrac{(k-j)w(k/m)}{\sigma_h\sqrt{m}} \left| \hatmu_{m+j+1}^{m+k}(h) - \hatmu_1^{m+j}(h) \right|
\end{align*}
holds, where we use the fact that $\bm{V}$ contains both, the positive and negative version of all random variables which 
appear in the maximum in the definition of the statistic
 $\Tmd(t_0)$.
Further note that the dimension of $\bm{V}$ is bounded by 
\begin{align}\label{ineq:dimv}
d_{\bm{V}}
\leq 2d(Tm)^2.
\end{align}
By the construction above each component $V_i$ corresponds either to  $v_{m,k,j,h}$ or to $-v_{m,k,j,h}$ for some combination $k,j,h$.
Hence, it can be represented by 
$$
V_i = \dfrac{1}{\sqrt{m}}\sum\limits_{t=1}^{m(T+1)} X_{t,i}^*
$$
with
\begin{align}\label{eq:xtilde}
X_{t,i}^* = \begin{cases}
\dfrac{a_{t,m,k,j}}{\sigma_h} X_{t,h}\;\;&\text{for}\;\;1 \leq i \leq d_{\bm{V}}/2~,\\[16pt]
-\dfrac{a_{t,m,k,j}}{\sigma_h} X_{t,h}~\;\;&\text{for}\;\;d_{\bm{V}}/2+1 \leq i \leq d_{\bm{V}}~,
\end{cases}
\end{align}
where the indices $k,j,h$ correspond to $i$ according to the construction of the vector $\bm{V}$ and the coefficients $a_{t,m,k,j}$ are given by
\begin{align}\label{eq:atmkj}
a_{t,m,k,j}
= \begin{cases}
0  &\text{if}\;\;\; t > m+k~,\\
a^{(1)}_{m,k}:=w(k/m) \;\;  &\text{if}\;\;\; m+j< t \leq m+k~,\\[12pt]
a^{(2)}_{m,k,j}:=-\dfrac{(k-j)w(k/m)}{(m+j)} \;\; &\text{if}\;\;\; t \leq m+j~.
\end{cases}
\end{align}
Using the fact $w(k/m) = 1/(1 + k/m)$ and $1 \leq k \leq mT$, we obtain
\begin{align}\label{ineq:a1mk}
\dfrac{1}{T+1} \leq w(k/m) = a^{(1)}_{m,k} \leq 1
\end{align}
and as $t_0m \leq k-j$ and $j\leq k \leq Tm$ it follows that
\begin{align}\label{ineq:a2mkj}
\dfrac{t_0}{(T+1)^2} \leq  |a^{(2)}_{m,k,j}| = \bigg| \dfrac{(k-j)w(k/m)}{m+j} \bigg| \leq T~,
\end{align}
which yields by definition of $a_{t,m,k,j}$ in \eqref{eq:atmkj} the upper bound
\begin{align}\label{bound:atmkj}
|a_{t,m,k,j}|
\leq T_+:=\max\{T,1\}~.
\end{align}
Moreover, the temporal dependence structure of 
the $d_{\bm{V}}$-dimensional time series
$$
\big\{(X^*_{t,1}, \ldots , X^*_{1,d_{\bm{V}}} )^\top \big\}_{t\in\mathbb{Z}}
$$ 
still satisfies the concept of physical dependence as
\begin{align*}
X^*_{t,i}=g^*_{t,i}(\varepsilon_t,\varepsilon_{t-1},\ldots )
\end{align*}
with
\begin{align}\label{eq:gtilde}
g^*_{m,t,i}(\varepsilon_t,\varepsilon_{t-1},\ldots ):= \begin{cases}
\dfrac{a_{t,m,k,j}}{\sigma_h} g_{h}(\varepsilon_t,\varepsilon_{t-1},\ldots )\;\;&\text{for}\;\;1 \leq i \leq d_{\bm{V}}/2~,\\[16pt]
-\dfrac{a_{t,m,k,j}}{\sigma_h} g_{h}(\varepsilon_t,\varepsilon_{t-1},\ldots )~\;\;&\text{for}\;\;d_{\bm{V}}/2+1 \leq i \leq d_{\bm{V}}~,
\end{cases}
\end{align}
where the indices $k,j,h$ correspond to $i$ according to the construction of the vector $\bm{V}$.\\
In the following let $(\dot{V}^{(z)}_1,\dots,\dot{V}^{(z)}_{d_{\bm{V}}})^\top$ denote a centered Gaussian distributed vector with the same covariance structure as $\bm{V}$.
Next, recall the definition of the Gaussian random variables $\{Z_{t,h}\}_{t=1,\dots,m+mT}^{h=1,\dots,d}$ in \eqref{def:GaussFields} and let
$$
Z^*_{t,i} = \begin{cases}
a_{t,m,k,j} Z_{t,h}\;\;&\text{for}\;\;1 \leq i \leq d_{\bm{V}}/2~,\\[16pt]
-a_{t,m,k,j} Z_{t,h}~\;\;&\text{for}\;\;d_{\bm{V}}/2+1 \leq i \leq d_{\bm{V}}~.
\end{cases}
$$
Further define the vector $\bm{V}^{(z)} = (V^{(z)}_1,\dots,V^{(z)}_{d_{\bm{V}}})^\top $ by
$$
V^{(z)}_i := \dfrac{1}{\sqrt{m}}\sum_{t=1}^{m(T+1)} Z^*_{t,i}\;\;\;\;\;i=1,\dots,d_{\bm{V}}~.
$$
We now proceed as follows:\\[12pt]
\textbf{Step 1:} Show that for  some (sufficiently small) constant $\tilde{C}_1>0$
\begin{align}\label{ineq:KolmDistance1}
\sup_{x \in \R} \Big| \Pb\Big( \max_{i=1}^{d_{\bm{V}}} \dot{V}_i^{(z)}  \leq x \Big) 
- \Pb\Big( \max_{i=1}^{d_{\bm{V}}} V_i^{(z)}  \leq x \Big)  \Big|
\lesssim m^{-\tilde{C}_1} \,.
\end{align}
\textbf{Step 2:}
Establish that  for some (sufficiently small) constant $\tilde{C}_2>0$
\begin{align}\label{ineq:KolmDistance2}
\sup_{x \in \R} \Big| \Pb\Big( \max_{i=1}^{d_{\bm{V}}} V_i \leq x \Big) 
- \Pb\Big( \max_{i=1}^{d_{\bm{V}}} \dot{V}_i^{(z)}  \leq x \Big)  \Big|
\lesssim m^{-\tilde{C}_2}~.
\end{align}
If both steps have been proven, the claim of Lemma \ref{lem:gaussianapprox} follows from the identity
\begin{align*}
\max_{i=1}^{d_{\bm{V}}} V^{(z)}_i = \TmdZ(t_0)~.
\end{align*}

\noindent \textbf{Proof of Step 1:}
As we aim to compare the maxima of the two Gaussian distributed vectors $\bm{V}^{(z)}$ and $\dot{\bm{V}}^{(z)}$ we will apply Lemma \ref{lem:ineqcherno2}.
Therefore, we analyze the covariance structures of $\bm{V}^{(z)}$ and $\dot{\bm{V}}^{(z)}$ [or equivalently $\bm{V}$].
Let $i_1, i_2 \in \{1,\dots,d_{\bm{V}}/2\}$ with corresponding indices $h_1,j_1,k_1$ and $h_2,j_2,k_2$ according to equation \eqref{eq:xtilde}.
For the calculation we assume without loss of generality that $j_1\leq j_2$ and use the notation $k_{min} = \min\{k_1,k_2\}$, $k_{max} = \max\{k_1, k_2\}$ and $j_2 \wedge k_1 = \min\{j_2,k_1\}$.
Further we use the convention $\sum_{i=j}^k \beta_i =0$, whenever $k<j$.
For the covariance of the components of the vector $\bm{V}^{(z)}$ note that temporal independence yields
\begin{align*}
&\Cov\Big(V^{(z)}_{i_1},\,V^{(z)}_{i_2}\Big)
= \dfrac{1}{m}\sum_{t=1}^{m+k_{min}} \Cov\Big(a_{t,m,k_1,j_1} Z_{t,h_1},\,a_{t,m,k_2,j_2} Z_{t,h_2}\Big)\\
&= \dfrac{1}{m}\sum_{t=1}^{m+j_1} \Cov\Big(a_{t,m,k_1,j_1} Z_{t,h_1},\,a_{t,m,k_2,j_2} Z_{t,h_2}\Big)
+ \dfrac{1}{m}\sum_{t=m+j_1+1}^{m+(j_2 \wedge k_1)} \Cov\Big(a_{t,m,k_1,j_1} Z_{t,h_1},\,a_{t,m,k_2,j_2} Z_{t,h_2}\Big)\\
&\quad+ \dfrac{1}{m}\sum_{t=m+(j_2 \wedge k_1)+1}^{m+k_{min}} \Cov\Big(a_{t,m,k_1,j_1} Z_{t,h_1},\,a_{t,m,k_2,j_2} Z_{t,h_2}\Big)~.
\end{align*}
Using the definition in \eqref{eq:atmkj} and \eqref{def:GaussFields} we obtain 
\begin{align}
\begin{split}\label{eq:covVZ}
\Cov\Big(V^{(z)}_{i_1},\,V^{(z)}_{i_2}\Big)
&= \dfrac{a^{(2)}_{m,k_1,j_1}a^{(2)}_{m,k_2,j_2}}{m\sigma_{h_1}\sigma_{h_2}}(m+j_1)\gamma_{h_1,h_2}
+\dfrac{a^{(1)}_{m,k_1}a^{(2)}_{m,k_2,j_2}}{m\sigma_{h_1}\sigma_{h_2}}\big((j_2 \wedge k_1)-j_1\big)\gamma_{h_1,h_2}\\
&\hspace{2cm}+\dfrac{a^{(1)}_{m,k_1}a^{(1)}_{m,k_2}}{m\sigma_{h_1}\sigma_{h_2}} (k_{min}-j_2)\gamma_{h_1,h_2} \,I\{j_2 < k_{min}\}\,.
\end{split}
\end{align}
Similar calculations also yield
\begin{align*}
\Var\big(V^{(z)}_{i_1}\big)
= \Cov\big(V^{(z)}_{i_1},V^{(z)}_{i_1}\big)
= \big(a^{(2)}_{m,k_1,j_1}\big)^2 \dfrac{(m+j_1)}{m} + \big(a^{(1)}_{m,k_1}\big)^2\dfrac{k_1-j_1}{m}
\end{align*}
and from  \eqref{ineq:a1mk} and \eqref{ineq:a2mkj} it follows that 
\begin{align}\label{ineq:varzh1}
\dfrac{t_0^2}{(T+1)^4}+\dfrac{t_0}{(T+1)^2} 
\leq \Var\big(V^{(z)}_{i_1}\big)
\leq T^3 + T~.
\end{align}
By the same arguments we obtain for the covariance structure of the components of the vector
$\dot{\bm{V}}^{(z)}$ [note that we cannot use temporal independence here]:
\begin{align}\label{eq:decompCovarVi}
\begin{split}
\Cov\big(\dot{V}^{(z)}_{i_1},\,\dot{V}^{(z)}_{i_2}\big)
&= \Cov\bigg(\dfrac{1}{\sqrt{m}}\sum_{t=1}^{m+k_1}X^*_{t,i_1} ,\,\dfrac{1}{\sqrt{m}}\sum_{t=1}^{m+k_2}X^*_{t,i_2}\bigg)\\
&= \sum_{\ell=1}^4 \Cov\Big(S_\ell^{(1)}, S_\ell^{(2)}\Big) 
+ \sum_{\substack{\ell,j=1\\\ell\neq j}}^4 \Cov\Big(S_\ell^{(1)}, S_j^{(2)}\Big)~,
\end{split}
\end{align}
where the terms in the last line are defined for $u = 1,2$ by
\begin{align*}
S_1^{(u)} &= \dfrac{1}{\sqrt{m}}\sum_{t=1}^{m+j_1}X^*_{t,i_u}
\;,
\;\;\;
S_2^{(u)} = \dfrac{1}{\sqrt{m}}\sum_{t=m+j_1+1}^{m+(j_2 \wedge k_1)}X^*_{t,i_u}~,
\\
S_3^{(u)} &= \dfrac{1}{\sqrt{m}}\sum_{t=m+(j_2 \wedge k_1)+1}^{m+k_{min}}X^*_{t,i_u}\;,
\;\;\;
S_4^{(u)} = \dfrac{1}{\sqrt{m}}\sum_{t=m+k_{min}+1}^{m+k_{max}}X^*_{t,i_u}~.
\end{align*}
We will now treat the two sums on the right-hand side of \eqref{eq:decompCovarVi} separately
and show that the first sum is close to $\Cov\big(V^{(z)}_{i_1},\,V^{(z)}_{i_2}\big)$, while the second vanishes sufficiently fast.
Using that by construction, either $S_4^{(1)}=0$ or $S_4^{(2)}=0$, we obtain that
\begin{align*}
\sum_{\ell=1}^4 \Cov\Big(S_\ell^{(1)}, S_\ell^{(2)}\Big) 
&= \dfrac{a^{(2)}_{m,k_1,j_1}a^{(2)}_{m,k_2,j_2}}{m\sigma_{h_1}\sigma_{h_2}}
\sum_{t=1}^{m+j_1}\sum_{s=1}^{m+j_1} \Cov\Big(X_{t,h_1},X_{s,h_2}\Big)\\
&\quad+ \dfrac{a^{(1)}_{m,k_2}a^{(2)}_{m,k_2,j_2}}{m\sigma_{h_1}\sigma_{h_2}}
\sum_{t=m+j_1+1}^{m+(j_2 \wedge k_1)}\sum_{s=m+j_1+1}^{m+(j_2 \wedge k_1)}\Cov\Big(X_{t,h_1},X_{s,h_2}\Big)\\
&\quad+ \dfrac{a^{(1)}_{m,k_1}a^{(1)}_{m,k_2}}{m\sigma_{h_1}\sigma_{h_2}}\sum_{t=m+(j_2 \wedge k_1)+1}^{m+k_{min}}\sum_{s=m+(j_2 \wedge k_1)+1}^{m+k_{min}}\Cov\Big(X_{t,h_1},X_{s,h_2}\Big)\\
&=\dfrac{a^{(2)}_{m,k_1,j_1}a^{(2)}_{m,k_2,j_2}}{m\sigma_{h_1}\sigma_{h_2}}\sum_{t=-m-j_1}^{m+j_1}\big(m+j_1-|t|\big)\phi_{t,h_1,h_2}\\
&\quad+  \dfrac{a^{(1)}_{m,k_2}a^{(2)}_{m,k_2,j_2}}{m\sigma_{h_1}\sigma_{h_2}}\sum_{t=-(j_2 \wedge k_1)+j_1}^{(j_2 \wedge k_1)-j_1}\big((j_2 \wedge k_1)-j_1-|t|\big)\phi_{t,h_1,h_2}\\
&\quad+ \dfrac{a^{(1)}_{m,k_1}a^{(1)}_{m,k_2}}{m\sigma_{h_1}\sigma_{h_2}}I\big\{k_{min} > j_2 \wedge k_1 \big\}\sum_{t=-k_{min}+(j_2 \wedge k_1)}^{k_{min}-(j_2 \wedge k_1)}\big(k_{min}-(j_2 \wedge k_1)-|t|\big)\phi_{t,h_1,h_2}\,,
\end{align*}
where we used the notation $\phi_{t,h_1,h_2} := \Cov(X_{0,h_1},X_{t,h_2})$.  
Combining the bounds in \eqref{ineq:a1mk} and \eqref{ineq:a2mkj} with \eqref{eq:covVZ} and Assumption \ref{assump:temporalDependence} \eqref{assump:TD2}, we deduce that
\begin{align}\label{ineq:covarVi1}
\begin{split}
\bigg|&\Cov\big(V^{(z)}_{i_1},\,V^{(z)}_{i_2}\big) - \sum_{\ell=1}^4 \Cov\Big(S_\ell^{(1)}, S_\ell^{(2)}\Big) \bigg|
\leq \dfrac{C_{T,t_0}}{c_\sigma^2 m}\Bigg[ \sum_{t \in \Z} \min\{|t|,m+j_1\}|\phi_{t,h_1,h_2}| \\
&\quad + \sum_{t \in \Z} \min\{|t|,\,(j_2 \wedge k_1)-j_1\}|\phi_{t,h_1,h_2}|
+ \sum_{t \in \Z} \min\{|t|,\,k_{min}-(j_2 \wedge k_1)\}|\phi_{t,h_1,h_2}|\Bigg]\\
&\leq \dfrac{3C_{T,t_0}}{c_\sigma^2 m} \sum_{t \in \Z} |t||\phi_{t,h_1,h_2}|\,,
\end{split}
\end{align}
where the constant $C_{T,t_0}$ depends on $T$ and $t_0$ only and we used the definition of $\gamma_{h_1,h_2}$ in \eqref{eq:longruncovariance}.
Using Assumption \ref{assump:temporalDependence} \eqref{assump:TD1} and Lemma E4 from \cite{Jirak2015} it follows that
\begin{align}\label{bound:crosscovariance}
\sup_{h_1,h_2 \in \N} \sum_{t \in \Z} |t||\phi_{t,h_1,h_2}|< \infty~,
\end{align}
which yields 
\begin{align}\label{bound:Ssym}
\bigg|\Cov\big(V^{(z)}_{i_1},\,V^{(z)}_{i_2}\big) - \sum_{\ell=1}^4 \Cov\Big(S_\ell^{(1)}, S_\ell^{(2)}\Big) \bigg|
\lesssim \dfrac{1}{m}~,
\end{align}
where the involved constant is independent of $i_1$ and $i_2$ [or equivalently $j_1,j_2,k_1,k_2,h_1$ and $h_2$].
Next, we treat the second sum on the right-hand of \eqref{eq:decompCovarVi}.
For that purpose, note that for arbitrary points in time $p_1 < p_2 < p_3 < p_4$, it holds that
\begin{align}\label{ineq:Snonoverlap}
\begin{split}
\bigg| \Cov\Big( &\sum_{t=p_1}^{p_2} X_{t,h_1},\,\sum_{s=p_3}^{p_4} X_{s,h_2}\Big)\bigg|
\leq \sum_{t=1}^{p_2}\sum_{s=p_2+1}^{p_4} \Big|\Cov\Big(X_{t,h_1},\, X_{s,h_2}\Big)\Big|\\
&= \sum_{t=1}^{p_2}\sum_{s=p_2+1}^{p_4} |\phi_{s-t,h_1,h_2}|
= \sum_{t=1}^{p_2}\sum_{s=p_2-t+1}^{p_4-t} |\phi_{s,h_1,h_2}|\\
&\leq \sum_{t=1}^{p_2}\sum_{s=p_2-t+1}^{p_4} |\phi_{s,h_1,h_2}|
= \sum_{s=1}^{p_4}\sum_{t=p_2-s+1}^{p_2} |\phi_{s,h_1,h_2}|
= \sum_{s=1}^{p_4} s|\phi_{s,h_1,h_2}|~.
\end{split}
\end{align}
Using the upper bound for the coefficients $a_{t,m,k,j}$ in \eqref{bound:atmkj}, the uniform bound in \eqref{bound:crosscovariance} and that all the pairs of the sums under consideration are non-overlapping as treated above in \eqref{ineq:Snonoverlap}, we obtain directly that
\begin{align}\label{bound:Scross}
\sum_{\substack{\ell,j=1\\i\neq j}}^4 \Cov\Big(S_\ell^{(1)}, S_j^{(2)}\Big) 
\lesssim \dfrac{1}{m}~,
\end{align}
where the constant is again independent of $i_1, i_2$.
Combining the estimates \eqref{bound:Ssym} and \eqref{bound:Scross}, we conclude
\begin{align}\label{ineq:delta}
\Delta_m
:= \max_{i_1,i_2=1}^{d_{\bm{V}}}
\bigg|\Cov\big(V^{(z)}_{i_1},\,V^{(z)}_{i_2}\big) 
- \Cov\big(\dot{V}^{(z)}_{i_1},\,\dot{V}^{(z)}_{i_2}\big)\bigg|
\lesssim \dfrac{1}{m}\,.
\end{align}
Due to \eqref{ineq:varzh1} we can now apply Lemma \ref{lem:ineqcherno2}, which gives
\begin{align*}
\sup_{x \in \R} \Big| \Pb\Big( \max_{i=1}^{d_{\bm{V}}} V^{(z)}_i  \leq x \Big) 
- \Pb\Big( \max_{i=1}^{d_{\bm{V}}} \dot{V}^{(z)}_i  \leq x \Big)  \Big| &\lesssim \Delta_m^{1/3}\cdot\max\Big\{1,\, \log\big(d_{\bm{V}}/\Delta_m \big) \Big\}^{2/3}\\
&\lesssim \max\Big\{\Delta_m^{1/2},\, \Delta_m^{1/2}\big|\log d_{\bm{V}}\big|+ \Delta_m^{1/2}\big|\log\Delta_m\big|\Big\}^{2/3}\,.
\end{align*}
Using \eqref{ineq:dimv} and Assumption \ref{assump:model} the assertion of Step 1 follows.\\

\noindent \textbf{Proof of Step 2:}
Corollary 2.2 of \cite{Zhang2018b} yields the Gaussian approximation in \eqref{ineq:KolmDistance2} 
if the following three conditions hold uniformly in $t$ and $i$ (or equivalently in $t,k,j,h$).
\begin{enumerate}[label=(\roman*)]
\item There exist a constant  $b\in [0,1/11)$ and a deterministic sequence $B^*_m\lesssim m^{(3-17b)/8}$ such that $$d_{\bm{V}}\lesssim \exp((Tm)^b) \quad \text{ and }\quad \E\Big[\exp\big(|X^*_{t,i}|/B^*_m\big)\Big] \leq C_e\,,$$
where $C_e>1$ is the constant from Assumption \ref{assump:RV} \eqref{assump:S1}.
\item With $\beta$ as in Assumption \ref{assump:temporalDependence} \eqref{assump:TD1} it holds
$$\sum_{\ell=u}^{\infty}\sup_{t\in\mathbb{Z}}
\big\Vert g^*_{t,i}(\varepsilon_t,\varepsilon_{t-1},\ldots )-g^*_{t,i}(\varepsilon_t,\varepsilon_{t-1},\ldots,\varepsilon_{t-\ell+1},\varepsilon_{t-\ell}',\varepsilon_{t-\ell-1},\ldots ) \big\Vert_p\lesssim \beta^u~,
$$ where $\varepsilon_{t-l}'$ is an independent copy of $\varepsilon_{t-\ell}$.
\item There exist positive constants $c_1,c_2$ such that $c_1 \leq \Var\big(V_i\big)\leq c_2$.
\end{enumerate}
Therefore  the proof of Lemma \ref{lem:gaussianapprox} is completed by establishing these conditions.\\

\noindent Proof of (i): By \eqref{ineq:dimv} and Assumption \ref{assump:model} \eqref{assump:D1} the inequality $d_{\bm{V}}\lesssim \exp((Tm)^b)$ holds for any $b>0$.
Due to Assumption \ref{assump:RV} and the upper bound on $|a_{t,m,k,j}|$ in \eqref{bound:atmkj} we obtain that
\begin{align}\label{ineq:3aCherno}
|X^*_{t,i}|
\eqd \dfrac{|a_{t,m,k,j}|}{\sigma_h}|X_{1,h}|
\leq \dfrac{T_+}{c_\sigma}|X_{1,h}|.
\end{align}
Defining $B^*_m:=\dfrac{T_+}{c_\sigma}B_m \lesssim m^B$, where $B_m$ is the sequence from Assumption \ref{assump:RV} and $B<3/8$, it follows that
\begin{align*}
\E\Big[\exp\big(|X^*_{t,i}|/B^*_m\big)\Big] &=\E\Big[\exp\big(|a_{t,m,k,j}||X_{t,h}|/(\sigma_h B^*_m)\big)\Big]\\
&\leq \E\Big[\exp\big(|X_{1,h}|/B_m\big)\Big]\leq C_e\,.
\end{align*}
As $d_{\bm{V}}\lesssim \exp((Tm)^b)$ holds for any $b>0$, we can choose $b$ to be sufficiently small such that $B<(3-17b)/8<3/8$. \\
\par

\noindent Proof of (ii): In view of \eqref{eq:gtilde}, \eqref{bound:atmkj}, Assumption \ref{assump:temporalDependence} \eqref{assump:TD1} and the stationarity of $\{\varepsilon_t\}_{t\in\mathbb{Z}}$ we obtain
\begin{align*}
\sum_{\ell=u}^{\infty}&\sup_{t\in\mathbb{Z}}\big\Vert g^*_{m,t,i}(\varepsilon_t,\varepsilon_{t-1},\ldots ) - g^*_{m,t,i}(\varepsilon_t,\varepsilon_{t-1},\ldots,\varepsilon_{t-\ell+1},\varepsilon_{t-\ell}',\varepsilon_{t-\ell-1},\ldots ) \big\Vert_p\\
&= \sum_{\ell=u}^{\infty}\sup_{t\in\mathbb{Z}}\dfrac{|a_{t,m,k,j}|}{\sigma_h}\big\Vert g_{h}(\varepsilon_\ell,\varepsilon_{\ell-1},\ldots )-g_{h}(\varepsilon_\ell,\varepsilon_{\ell-1},\ldots,\varepsilon_{1},\varepsilon_{0}',\varepsilon_{-1},\ldots ) \big\Vert_p\\
&\leq  \dfrac{T_+}{c_{\sigma}}\sum_{\ell=u}^{\infty}\vartheta_{\ell,h,p}\leq \dfrac{T_+C_\vartheta}{c_{\sigma}} \sum_{\ell=u}^{\infty}\beta^\ell\lesssim \beta^u\,.
\end{align*}

\noindent Proof of (iii): The assertion follows by combining \eqref{ineq:varzh1} and \eqref{ineq:delta}.
\end{proof}

\begin{proof}[Proof of Lemma \ref{lem:relaxCorr}]
Recall the definition of the Gaussian vector $\bm{V}^{(z)} = \big(V^{(z)}_1 ,\dots, V^{(z)}_{\dV} \big)^\top$ from the proof of Lemma \ref{lem:gaussianapprox}, which fulfills the identity
\begin{align*}
\TmdZ(t_0) = \max_{i=1}^{\dV} V^{(z)}_i~.
\end{align*}
Applying again the vectorization technique as introduced in the proof of Lemma \ref{lem:gaussianapprox}, we can define analogously a Gaussian vector for the statistic  $\tildeTmdZ$ in \eqref{def:truncatedTmds}.
Recall the definition of  $\tilde{z}_i^j$ in  \eqref{def:smallzTildez} and  
introduce the notation
\begin{align}\label{def:tildevmkjh}
\tilde{v}_{m,k,j,h} &:= \dfrac{(k-j)w(k/m)}{\sqrt{m}}
\Big( \tilde{z}_{m+j+1}^{m+k}(h) - \tilde{z}_1^{m+j}(h) \Big)~,
\end{align}
with $k=t_0m+1,\dots,Tm $ and $j=0\dots,k-t_0m-1$ and $h=1,\dots,d$.
We stack all these quantities together in one vector, this is
{\footnotesize \begin{align*}
&\tildebVz_+:=(\tilde{v}_{m,t_0m+1,0,1},\tilde{v}_{m,t_0m+2,0,1},\tilde{v}_{m,t_0m+2,1,1},\dots,\tilde{v}_{m,Tm,Tm-t_0m-1,1},\tilde{v}_{m,t_0m+1,0,2},\dots,\tilde{v}_{m,Tm,Tm-t_0m-1,d})^\top.
\end{align*}}
Next let $\tildebVz = \Big(\big(\tildebVz_+\big)^\top ,\, -\big(\tildebVz_+\big)^\top\Big)^\top$ with dimension $d_{\bm{V}}$ and denote its components by
\begin{align*}
\tildebVz = (\tildeVz_1,\tildeVz_2,\dots,\tildeVz_{\dV})^\top~.
\end{align*}
By construction of $\tildebVz$ we have 
\begin{align*}
\tildeTmdZ = \max_{i=1}^{\dV} \widetilde{V}^{(z)}_i~.
\end{align*}
The covariance structure of $\bm{V}^{(z)}$ was already calculated in \eqref{eq:covVZ} and is given by
\begin{align}\label{eq:relaxLemma1}
\begin{split}
\Cov\Big(V_{i_1}^{(z)},\,V_{i_2}^{(z)} \Big)
= &\dfrac{a^{(2)}_{m,k_1,j_1}a^{(2)}_{m,k_2,j_2}}{m}(m+j_1)\rho_{h_1,h_2}
+\dfrac{a^{(1)}_{m,k_1}a^{(2)}_{m,k_2,j_2}}{m}\big((j_2 \wedge k_1)-j_1\big)\rho_{h_1,h_2}\\
&\hspace{3cm}+\dfrac{a^{(1)}_{m,k_1}a^{(1)}_{m,k_2}}{m}\rho_{h_1,h_2} (k_{min}-j_2)I\{j_2 < k_{min}\}~,
\end{split}
\end{align}
where $k_1, j_1, h_1$ and $k_2,j_2,h_2$ are the corresponding indices to $i_1$ and $i_2$, respectively, and we use the notation $k_{min} = \min\{k_1, k_2\}$.
A similar calculation  for the vector $\widetilde{\bm{V}}^{(z)}$ gives
\begin{align}\label{eq:relaxLemma2}
\begin{split}
\Cov\Big(\widetilde{V}_{i_1}^{(z)},\,\widetilde{V}_{i_2}^{(z)} \Big)
= &\dfrac{a^{(2)}_{m,k_1,j_1}a^{(2)}_{m,k_2,j_2}}{m}(m+j_1)\tilderho_{h_1,h_2}
+\dfrac{a^{(1)}_{m,k_1}a^{(2)}_{m,k_2,j_2}}{m}\big((j_2 \wedge k_1)-j_1\big)\tilderho_{h_1,h_2}\\
&\hspace{3cm}+\dfrac{a^{(1)}_{m,k_1}a^{(1)}_{m,k_2}}{m}\tilderho_{h_1,h_2} (k_{min}-j_2)I\{j_2 < k_{min}\}~.
\end{split}
\end{align}
Note that by definition of the truncated correlations in \eqref{eq:longruncorrW} the quantities in \eqref{eq:relaxLemma1} and \eqref{eq:relaxLemma2} coincide, whenever $|h_1-h_2| \leq L_d$.
Therefore we obtain for the maximum difference of the covariances,
\begin{align*}
\Delta_m 
:= \max_{i_1,i_2=1}^{\dV} 
\bigg| \Cov\Big(V_{i_1}^{(z)},\,V_{i_2}^{(z)} \Big) - \Cov\Big(\widetilde{V}_{i_1}^{(z)},\,\widetilde{V}_{i_2}^{(z)} \Big) \bigg|
\leq C_T \sup_{\substack{h_1,h_2=1\dots,d\\ |h_1-h_2| > L_d}} |\rho_{h_1,h_2}|~,
\end{align*}
where $C_T$ is a constant depending on  $T$ only, as we used that $j_1,j_2,k_1,k_2 \leq mT$ and the upper bound in \eqref{bound:atmkj}.
Assumption \ref{assump:spatialDependence} \eqref{assump:SD1} and $L_d = d^\Delta$ now yields 
\begin{align*}
\Delta_m \lesssim \log^{-2}(L_d) r_{L_d} = o\big(\log^{-2}(d)\big)~.
\end{align*}
Due to \eqref{ineq:varzh1}, we can apply Lemma \ref{lem:ineqcherno2}, which gives
\begin{align*}
\sup_{x \in \R} 
\bigg| \Pb\Big( \max_{i=1}^{\dV} V^{(z)} \leq x \Big) 
- \Pb\Big(  \max_{i=1}^{\dV} \widetilde{V}^{(z)}_i\leq x \Big) \bigg|
&\lesssim \Delta_m^{1/3}\cdot\max\Big\{1,\, \log\big(\dV/\Delta_m \big) \Big\}^{2/3}~\\
&\leq \max\Big\{\Delta_m^{1/2},\, \Delta_m^{1/2}\big|\log \dV \big|+ \Delta_m^{1/2}\big|\log\Delta_m\big|\Big\}^{2/3}.
\end{align*}
In view of \eqref{ineq:dimv} and Assumption \ref{assump:model} \eqref{assump:D1}, the proof of Lemma \ref{lem:relaxCorr} is completed.
\end{proof}

\begin{proof}[Proof of Lemma \ref{lem:discretize}]
We use similar arguments as given in the proof of Lemma B.7 of \cite{Jirak2015}.
Let $\{W_h'\}_{h \in \N}$ denote an independent copy of the sequence of Brownian motions $\{W_h\}_{h \in \N}$
 defined in  \eqref{eq:longruncorrW}.
Recalling the notation \eqref{def:smallzTildez}  we obtain the representation
\begin{align}\label{eq:TmdZt0}
\tildeTmdZ(t_0) 
= \maxhd \max_{k=t_0m+1}^{Tm}\max_{j=0}^{k-mt_0-1} \dfrac{1}{\sqrt{m}(1+k/m)} 
\bigg|\sum_{t=m+j+1}^{m+k} \tildeZ_{t,h} - \dfrac{k-j}{m+j} \sum_{t=1}^{m+j} \tildeZ_{t,h} \bigg|~.
\end{align}
To investigate the quantities in the maximum we note that 
\begin{align}\label{eq:datatoBM}
\begin{split}
&\hspace{0.51cm}\dfrac{1}{\sqrt{m}(1+k/m)} \bigg|  \sum_{t=m+j+1}^{m+k} \tildeZ_{t,h} - \dfrac{k-j}{m+j} \sum_{t=1}^{m+j} \tildeZ_{t,h} \bigg|\\
&=
\dfrac{1}{\sqrt{m}(1+k/m)} \bigg|  \sum_{t=1}^{m+k} \tildeZ_{t,h} - \dfrac{m+k}{m+j} \sum_{t=1}^{m+j} \tildeZ_{t,h} \bigg|\\
&\eqd \dfrac{1}{1+k/m} \bigg| W_h(k/m+1) -  \dfrac{m+k}{m+j} W_h(j/m+1) \bigg|\\
&= \dfrac{1}{1+k/m} \bigg| W_h(k/m+1) - W_h(1) - \dfrac{m+k}{m+j} \Big(W_h(j/m+1) - W_h(1) \Big) - \dfrac{k-j}{m+j}W_h(1) \bigg|\\
&\eqd \dfrac{1}{1+k/m} \bigg| W_h(k/m) - \dfrac{m+k}{m+j} W_h(j/m) - \dfrac{k-j}{m+j}W_h'(1) \bigg|\\
&= \dfrac{1}{(1+k/m)(1+j/m)} \bigg| (1+j/m)\Big\{ W_h(k/m) - k/m W_h'(1) \Big\}\\
&\hspace{6cm} -(1+k/m)\Big\{ W_h(j/m) - j/mW_h'(1) \Big\} \bigg|\,,
\end{split}
\end{align}
where in all steps  the correlation structure of $\{W_h\}_{h \in \N}$ is preserved.
A calculation of the covariance kernel implies the  identity (in distribution)
\begin{align*}
\Big\{ W_h(t) - tW_h'(1) \Big\}_{t \geq 0,\, h \in \N}
\eqd \Big\{ (1+t)W_h\Big(\dfrac{t}{t+1}\Big)\Big\}_{t \geq 0,\, h \in \N}~.
\end{align*}
Applying this to \eqref{eq:datatoBM} yields 
\begin{align*}
&\dfrac{1}{(1+k/m)(1+j/m)} \bigg| (1+j/m)\Big\{ W_h(k/m) - k/m W_h'(1) \Big\} -(1+k/m)\Big\{ W_h(j/m) - j/mW_h'(1) \Big\} \bigg|\\
&\quad \eqd \bigg| W_h\Big( \dfrac{k}{m+k} \Big) - W_h\Big(\dfrac{j}{m+j}\Big) \bigg|.
\end{align*}
This now gives 
\begin{align*}
\max_{k=t_0m+1}^{Tm}&\max_{j=0}^{k-mt_0-1} \dfrac{1}{\sqrt{m}(1+k/m)} 
\bigg|\sum_{t=m+j+1}^{m+k} \tildeZ_{t,h}- \dfrac{k-j}{m+j} \sum_{t=1}^{m+j} \tildeZ_{t,h} \bigg|\\
&\eqd \max_{k=t_0m+1}^{Tm}\max_{j=0}^{k-mt_0-1} \bigg| W_h\Big( \dfrac{k}{m+k} \Big) - W_h\Big(\dfrac{j}{m+j}\Big) \bigg|\\
&=  \max_{\substack{j,k \in \{1,\dots,Tm\} \\ k-j> mt_0}} \bigg| W_h\Big( \dfrac{k}{m+k} \Big) - W_h\Big(\dfrac{j}{m+j}\Big) \bigg|=:M_{h,m}(t_0)
~,
\end{align*}
 which is the discrete counterpart of the random variable $M_h(t_0)$ defined in \eqref{def:Mhc}.
Observing  the identity
\begin{align*}
M_h(t_0)
= &\max_{t \in [q(t_0),~\qT]} \max_{s \in [0,~q(q^{-1}(t)-t_0)]} \big|W_h(t) - W_h(s)\big|\\
= &\max_{t \in [q(t_0),~\qT]} \max_{s \in [0,~q^{-1}(t)-t_0]} \big|W_h(t) - W_h(q(s))\big|\\
= &\max_{t \in [t_0,T]} \max_{s \in [0,t-t_0]} \bigg|W_h\Big(\dfrac{t}{t+1}\Big) - W_h\Big(\dfrac{s}{s+1}\Big)\bigg|\,,
\end{align*}
the inequality $M_{h,m}(t_0) \leq M_h(t_0)$ already yields
\begin{align*}
&\Pb\bigg( \tildeTmdZ(t_0) \leq u_d(x) \bigg)
= \Pb\bigg( \maxhd M_{h,m}(t_0) \leq u_d(x) \bigg)\\
&\quad \geq \Pb\bigg( \maxhd M_h(t_0) \leq u_d(x) \bigg)
= \Pb\bigg( \mathcal{W}_d(t_0) \leq u_d(x) \bigg)~
\end{align*}
 for all $x \in \R$.
So it remains to find a suitable upper bound for
\begin{align}\label{toshowdisc}
\Pb\bigg( \tildeTmdZ(t_0) \leq u_d(x) \bigg) - \Pb\bigg( \Wd(t_0) \leq u_d(x) \bigg) \geq 0~.
\end{align}
Observing the inequality (which holds for all $y \in \R$)
\begin{align*}
\Pb\Big( \Wd(t_0) \leq u_d(x) \Big)
&\geq \Pb\Big( \tildeTmdZ(t_0) \leq u_d(x) - y , \big|\tildeTmdZ(t_0) - \Wd(t_0)\big| < y \Big)\\
&\geq \Pb\Big( \tildeTmdZ(t_0) \leq u_d(x) - y \Big) - \Pb\Big(\big|\tildeTmdZ(t_0) - \Wd(t_0)\big| > y \Big)
\end{align*}
the left-hand side in \eqref{toshowdisc} is bounded by
$$
\Pb\bigg( \tildeTmdZ(t_0) \leq u_d(x) \bigg) - \Pb\Big( \tildeTmdZ(t_0) \leq u_d(x) - y \Big) + \Pb\Big(\big|\tildeTmdZ(t_0) - \Wd(t_0)\big| > y \Big)~.
$$
We now choose $y_d=m^{-1/3}$.
Then the claim is a consequence of the following two assertions:
\begin{enumerate}[label=(\roman*)]
\item $\Pb\Big(\big|\tildeTmdZ(t_0) - \Wd(t_0)\big| > y_d \Big) = o(1)~$,
\item $\Pb\bigg( \tildeTmdZ(t_0) \leq u_d(x) \bigg) - \Pb\Big( \tildeTmdZ(t_0) \leq u_d(x) - y_d \Big) =  o(1)~$,
\end{enumerate}
which will be proven below to complete the proof of Lemma \ref{lem:discretize}.
To show (i), note that due to the time reversal and scaling properties of Brownian motions, it holds for all $k \leq Tm$, $1\leq h \leq d$
\begin{align*}
&\max_{t \in [(k-1)/m,~k/m]}\Big|W_h\Big(\dfrac{t}{t+1}\Big) - W_h\Big( \dfrac{k/m}{1+k/m} \Big)\Big| \\
&\eqd\max_{\lambda \in  \big[0,~ \frac{k}{m+k}-\frac{k-1}{m+k-1}\big]}\Big|W_h\Big(\dfrac{k/m}{1+k/m} - \lambda \Big) - W_h\Big( \dfrac{k/m}{1+k/m} \Big)\Big|\\
&\eqd\max_{\lambda \in [0,~ \frac{k}{m+k}-\frac{k-1}{m+k-1}]}\Big|W_h(\lambda)\Big|
\leq \max_{0\leq \lambda\leq 1/m}\Big|W_h(\lambda)\Big|\eqd \max_{0 \leq \lambda \leq 1} \big| W_h(\lambda) \big|/\sqrt{m}~,
\end{align*}
which yields
\begin{align*}
&\Pb\Big(\big|\tildeTmdZ(t_0) - \Wd(t_0)\big| > y_d \Big) =
\Pb\bigg(
\maxhd M_h(t_0) 
-\maxhd M_{h,m}(t_0)> y_d \bigg)\\
= &\Pb\bigg(
\maxhd \max_{t \in [t_0,\,T]} \max_{s \in [0,\,t-t_0]} \bigg|W_h\Big(\dfrac{t}{t+1}\Big) - W_h\Big(\dfrac{s}{s+1}\Big)\bigg|
\\
&\hspace{4cm}-\maxhd \max_{k=t_0m+1}^{Tm}\max_{j=0}^{k-mt_0-1} \bigg| W_h\Big( \dfrac{k}{m+k} \Big) - W_h\Big(\dfrac{j}{m+j}\Big)  \bigg|> y_d \bigg)\\
\leq &\Pb\bigg(
\maxhd \max_{k=t_0m+1}^{Tm}\max_{j=0}^{k-mt_0-1} \max_{t \in [(k-1)/m,\,k/m]}
\max_{s \in [j/m,\,(j+1)/m]} \bigg|W_h\Big(\dfrac{t}{t+1}\Big) - W_h\Big(\dfrac{s}{s+1}\Big)\bigg|
\\
&\hspace{4cm}-\maxhd \max_{k=t_0m+1}^{Tm}\max_{j=0}^{k-mt_0-1} \bigg| W_h\Big( \dfrac{k/m}{1+k/m} \Big) - W_h\Big(\dfrac{j/m}{1+j/m}\Big) \bigg|> y_d \bigg)\\
\leq  &\Pb\bigg(
\maxhd \max_{k=t_0m+1}^{Tm}\max_{j=0}^{k-mt_0-1} \max_{t \in [(k-1)/m,~k/m]} \max_{s \in [j/m,~(j+1)/m]} \bigg|W_h\Big(\dfrac{t}{t+1}\Big) - W_h\Big(\dfrac{s}{s+1}\Big)\\
&\hspace{7.5cm}- W_h\Big( \dfrac{k/m}{1+k/m} \Big) - W_h\Big(\dfrac{j/m}{1+j/m}\Big) \bigg|> y_d \bigg)\\
\leq& \Pb\bigg( 2 \maxhd \max_{k=1}^{Tm} \max_{t \in [(k-1)/m,~k/m]}\Big|W_h\Big(\dfrac{t}{t+1}\Big) - W_h\Big( \dfrac{k/m}{1+k/m} \Big)\Big|  > y_d \bigg) \\
\leq &\sum_{h=1}^d \sum_{k=1}^{Tm} \Pb\bigg( \sup_{0 \leq \lambda \leq 1} | W_h(\lambda) | > y_d\sqrt{m}/2 \bigg)
\leq dTm\dfrac{4}{\sqrt{2\pi}y_d}e^{-y_d^2m/8}~,
\end{align*}
where we have used the elementary bound [see for instance \cite{Karatzas1991}]
\begin{align*}
\Pb\bigg( \sup_{0 \leq \lambda \leq 1} |W_h(\lambda)| \geq z \bigg)
\leq \dfrac{4}{\sqrt{2\pi}z}e^{-z^2/2}~.
\end{align*}
This yields (i) since by Assumption \ref{assump:model} \eqref{assump:D1} the choice of $y_d$ gives $my_d^2= m^{1/3}$. \\[10pt]
 
\noindent To obtain the estimate (ii), recall the definition of the Gaussian vector $\tildebVz =\big(\tildeVz_1,\dots, \tildeVz_{\dV} \big)^\top$ in the proof of Lemma \ref{lem:relaxCorr}, which yields the identity
\begin{align*}
\max_{i=1}^{\dV} \tildeVz_i 
= \tildeTmdZ(t_0)~.
\end{align*}
For each component $\tildeVz_i$ of $\tildebVz$ there are indices $k,j,h$ such that 
$$
\tildeVz_i = \tilde{v}_{m,k,j,h}~,
$$
where $\tilde{v}_{m,k,j,h}$ is s defined in \eqref{def:tildevmkjh}.
Thus, we obtain the following bounds for the variance of the components of $\tildebVz$:
\begin{align*}
\Var \Big(\tilde{v}_{m,k,j,h} \Big)
&= \dfrac{\Big((k-j)w(k/m)\Big)^2}{m} \Var \Big(\tilde{z}_{m+j+1}^{m+k}(h) - \tilde{z}_1^{m+j}(h) \Big) \\
&= \dfrac{\Big((k-j)w(k/m)\Big)^2}{m} \bigg[ \Var \Big(\tilde{z}_{m+j+1}^{m+k}(h)\Big) + \Var \Big(\tilde{z}_1^{m+j}(h) \Big)\bigg]\\
&\geq \dfrac{m t_0^2}{(1+T)^2} \Var \Big(\tilde{z}_1^{m+j}(h) \Big) 
= \dfrac{mt_0^2}{(m+j)(1+T)^2} \geq \dfrac{t_0^2}{(1+T)^3}
\end{align*}
and
\begin{align*}
\Var \Big(\tilde{v}_{m,k,j,h} \Big)
&\leq \dfrac{mT^2}{(1+T)^2}\Big(\dfrac{1}{k-j} + \dfrac{1}{m}\Big)
\leq \dfrac{T^2}{(1+T)^2}\dfrac{1+t_0}{t_0}~.
\end{align*}
Using these bounds, we can apply Lemma \ref{lem:ineqcherno1} which yields
\begin{align*}
\Pb\bigg( \tildeTmdZ(t_0) \leq u_d(x) \bigg) - \Pb\Big( &\tildeTmdZ(t_0) \leq u_d(x) - y_d \Big)
=
\Pb\bigg( - y_d \leq \tildeTmdZ(t_0) - u_d(x) \leq 0 \bigg)\\
&\leq \sup_{z \in \R} \Pb\bigg( \Big|\tildeTmdZ(t_0) - z\Big| \leq y_d \bigg)\\
&\leq C_{T,t_0} \cdot y_d\Big( \sqrt{2\log(d)} + \sqrt{\max\{1, \log(\sigma_\ell/y_d)\}} \Big) = o(1)~,
\end{align*}
such that the assertion of Lemma \ref{lem:discretize} follows by the choice of $y_d$.
\end{proof}


\begin{proof}[Proof of Lemma \ref{lem:removeTrunc}]
First, recall the definition of $\mathcal{W}_d$ and $\mathcal{W}_d(t_0)$ in \eqref{def:wd} and note that
\begin{align*}
\mathcal{W}_d
&= \maxhd \max_{t \in [0,\qT]} \max_{s \in [0,t]} \big|W_h(s) - W_h(t)\big|\\
&=  \max \bigg\{\mathcal{W}_d(t_0)~, \maxhd\max_{t \in [0,q(t_0)]} \max_{s \in [0,t]} \big|W_h(s) - W_h(t)\big|~,\\
&\hspace{5cm}\maxhd\max_{t \in [q(t_0),\qT]} \max_{s \in [q(q^{-1}(t)-t_0),t]} \big|W_h(s) - W_h(t)\big| \bigg\}\\
&\leq  \max\bigg\{\mathcal{W}_d(t_0)~, \maxhd\max_{\substack{|t-s|\leq t_0\\s,t\in[0,\qT]}} |W(t) - W(s)| \bigg\}
\end{align*}
as $q(t_0)\leq t_0$ and $t-q(q^{-1}(t)-t_0)\leq t_0$ .
Hence, we obtain
\begin{align*}
\Pb \big( \mathcal{W}_d(t_0) \leq u_d(x) \big) - \Pb \big( \mathcal{W}_d  \leq u_d(x)  \big)
\leq\;
&\Pb \bigg( \maxhd \max_{\substack{|t-s|\leq t_0\\s,t\in[0,\qT]}} |W_h(t) - W_h(s)| >u_d(x) \bigg)\\
\leq \;&
d \;\Pb \bigg( \max_{\substack{|t-s|\leq t_0\\s,t\in[0,\qT]}} |W_1(t) - W_1(s)| >u_d(x) \bigg)~.
\end{align*}
To control this probability we define an overlapping decomposition of the interval $[0,\qT]$ by
$$
I_j:= [jt_0,(j+2)t_0] ~, \quad j=0,1,2,\dots,\ceil{\qT/t_0}-2~.
$$
Observing that the length of $I_j$ is $2t_0$ we obtain
\begin{align*}
&\Pb\bigg( \max_{\substack{|t-s|\leq t_0\\s,t\in[0,\qT]}} |W_1(t) - W_1(s)| > u_d(x) \bigg)\nonumber\\
\leq &\sum_{j=1}^{\ceil{\qT/t_0}-2} \Pb\bigg( \max_{\substack{|t-s|\leq t_0\\s,t\in I_j}} |W_1(t) - W_1(s)| > u_d(x) \bigg)\nonumber\\
\leq \; &\dfrac{\qT}{t_0} \Pb\bigg( \max_{\substack{|t-s|\leq t_0\\s,t\in [0,2t_0]}} |W_1(t) - W_1(s)| > u_d(x) \bigg)\nonumber\\
\leq \;&\dfrac{\qT}{t_0} \Pb\bigg( \max_{s,t\in [0,2t_0]} |W_1(t) - W_1(s)| > u_d(x) \bigg)\nonumber\\
= &\dfrac{\qT}{t_0} \Pb\bigg( \max_{s,t\in [0,\qT]} |W_1(t) - W_1(s)| > u_d(x) \cdot \sqrt{\qT/(2t_0)}\bigg)\\
\leq &\dfrac{\qT}{t_0} \Pb\bigg( \max_{s,t\in [0,\qT]} |W_1(t) - W_1(s)| > c\sqrt{\log(d)\qT/(2t_0)} \bigg)
\end{align*}
as for fixed $x$ \eqref{eq:adbd} yields, that there exists a positive constant $c<\sqrt{2\qT}$, such that $u_d(x) \geq c \cdot \sqrt{\log(d)}$ for $d$ sufficiently large.
Using the representation of the distribution function $\FM$ in \eqref{eq:repF} we obtain
\begin{align*}
\Pb\bigg( \maxhd \max_{\substack{|t-s|\leq t_0\\s,t\in[0,\qT]}} |W_h(t) - W_h(s)| > u_d(x) \bigg)
\leq d \dfrac{\qT}{t_0}\left[1-\FM\Big(c\sqrt{\log(d)\qT/(2t_0)}\Big)\right]
\end{align*}
and L'H\^{o}spital's rule gives 
\begin{align*}
\limd d\left[1-\FM\Big(c\sqrt{\log(d)\qT/(2t_0)}\Big)\right]
&= c\sqrt{\qT/(2t_0)} \lim_{d\rightarrow\infty} d \dfrac{\FM'\Big(c\sqrt{\log(d)\qT/(2t_0)}\Big)}{2\sqrt{\log(d)}}\\
&\leq c\sqrt{\qT/(2t_0)} \lim_{d\rightarrow\infty} d \FM'\Big(c\sqrt{\log(d)\qT/(2t_0)}\Big)\,.
\end{align*}
Now substituting $d = \exp\Big(\dfrac{2y^2t_0}{c^2\qT}\Big)$ yields that the last display can be written as
\begin{align*}
c\sqrt{\qT/(2t_0)} \lim_{y \to \infty} \exp\Big(\dfrac{2y^2t_0}{c^2\qT}\Big) \FM'(y)\,,
\end{align*}
which by assertion \eqref{eq:F1} tends to zero for sufficiently small $t_0>0$ and thus completes the proof of Lemma \ref{lem:removeTrunc}.
\end{proof}


\subsection{Proof of Theorem \ref{thm:alternative}}
\label{seca4}
Denote the size of the change by $\Delta\mu_m=|\mu_{m+k^*-1,h^*} - \mu_{m+k^*,h^*}|$ and the centered observations in component $h^*$ by
$$ 
X_{t,h^*}^{(c)} := X_{t,h^*} - \E[X_{t,h^*}]~.
$$
Observe the following lower bound
\begin{align}\label{ineq:alternative1}
\quad \hatTmd 
&= \maxhd \maxkTm w(k/m)\hatE_{m,h}(k) \overset{h=h^*}{\geq} \maxkTm w(k/m)\hatE_{m,h^*}(k)
\overset{k=mT}{\geq} w(T)\hatE_{m,h^*}(mT)\notag\\
&= \dfrac{1}{1+T}\max_{j=0}^{Tm-1} \dfrac{mT-j}{\sqrt{m}\hatsigma_h} \Big|\hatmu_{m+j+1}^{m+mT}(h^*) - \hatmu_{1}^{m+j}(h^*) \Big| \notag\\
&\overset{j=k^*-1}{\geq} \dfrac{1}{1+T}\dfrac{mT-k^*+1}{\sqrt{m}\hatsigma_h} \Big|\hatmu_{m+k^*}^{m+mT}(h^*) - \hatmu_{1}^{m+k^*-1}(h^*) \Big|\notag\\
&\geq \dfrac{1}{1+T}\bigg\{\dfrac{mT-k^*+1}{\sqrt{m}\hatsigma_h}\Delta\mu_m 
- \bigg\vert\dfrac{1}{\sqrt{m}\hatsigma_h} \sum_{t=m+k^*}^{m+mT} X_{t,h^*}^{(c)}
- \dfrac{mT-k^*+1}{\sqrt{m}\hatsigma_h(m+k^*-1)} \sum_{t=1}^{m+k^*-1} X_{t,h^*}^{(c)} \bigg\vert \bigg\}\,.
\end{align}
The consistency of the long-run variance estimator $\hatsigma_h$, Assumption \ref{assump:temporalDependence},
 the FCLT in Theorem~3 of \cite{Wu2005} and   the Continuous Mapping Theorem show that 
\begin{align}\label{conv:thmaltub}
\begin{split}
\bigg\vert\dfrac{1}{\sqrt{m}\hatsigma_h} \sum_{t=m+k^*}^{m+mT} &X_{t,h^*}^{(c)}
- \dfrac{mT-k^*+1}{\sqrt{m}\hatsigma_h(m+k^*-1)} \sum_{t=1}^{m+k^*-1} X_{t,h^*}^{(c)} \bigg\vert\\
&\leq \max_{s \in [0,T]} \bigg|\dfrac{1}{\sqrt{m}\hatsigma_h} \sum_{t=m+\floor{ms}+1}^{m+mT} X_{t,h^*}^{(c)}
- \dfrac{mT-\floor{ms}}{\sqrt{m}\hatsigma_h(m+\floor{ms})} \sum_{t=1}^{m+\floor{ms}} X_{t,h^*}^{(c)} \bigg|\\
&\hspace{0.5cm}\convd \max_{s \in [0,T]}\big|W(1+T) - W(1+s) - \dfrac{T-s}{1+s}W(1+s) \big|~,
\end{split}
\end{align}
where $W$ is a standard one-dimensional Brownian motion.
Next, note that \eqref{ineq:alternative1} gives that
{\footnotesize \begin{align*}
&\Pb\bigg( a_d \big(\hatTmd - b_d\big) > g_{1-\alpha} \bigg)
= \Pb\bigg( \hatTmd > \dfrac{g_{1-\alpha}}{a_d} + b_d \bigg)\\
&\geq \Pb\bigg( -\bigg| \dfrac{1}{\sqrt{m}\hatsigma_h}\sum_{t=m+k^*}^{m+mT} X_{t,h}^{(c)}
- \dfrac{mT-k^*+1}{\sqrt{m}\hatsigma_h(m+k^*-1)}\sum_{t=1}^{m+k^*-1} X_{t,h}^{(c)}\bigg|> \Big(\dfrac{g_{1-\alpha}}{a_d} + b_d\Big) (T+1) - \dfrac{mT-k^*+1}{\sqrt{m}\hatsigma_h}\Delta\mu_m \bigg).
\end{align*}}
By Assumption \ref{assump:model} \eqref{assump:D1} and \eqref{eq:adbd} we obtain $ b_d \lesssim \sqrt{\log m}$. 
Applying now \eqref{assump:alt1} we get
\begin{align}\label{conv:thmaltlb}
\Big(\dfrac{g_{1-\alpha}}{a_d} + b_d\Big) (T+1) - \dfrac{mT-k^*+1}{\sqrt{m}\hatsigma_h}\Delta\mu_m \convp -\infty~.
\end{align}
Combining  \eqref{conv:thmaltub}, \eqref{conv:thmaltlb} with an application of Slutsky's Theorem shows that the probability  $\Pb\big( a_d \big(\hatTmd - b_d\big) > g_{1-\alpha} \big)$ tends to 1, which completes the proof.

\renewcommand{\qedsymbol}{}
\end{proof}

\subsection{Proof of Corollary \ref{cor:Gumbel}}
The result  is obtained analogously to the corresponding parts of Theorem 2.5 in \cite{Jirak2015} or Theorem 3.11 in \cite{Dette2018}.
Therefore the proof is omitted. 

\subsection{Proof of Theorem \ref{thm:Bootstrap}}
Recall the definition of $u_d(x) = x/a_d + b_d$,   $Z_{t,h}$, $\hatZ_{t,h}$, $\TmdZ$, $\TmdZ(c)$ and $\hatTmdZ(c)$
in \eqref{def:udx}, \eqref{def:GaussFields}, \eqref{eq:CovarBootstrapZ},    \eqref{def:TmdZs} and \eqref{def:truncatedTmds}, respectively.
The proof of Theorem \ref{thm:Bootstrap} is based on the following three Lemmas.

\begin{lemma}\label{lem:corrConvergence}
For the constant $C_\gamma$ from Assumption \ref{assump:Bootstrap} it holds that
\begin{align}\label{ineq:uniformCorr}
m^{C_\gamma}\max_{h,i=1}^d \big| \hat{\rho}_{h,i} - \rho_{h,i}  \big| = o_\Pb(1)~.
\end{align}
\end{lemma}
\begin{proof}
First, note that Lemma E.4 from \cite{Jirak2015supp} implies the existence of a global constant $C_\sigma >0$, such that
$\sup_{h\in \N} \sigma_{h} \leq C_\sigma$.
Next, Assumption \ref{assump:temporalDependence} and the Cauchy-Schwarz inequality imply that
\begin{enumerate}[label=(\roman*)]
\item $\maxhid \gamma_{h,i}
\leq \sigma_h\sigma_i
\leq C_\sigma^2 $~,
\item $\minhd \hatsigma_h
\geq \minhd \sigma_h - \maxhd |\hatsigma_h - \sigma_h |
\geq c_\sigma - \maxhd |\hatsigma_h - \sigma_h |$~,
\item $
\maxhid |\hatsigma_h\hatsigma_i - \sigma_h\sigma_i|
\leq \maxhid \hatsigma_i|\hatsigma_h - \sigma_h| + C_\sigma\maxhd |\hatsigma_h - \sigma_h|\\
\phantom{a}\hspace{3cm}\leq C_\sigma \maxhd |\hatsigma_h - \sigma_h|^2 + 2C_\sigma\maxhd |\hatsigma_h - \sigma_h|~.$
\end{enumerate}
Combining (i), (ii) and using again Assumption \ref{assump:temporalDependence}
gives
\begin{align*}
\maxhid \bigg|\dfrac{\hat{\gamma}_{h,i}}{\hat{\sigma}_h\hat{\sigma}_i} \bigg|
\leq \dfrac{1}{\big(c_\sigma - \maxhd |\hatsigma_h - \sigma_h |\big)^2}
\cdot \Big( C_\sigma^2 + \maxhid \big| \hatgamma_{h,i} -\gamma_{h,i}\big|\Big) = O_\Pb(1)~.
\end{align*}
Thus  we obtain the upper bound
\begin{align*}
\maxhid \big| \hat{\rho}_{h,i} - \rho_{h,i}  \big|
&\leq \maxhid \bigg| \dfrac{\hat{\gamma}_{h,i}}{\hat{\sigma}_h\hat{\sigma}_i} - \dfrac{\hat{\gamma}_{h,i}}{\sigma_h\sigma_i} \bigg|
+ \maxhid \bigg| \dfrac{\hat{\gamma}_{h,i}-\gamma_{h,i}}{\sigma_h\sigma_i} \bigg|\\
& \leq \dfrac{1}{c^2_\sigma}\maxhid \bigg| \dfrac{\hat{\gamma}_{h,i}}{\hat{\sigma}_h\hat{\sigma}_i} \bigg| \big|\hat{\sigma}_h\hat{\sigma}_i-\sigma_h\sigma_i \big|
+  \dfrac{1}{c^2_\sigma}\maxhid \big|\hat{\gamma}_{h,i}-\gamma_{h,i} \big|\\
&\lesssim O_\Pb(1)\maxhd |\hatsigma_h - \sigma_h|^2 + O_\Pb(1)\maxhd |\hatsigma_h - \sigma_h|
+ \maxhid \big|\hat{\gamma}_{h,i}-\gamma_{h,i} \big|
\,.
\end{align*}
The assertion of Lemma \ref{lem:corrConvergence} now follows from Assumption \ref{assump:Bootstrap}.
\end{proof}

\medskip
\begin{lemma}\label{lem:BSLemma2}
There exists a sufficiently small constant $t_0>0$, such that for $x \in \R$ it holds
\begin{align*}
\bigg| \PbX \Big( \hatTmdZ(t_0) \leq u_d(x) \Big) 
-  \PbX\Big( \hatTmdZ \leq u_d(x) \Big) \bigg|
= o_{\Pb}(1)\,, \quad \text{ as } m,d\to \infty\,.
\end{align*}
\end{lemma}
\begin{proof}
We provide a (stochastic) version of the proof of Lemma \ref{lem:Truncation}.
First note that
\begin{align*}
\hatTmdZ
&= \maxhd \max_{k=1}^{Tm} \max_{j=0}^{k-1}  \frac{(k-j)w(k/m)}{\sqrt{m}}\Big|\hatmu_{m+j+1}^{m+k}(h) - \hatmu_1^{m+j}(h) \Big|\\
&=  \max \bigg\{\hatTmdZ(t_0)~,  \maxhd \max_{k=t_0m+1}^{Tm} \max_{j=k-t_0m}^{k-1} \frac{(k-j)w(k/m)}{\sqrt{m}}\Big|\hat{z}_{m+j+1}^{m+k}(h) - \hat{z}_1^{m+j}(h) \Big|~,\\
&\hspace{4.7cm}\max_{h=1}^d\max_{k=1}^{t_0m} \max_{j=0}^{k-1} \frac{(k-j)w(k/m)}{\sqrt{m}}\Big|\hat{z}_{m+j+1}^{m+k}(h) - \hat{z}_1^{m+j}(h) \Big|\bigg\}~.
\end{align*}
Hence, we obtain
\begin{align}\label{ineq:trunc1stochastic}
\Big|\PbX\Big(\hatTmdZ(t_0) \leq u_d(x) \Big) - \PbX\Big(\hatTmdZ \leq u_d(x) \Big)\Big|
\leq P_1(x) + P_2(x)~,
\end{align}
where the random variables $P_1(x)$  and $P_2(x)$ are defined by
\begin{align*}
P_1(x)
&= \PbX\Big( \maxhd \max_{k=t_0m+1}^{Tm} \max_{j=k-t_0m}^{k-1} \frac{(k-j)}{\sqrt{m}}\Big|\hat{z}_{m+j+1}^{m+k}(h) - \hat{z}_1^{m+j}(h) \Big|\geq u_d(x)\Big)~,\\
P_2(x)
&= \PbX\Big(\max_{h=1}^d\max_{k=1}^{t_0m} \max_{j=0}^{k-1} \frac{(k-j)}{\sqrt{m}}\Big|\hat{z}_{m+j+1}^{m+k}(h) - \hat{z}_1^{m+j}(h) \Big|\geq u_d(x)\Big)
\end{align*}
and we additionally used that $w(k/m)\leq 1$.
To complete the proof, it suffices by Markov's inequality to establish that
$$\E[P_1(x)]=o(1)
\;\;\;\text{and}\;\;\;
\E[P_2(x)] = o(1)~.$$
To prove these assertions, observe the bounds
\begin{align}\label{ineq:truncBootstrap1}
\begin{split}
\E[P_1(x)]
&= \Pb\Big( \maxhd \max_{k=t_0m+1}^{Tm} \max_{j=k-t_0m}^{k-1} \frac{(k-j)}{\sqrt{m}}\Big|\hat{z}_{m+j+1}^{m+k}(h) - \hat{z}_1^{m+j}(h) \Big|\geq u_d(x)\Big)\\
&\leq \sum_{h=1}^d \Pb\Big(\max_{k=t_0m+1}^{Tm} \max_{j=k-t_0m}^{k-1} \frac{(k-j)}{\sqrt{m}}\Big|\hat{z}_{m+j+1}^{m+k}(h) - \hat{z}_1^{m+j}(h) \Big|\geq u_d(x)\Big)
\end{split}
\end{align}
and
\begin{align}\label{ineq:truncBootstrap2}
\begin{split}
\E[P_2(x)]
&= \Pb\Big(\max_{h=1}^d\max_{k=1}^{t_0m} \max_{j=0}^{k-1} \frac{(k-j)}{\sqrt{m}}\Big|\hat{z}_{m+j+1}^{m+k}(h) - \hat{z}_1^{m+j}(h) \Big|\geq u_d(x)\Big)\\
&\leq \sum_{h=1}^d\Pb\Big(\max_{k=1}^{t_0m} \max_{j=0}^{k-1} \frac{(k-j)}{\sqrt{m}}\Big|\hat{z}_{m+j+1}^{m+k}(h) - \hat{z}_1^{m+j}(h) \Big|\geq u_d(x)\Big)~.
\end{split}
\end{align}
The terms in \eqref{ineq:truncBootstrap1} and \eqref{ineq:truncBootstrap2} can now be controlled by 
the same arguments as given in   the proof of Lemma \ref{lem:Truncation}.
\end{proof}
\medskip

\begin{lemma}\label{lem:BSnotuniform}
For $x\in \R$ it holds that
\begin{align}\label{eq:BSnotuniform}
\bigg| \Pb_{|\mathcal{X}}\Big( \hatTmdZ\leq u_d(x) \Big) 
-  \Pb_{H_0}\Big( \Tmd \leq u_d(x) \Big) \bigg|
= o_{\Pb}(1)\,, \quad \text{ as } m,d\to \infty\,.
\end{align}
\end{lemma}
\begin{proof}
Observing Lemmas \ref{lem:Truncation}, \ref{lem:gaussianapprox}, \ref{lem:BSLemma2}, the assertion of Lemma \ref{lem:BSnotuniform} follows, if we can establish that
\begin{align}\label{eq:ToShowBSTheorem}
\bigg| \PbX \Big( \hatTmdZ(t_0) \leq u_d(x) \Big) 
-  \Pb\Big( \TmdZ(t_0) \leq u_d(x) \Big) \bigg|
= o_{\Pb}(1)~,
\end{align}
To obtain this, we will reuse the vector technique applied in the proof of Lemma \ref{lem:gaussianapprox}.
>From the proof of this Lemma recall the definition and construction of the Gaussian vector $\bVz =\Big(\Vz_1,\dots,\Vz_{\dV}\Big)^\top $ which fulfilled the identity
\begin{align*}
\max_{i=1}^{\dV} \Vz_i = \TmdZ(t_0)~.
\end{align*}
Analogously we construct a vector $\hatbVz = \Big(\hatVz_1,\dots,\hatVz_{\dV}\Big)^\top $ from $\big\{\hatZ_{t,h}\big\}$, such that
\begin{align*}
\max_{i=1}^{d_{\bm{V}}} \hatVz_i = \hatTmdZ(t_0)~.
\end{align*}
The covariance structure of $\bVz$ was already calculated in Lemma \ref{lem:gaussianapprox}.
Repeating these steps for the conditional covariance structure of $\hatbVz$ with respect to $\mathcal{X}$, we directly obtain that
\begin{align}\label{ineq:BSCovStructure}
\max_{i_1,i_2=1}^{d_{\bm{V}}}
\Big| \Cov\big(\Vz_{i,1},\,\Vz_{i,2} \big) - 
\Cov_{|\mathcal{X}}\big(\hatVz_{i,1},\,\hatVz_{i,2} \big) \Big|
\lesssim \max_{h,i=1}^d \big| \hat{\rho}_{h,i} - \rho_{h,i}  \big|~.
\end{align}
In the remainder of the proof we use the notation $\Delta_\rho
= \max_{h,i=1}^d \big| \hat{\rho}_{h,i} - \rho_{h,i}  \big|$. 
In view of \eqref{ineq:BSCovStructure}, we are able to apply the Gaussian comparison inequality from Lemma \ref{lem:ineqcherno2}, which gives
\begin{align*}
&\sup_{x \in \R}\bigg|  \PbX \Big( \hatTmdZ(t_0) \leq x \Big) 
-  \Pb\Big( \TmdZ \leq x \Big) \bigg|\\
= &\sup_{x \in \R}\bigg| \PbX \Big(\max_{i=1}^{\dV} \hatVz_i \leq x \Big) 
-  \Pb\Big( \max_{i=1}^{\dV} \Vz_i \leq x \Big) \bigg|
\leq C \Delta_\rho^{1/3} \cdot \max\Big\{1,\,\log\big(\dV/\Delta_\rho\big)\Big\}^{2/3}.
\end{align*}
Due to Lemma \ref{lem:corrConvergence} and Assumption \ref{assump:model} the upper bound in the last display is of order $o_\Pb(1)$, which proves \eqref{eq:ToShowBSTheorem}.
\end{proof}
\begin{proof}[Final step in proof of Theorem \ref{thm:Bootstrap}]
To obtain the theorem's assertions, note that from Corollary \ref{cor:Gumbel} we already know that
\begin{align}\label{conv:BootstrapGumbel}
a_d\big( \hatTmd - b_d) \convd G~,
\end{align}
and as the Gumbel distribution has a continuous c.d.f., Polya's theorem  [see \cite{serfling2009}, p. 18] directly implies convergence in Kolmogorov-metric, that is
\begin{align}\label{kolmog:bootstrap1}
\sup_{x \in \R} 
\bigg|\Pb\Big( a_d\big( \hatTmd - b_d) \leq u_d(x) \Big) - \Pb\big( G \leq x\big)\bigg|
= o(1)~.
\end{align}
On the other hand, combining \eqref{eq:BSnotuniform} with Theorem \ref{thm:mainGumbel} implies that
\begin{align}\label{kolmog:bootstrap2}
a_d\big( \hatTmdZ - b_d) \convd G~,
\end{align}
conditional on $\mathcal{X}$ in probability. So a conditional version of Polya's theorem gives
\begin{align*}
\sup_{x \in \R} \bigg|\PbX\Big( a_d\big( \hatTmdZ - b_d) \leq u_d(x) \Big) - \Pb\big( G \leq x\big)\bigg|
= o_{\Pb}(1)~.
\end{align*}
By \eqref{kolmog:bootstrap1} and \eqref{kolmog:bootstrap2} the proof of Theorem \ref{thm:Bootstrap} is complete.
\end{proof}

\subsection{Proof of Theorem \ref{thm:algorithmConsistent}}
Denote the centered observations by
$$
X_{t,h}^{(c)} = X_{t,h} -\E[X_{t,h}]~.
$$
We first prove assertions \eqref{ineq:SdSubset} and \eqref{eq:SdC} using the Gumbel quantile $g_{1-\alpha}$.\\[12pt]
\textbf{Proof of \eqref{ineq:SdSubset}:} 
It holds that
\begin{align*}
\Pb\Big( \mathcal{S}_d \subset \hatSdalpha   \Big)
&= \Pb\Big( \max_{h \in \mathcal{S}_d} \maxkTm w(k/m) \hatE_{m,h}(k) \leq g_{1-\alpha}/a_d + b_d \Big)\\
&= \Pb_{H_0}\Big( \max_{h \in \mathcal{S}_d} \maxkTm w(k/m) \hatE_{m,h}(k) \leq g_{1-\alpha}/a_d + b_d \Big)\\
&\geq \Pb_{H_0}\Big( \maxhd \maxkTm w(k/m) \hatE_{m,h}(k) \leq g_{1-\alpha}/a_d + b_d \Big) \longrightarrow 1 - \alpha~,
\end{align*}
where we applied Corollary \ref{cor:Gumbel} for the last convergence.\\[12pt]
\textbf{Proof of \eqref{eq:SdC}:} First, note that:
\begin{align}\label{eq:SedCProb}
\Pb\Big( \mathcal{S}_d^c \subset \hatSdalpha^c \Big)
= \Pb\Big( \min_{h \in \mathcal{S}_d^c} \maxkTm w(k/m) \hatE_{m,h}(k) > g_{1-\alpha}/a_d + b_d  \Big)~.
\end{align}
We have the lower bound 
\begin{align*}
&\min_{h \in \mathcal{S}_d^c} \maxkTm w(k/m) \hatE_{m,h}(k) 
\geq \min_{h \in \mathcal{S}_d^c} w(T) \hatE_{m,h}(Tm)\\
&\quad \geq \min_{h \in \mathcal{S}_d^c} \dfrac{Tm-k_h^*-1}{\sqrt{m}\hatsigma_h (T+1)} \big| \hatmu_{m+k_h^*}^{m+Tm} - \hatmu_{1}^{m+k_h^*-1} \big|
\geq  A_1 - A_2 ~,
\end{align*}
where the terms $A_1$ and $A_2$ are given by
\begin{align*}
A_1 &= \min_{h \in \mathcal{S}_d^c}  \dfrac{Tm-k_h^*+1}{\sqrt{m}\hatsigma_h(T+1)}\big|\mu_{m+k_h^*-1} - \mu_{m+k_h^*}\big| ~,\\
A_2 &= \max_{h \in \mathcal{S}_d^c} \dfrac{1}{\sqrt{m} \hatsigma_h(T+1)} \bigg| \sum_{t=m+k_h^*}^{m+Tm} X_{t,h}^{(c)} - \dfrac{mT - k_h^*+1}{m+k_h^*-1}\sum_{t=1}^{m+k_h^*-1}X_{t,h}^{(c)} \bigg|~.
\end{align*}
Therefore the probability given in \eqref{eq:SedCProb} has the lower bound
\begin{align*}
\Pb\big( a_d(A_1 - 2b_d) - a_d(A_2 - b_d) > g_{1-\alpha} \big)~.
\end{align*}
Using Corollary \ref{cor:Gumbel} we obtain that
\begin{align}\label{ineq:uniform1}
\begin{split}
a_d&\big(A_2 - b_d\big)\\
&\leq a_d \bigg( \maxhd \maxkTm w(k/m) \max_{j=0}^{k-1} 
\bigg| \sum_{t=m+j+1}^{m+k} X_{t,h}^{(c)} - \dfrac{k - j}{m+j}\sum_{t=1}^{m+j}X_{t,h}^{(c)}  \bigg| - b_d \bigg) = O_\Pb(1)~.
\end{split}
\end{align}
Further it holds by \eqref{assump:uniformConsistency} that for $m$ sufficiently large
\begin{align*}
Tm - \max_{h \in \mathcal{S}_d^c} k_h^* > c~,
\end{align*}
where $c>0$ is a sufficiently small constant.
By Assumptions \ref{assump:temporalDependence} and \ref{assump:Bootstrap} we have  
\begin{align*}
\dfrac{1}{\maxhd \hatsigma_h} 
\geq \dfrac{1}{C_\sigma + \maxhd| \hat\sigma_h -\sigma_h|}
\convp \dfrac{1}{C_\sigma}~
\end{align*}
and \eqref{eq:adbd} shows that $b_d \lesssim \sqrt{\log m}$ and $a_d \to \infty$.
Combining this with the assertions above yields
\begin{align}\label{ineq:uniform2}
a_d\big(A_1 - 2b_d\big)
\gtrsim a_d\bigg(\sqrt{m}\dfrac{1}{\maxhd \hatsigma_h} \min_{h \in \mathcal{S}_d^c} 
\big|\mu_{m+k_h^*-1} - \mu_{m+k_h^*}\big| - 2b_d \bigg)\convp \infty
\end{align}
A combination of \eqref{ineq:uniform1} and \eqref{ineq:uniform2} now proves \eqref{eq:SdC}.\\[12pt]
To complete the proof of Theorem \ref{thm:algorithmConsistent} it remains to discuss the case, where bootstrap quantiles
\begin{align*}
\hat{q}_{m,1-\alpha}:= \inf \Big\{ x \in \R \;\Big|\; \PbX\Big( a_d\big(\hatTmdZ - b_d \big) \leq x \Big) \geq \alpha\Big\}~,
\end{align*}
are used in the algorithm.
Fortunately, it follows from Theorem \ref{thm:Bootstrap} combined with Lemma~21.2 and (the arguments from) Lemma 23.2 in \cite{vanderVaart1998} that
\begin{align*}
\hat{q}_{m,1-\alpha} \convp g_{1-\alpha}~.
\end{align*}
An application of Slutsky's Lemma to the statements above then completes the proof of Theorem \ref{thm:algorithmConsistent}.

\section{Technical auxiliary results}

\noindent We require the following Nagaev-type inequality as given in the online supplement of \cite{Jirak2015} which is a version of Theorem 2 in \cite{Liu2013}.
In particular the reader should note that the bound is independent of $h$.\\

\begin{lemma}\label{lem:nagaev}
Under Assumption \ref{assump:temporalDependence} it holds for $x \geq C_p'\sqrt{n}$
\begin{align*}
\Pb\bigg( \max_{k=1}^n\Big|  \sum_{t=1}^k X_{t,h} &- \E[X_{t,h}] \Big| > x \bigg)\leq C_p \dfrac{n}{x^p} + C_p\exp\bigg( -c_p \dfrac{x^2}{n} \bigg)~,
\end{align*}
where the constants $c_p, C_p, C_p >0$ depend on $p$ and the sequence $\Big\{\sup\limits_{h \in \N} \vartheta_{t,h,p}\Big\}_{t \in \N}$ only.
\end{lemma}
\noindent As an immediate consequence of the bound
$$
\max_{k=1}^n\Big|  \sum_{t=k}^n X_{t,h} - \E[X_{t,h}] \Big|
\leq 2 \max_{k=1}^n\Big|  \sum_{t=1}^k X_{t,h} - \E[X_{t,h}] \Big|
$$
Lemma \ref{lem:nagaev} holds with adjusted constants also for the reversed partial sum maximum.\\

\noindent The following inequality is Lemma 2.1 in \cite{Chernozhukov2013}.
\begin{lemma}\label{lem:ineqcherno1}
Let $\bm{Z} = (Z_1,\dots,Z_d)^\top$ be a centered Gaussian vector with covariance matrix $\Sigma^Z$ whose diagonal entries are bounded by two constants $\sigma_\ell$ and $\sigma_{u}$, that is
$$
0<\sigma_\ell 
\leq  \Sigma^Z_{j,j} 
\leq \sigma_{u}
$$
for $j=1,\dots,d$~.
Then for $\delta >0$ it holds that 
\begin{align*}
\sup_{z \in \R} \Pb\Big( \Big| \maxhd Z_h - z \Big| \leq \delta \Big)
&\leq C_{\sigma}\delta \Big( \sqrt{2\log(d)} + \sqrt{\max\{1, \log(\sigma_\ell/\delta)\}} \Big)~,
\end{align*}
where the constant $C_{\sigma}>0$ depends on $\sigma_\ell$ and $\sigma_u$ only.
\end{lemma}
\noindent The next tool is Lemma 3.1 from \cite{Chernozhukov2013}.
\begin{lemma}\label{lem:ineqcherno2}
Let $\bm{U}=(U_1,\dots,U_d)^\top$ and $\bm{V}=(V_1,\dots,V_d)^\top$ denote two centered, $d$-dimensional Gaussian vectors with covariance matrices $\Sigma^U$ and $\Sigma^V$, respectively.
Further assume that there are two constants $c_1,C_1>0$, such that for all $j=1,\dots,d$
$$
c_1 \leq \big|\Sigma^U_{j,j} \big| \leq C_1~.
$$
Denote the maximum entry-wise distance of both covariance matrices by 
$$
\Delta := \max_{i,j=1}^d \big|\Sigma^U_{i,j} - \Sigma^V_{i,j}  \big|~.
$$
Then it holds that 
$$
\sup_{x \in \R} \Big| \Pb\Big( \max_{i=1}^{d} U_i  \leq x \Big) 
- \Pb\Big( \max_{i=1}^{d} V_i  \leq x \Big)  \Big|
\leq  C\Delta^{1/3}\cdot\max\Big\{1,\, \log\big(d/\Delta \big) \Big\}^{2/3}~,
$$
where the constant $C>0$ depends on $c_1$ and $C_1$ only.
\end{lemma}
\end{document}